\documentclass[12pt,reqno]{amsart}
\usepackage{amsfonts, amssymb, latexsym, epsfig, amsthm, enumitem}
\usepackage[usenames,dvipsnames]{xcolor}
\usepackage{graphicx,pgf,tikz}
\usetikzlibrary{decorations.markings}
\usepackage{ytableau}
\usepackage[all]{xy}
\usepackage{wrapfig}
\usepackage{mathdots}
\usepackage{ bbold }

\newlist{thmlist}{enumerate}{1}
\setlist[thmlist]{label=(\roman{thmlisti}),noitemsep}

\usepackage[colorlinks=true, pdfstartview=FitV, linkcolor=blue, citecolor=blue, urlcolor=blue]{hyperref}

\newtheorem{theorem}{Theorem}[section]
\newtheorem{proposition}[theorem]{Proposition}
\newtheorem{lemma}[theorem]{Lemma}

\newtheorem{problem}[theorem]{Problem}
\newtheorem*{claim*}{Claim}
\newtheorem{corollary}[theorem]{Corollary}
\newtheorem{Main Conjecture}[theorem]{Main Conjecture}

\theoremstyle{remark}
\newtheorem{defn}[theorem]{Definition}
\newtheorem{examplex}[theorem]{Example}
\newenvironment{example}
  {\pushQED{\qed}\examplex}
  {\popQED\endexamplex}

\newenvironment{remark}
  {\pushQED{\qed}\remarkx}
  {\popQED\endremarkx}
\newtheorem{remarkx}[theorem]{Remark}
\newtheorem{notation}[theorem]{Notation}
\theoremstyle{plain}

\newtheorem*{MainThm}{Main Theorem}

\usepackage{enumitem}

\newcommand{\mydef}[1]{{\bf #1}}
\newcommand{\notdefterm}[1]{\emph{#1}}

\newcommand{\init}{{\tt in}}
\newcommand{\mult}{{\tt mult}}
\newcommand{\bpd}{{\sf BPD}}
\newcommand{\pipes}{{\sf Pipes}}
\newcommand{\hgt}{{\rm{ht }}}

\newcommand{\ess}{{\tt Ess}}
\newcommand{\perm}{{\tt Perm}}
\newcommand{\asm}{{\tt ASM}}
\newcommand{\wt}{{\tt wt}}

\newcommand{\ZZ}{\mathbb{Z}}
\newcommand{\CC}{\mathbb{C}}
\newcommand{\NN}{\mathbb{N}}
\newcommand{\QQ}{\mathbb{Q}}

\newcommand{\hilb}{\mbox{Hilb}}

\newcommand{\rk}{{\tt rk}}

\newcommand{\mat}{{\sf Mat}}
\newcommand{\spec}{{\mbox{Spec}}}
\newcommand{\proj}{{\mbox{Proj}}}

\newcommand{\fl}{{\mathcal F}}
\newcommand{\GL}{{\sf GL}}

\usepackage[colorinlistoftodos]{todonotes}

\begin{document}
\pagestyle{plain}

\title{Bumpless pipe dreams encode Gr\"obner geometry of Schubert polynomials}
\author{Patricia Klein}
\address{Department of Mathematics, Texas A\&M University, College Station, TX 77840, USA}
\email{pjklein@tamu.edu}

\author{Anna Weigandt}
\address{School of Mathematics, University of Minnesota, Minneapolis, MN 55455, USA}
\email{weigandt@umn.edu}

\date{\today}

\begin{abstract}
In their study of infinite flag varieties, Lam, Lee, and Shimozono (2021) introduced bumpless pipe dreams in a new combinatorial formula for double Schubert polynomials.  These polynomials are the $T \times T$-equivariant cohomology classes of matrix Schubert varieties and of their flat degenerations.  We give diagonal term orders with respect to which bumpless pipe dreams index the irreducible components of diagonal Gr\"obner degenerations of matrix Schubert varieties, counted with scheme-theoretic multiplicity.  

This indexing was conjectured by Hamaker, Pechenik, and Weigandt (2022).  This result establishes that bumpless pipe dreams are dual to and as geometrically natural as classical pipe dreams, for which an analogous anti-diagonal theory was developed by Knutson and Miller (2005).
\end{abstract}

\maketitle

\tableofcontents

\section{Introduction}

The \mydef{complete flag variety} $\fl(\CC^n)=B_-\backslash\GL(\CC^n)$ is the quotient of the general linear group by the Borel subgroup $B_-$ of lower triangular matrices.  There is a natural action of the Borel subgroup of upper triangular matrices $B_+$ on $\fl(\CC^n)$ by matrix multiplication.  The orbits $\Omega_w$ of this action, called \mydef{Schubert cells}, are indexed by permutations $w$ in the symmetric group $S_n$.  The closures $\mathfrak X_w=\overline{\Omega_w}$ of these orbits are called \mydef{Schubert varieties}.  Schubert varieties emerged in the study of the enumerative geometry problems posed by Schubert \cite{Sch79} and his contemporaries in the late nineteenth century. They have also played an essential role in the development of modern commutative algebra, providing crucial examples when the study of Cohen--Macaulay varieties was in its nascence.

Each Schubert variety gives rise to a \mydef{Schubert class} $\sigma_w$ in the integral cohomology ring $H^*(\fl(\CC^n))$.  Indeed, these Schubert classes form a $\mathbb Z$-linear basis for $H^*(\fl(\CC^n))$.
Borel \cite{Bor53} showed that $H^*(\fl(\CC^n))$ is isomorphic to $\mathbb Z[x_1,\ldots,x_n]/I^{S_n}$, where $I^{S_n}$ is the ideal generated by the nonconstant \notdefterm{elementary symmetric polynomials}.   Geometric properties of $\fl(\CC^n)$ are readily expressed in terms of Schubert classes. For instance, the coefficients $c^w_{u,v}$ in the product $\sigma_u\cdot\sigma_v=\sum_{w\in S_n} c_{u,v}^w\sigma_w$  are nonnegative integers; $c_{u,v}^w$ counts points in the intersection of three Schubert varieties that depend on $u$, $v$, and $w$ generically translated by elements of $\GL(\CC^n)$. Something that for decades hindered the study of $\mathbb Z[x_1,\ldots,x_n]/I^{S_n}$ was that there was no known choice of desirable coset representatives. 

Motivated by earlier work of Bern\v{s}te\u{\i}n, Gel'fand, and Gel'fand \cite{BGG73} as well as Demazure \cite{Dem74}, Lascoux and Sch\"utzenberger \cite{LS82} proposed one such choice: the \notdefterm{Schubert polynomials} $\mathfrak S_w(\mathbf x)$.  Schubert polynomials have many desirable combinatorial properties. Importantly, if $u, v \in S_n$, then, for $N$ sufficiently large with respect to $n$, the coefficients in the product $\mathfrak S_u(\mathbf x)\mathfrak S_v(\mathbf x)=\sum_{w \in S_N} c_{u,v}^w \mathfrak S_w(\mathbf x)$ agree with those arising from the multiplication of the corresponding Schubert classes in $H^*(\fl(\mathbb C^N))$.  Moreover, Schubert polynomials expand positively in the monomial basis, allowing for numerous combinatorial interpretations for these coefficients.  Of particular importance are the \emph{pipe dream formula} of \cite{BB93,BJS93,FK94,FS94} and a recent formula due to Lam, Lee, and Shimozono \cite{LLS21} in terms of \emph{bumpless pipe dreams}, which they introduced in their study of back stable Schubert calculus.  

Bumpless pipe dreams had also appeared earlier in a different form in work related to the study of the \notdefterm{six-vertex model}.  In this context, they are called \notdefterm{osculating lattice paths} (see, e.g.,\ \cite{Beh08}). In an unpublished preprint,  Lascoux \cite{Las} used the six-vertex model to give a formula for Grothendieck polynomials, which can be used to recover the formula of \cite{LLS21} for Schubert polynomials (see \cite{Wei21}).  By interpreting bumpless pipe dreams as planar histories for permutations, Lam, Lee, and Shimozono gave a formula for \emph{double Schubert polynomials} that is analogous to (but distinct from) the traditional pipe dream formula. Double Schubert polynomials represent classes of Schubert varieties in the Borel-equivariant cohomology of $\fl(\CC^n)$.  Lam, Lee, and Shimozono's innovation has inspired a great deal of further exploration of the combinatorics of Schubert polynomials (see, e.g.,\ \cite{FGS18,BS20,Hua20,Wei21, HPW, Hua21}).

Despite the combinatorial desirability of Schubert polynomials, there was for many years skepticism over whether they were really the right choice.   It was profoundly unclear whether Schubert polynomials reflected any of the geometric content of Schubert varieties. Progress on this front came by way of understanding torus-equivariant classes of \notdefterm{matrix Schubert varieties}.  

 The torus $T$ of diagonal matrices acts on $\mat(\mathbb C^n)$ by matrix multiplication on the right, and so we can study the ring $H^*_T(\mat(\mathbb C^n))\cong \mathbb Z[x_1,\ldots,x_n]$ of $T$-equivariant cohomology.
There is a projection map $\pi:\GL(\mathbb C^n)\rightarrow \fl(\CC^n)=B_-\backslash \GL(\CC^n)$ and an inclusion map $\iota:\GL(\mathbb C^n)\rightarrow \mat(\mathbb C^n)$ taking elements of the general linear group into the space of $n\times n$ matrices.  The 
\mydef{matrix Schubert variety} of $w$, introduced by Fulton \cite{Ful92}, is $X_w=\overline{\iota(\pi^{-1}(\mathfrak X_w))}$, which is an orbit closure for the natural $B_-\times B_+$ action on $\mat(\mathbb C^n)$. Because $X_w$ is stable under the action of $T$, it gives rise to a class $[X_w]_T\in H^*_T(\mat(\mathbb C^n))$.  Furthermore, this class is a polynomial representative for the Schubert class $\sigma_w$ in $H^*(\fl(\CC^n))$.  Remarkably, $[X_w]_T=\mathfrak S_w(\mathbf x)$, i.e., the coset representative for $\sigma_w$ that was singled out by Lascoux and Sch\"utzenberger is the same one identified by the theory of $T$-equivariant cohomology (see \cite{Ful92}, \cite[Theorem~4.2]{FR03}, \cite[Theorem~A]{KM05}).  In this sense, Schubert polynomials are canonical representatives for Schubert classes.  Analogously, double Schubert polynomials represent classes of matrix Schubert varieties in $H^*_{T\times T}(\mat(\mathbb C^n))$, and so double Schubert polynomials are identified as natural representatives for Schubert classes in $H_{B_+}^*(\fl(\CC^n))$. 

Furthermore, Knutson and Miller \cite{KM05} were able to use Gr\"obner geometry to explain the appearance of the traditional pipe dream formula for Schubert polynomials.  Fixing an \emph{anti-diagonal} term order $\sigma$ on the coordinate ring of $\mat(\mathbb C^n)$, one can degenerate $X_w$ to $\init_\sigma(X_w)$,  the scheme defined by the $\sigma$-initial ideal of the defining ideal of $X_w$, which Knutson and Miller showed to be a union of coordinate subspaces indexed by pipe dreams.  Through this work, the pipe dream formula gained geometric significance.

In this way, the pipe dream formula is a canonical choice of expression for Schubert polynomials, but only insofar as anti-diagonal term orders would be considered canonical term orders. Several years after \cite{KM05}, Knutson, Miller, and Yong \cite{KMY09} studied an arbitrary diagonal term order $\sigma$, but their results were restricted to the special case of \emph{vexillary} matrix Schubert varieties. They showed that, in this case, the irreducible components of $\init_\sigma(X_w)$ are indexed by \notdefterm{flagged tableaux} (or, equivalently, \notdefterm{diagonal pipe dreams}).  One  challenge of the diagonal degenerations of $X_w$ is that they are not always reduced.  For this reason, the complete story of the diagonal degenerations must include a count on the irreducible components with \emph{multiplicity} (see \cite[Chapter 12]{Eis95}). Outside of the vexillary setting, there was no combinatorial candidate  to index components of $\init_\sigma(X_w)$.
 
Recently, Hamaker, Pechenik, and Weigandt \cite{HPW} extended \cite{KMY09} to a wider class of matrix Schubert varieties.  They showed that in this larger special case the irreducible components of $\init_\sigma(X_w)$ are indexed by the bumpless pipe dreams of \cite{LLS21}.  The main theorem of the present work was previously conjectured by Hamaker, Pechenik, and Weigandt in \cite[Conjecture 1.2]{HPW}.

\begin{MainThm} Let $X_w$ be a matrix Schubert variety.  There exist diagonal term orders with respect to which the irreducible components of the Gr\"obner degeneration of $X_w$, counted with scheme-theoretic multiplicity, naturally correspond to the bumpless pipe dreams for the permutation $w$.
\end{MainThm}

In fact, we prove a more general version of this statement that applies to a family of term orders for which Gr\"obner degeneration reflects an algorithmic description of bumpless pipe dreams (see Theorem \ref{thm:main}).  Our theorem holds over an arbitrary field $\kappa$.  When $\kappa = \CC$, one recovers $T$-equivariant classes from the \notdefterm{multidegrees} of \cite{Jos84,Ros89}. Indeed, when $S$ is a multigraded polynomial ring over $\CC$ and $I$ is a multihomogeneous ideal, then the multidegree of $S/I$ is the class of $\spec(S/I)$ in the $T$-equivariant Chow ring of $\spec(S)$ (see \cite[Proposition~1.19]{KMS06}).  

We describe several consequences of our main theorem and of the machinery we build to prove it.  We give a recurrence on unions of matrix Schubert varieties, which leads to a recurrence on the corresponding multigraded Hilbert series, which in turn allows us to recover transition formulas for (double) Schubert and (double) Grothendieck polynomials (Proposition \ref{prop:hilbertAndKPoly} and Corollary \ref{cor:transitiongrothendieck}).  The algebro-geometric recurrence on unions of matrix Schubert varieties and that of the transition equations is mirrored by a corresponding transition on bumpless pipe dreams (Lemma~\ref{lem:bpdBij} and \cite[Section~5]{Wei21}).  This situation is (projectively) dual to that of \cite{Knu19}, which involves co-transition, pipe dreams, and unions of matrix Schubert varieties. Additionally, we initiate the study of Cohen--Macaulayness of unions of matrix Schubert varieties, of which our motivating examples are \emph{alternating sign matrix varieties} (see Corollary \ref{cor:CM}, Corollary \ref{cor:nonRadicalCM}, and Corollary \ref{cor:seqCM}).  

In light of the rich history of Schubert varieties in commutative algebra (see, e.g., \cite{HE71,Hoc73, DL81, HL82, MS83,    MR85,Ram85,Ram87, MS89, GL96, GM00}) and also myriad generalizations of classical determinantal varieties (see, e.g., \cite{HT92,Con96, CH97,BC98, Con98, LRT06, Gor07, Boo12, EHH13, Ber15, CDG20, FK20, CDF+21}), one might expect a wealth of commutative algebraic results on matrix Schubert varieties and alternating sign matrix varieties.  While there are results on matrix Schubert varieties using  commutative algebraic techniques \cite{Hsi13}, a good deal of our understanding of them comes via their connection to Schubert varieties or via primarily combinatorial techniques. 
For example, results on Hilbert--Samuel multiplicities and Castelnuovo--Mumford regularities of matrix Schubert varieties have primarily been pursued via combinatorial methods (see, e.g., \cite{LY12, WY12, PSW21, RRR+21, RRW23}).  The classification of Gorenstein matrix Schubert varieties \cite{WY06} passes through their connection to Schubert varieties and is geometrically and combinatorially driven. 

No comparable results are known by any method for alternating sign matrix varieties.  Indeed, very little is known about alternating sign matrices from the standpoint of commutative algebra.  It is not known when alternating sign matrix varieties are Cohen--Macaulay nor even equidimensional.  We hope that the current work showcases both the intrigue of matrix Schubert varieties and alternating sign matrix varieties from the perspective of commutative algebra and also 
the fact that both classes are amenable to study by such techniques.

\section*{Acknowledgments}
The authors thank Zach Hamaker, Oliver Pechenik, and Jenna Rajchgot for their previous collaboration with the authors on \cite{HPW} and \cite{KR}, respectively, and for valuable  discussions during the writing of the present paper.  They also thank Anders Buch, Allen Knutson, David Speyer, Alex Woo, and Alexander Yong for helpful conversations.  They express their gratitude to Allen Knutson, Ezra Miller, Oliver Pechenik, and Alexander Yong for comments on an earlier version of this manuscript.  The first author thanks Alain Lascoux for email correspondence concerning \cite{LS96}. The authors additionally thank Kuei-Nuan Lin and Yi-Huang Shen as well as a team of two anonymous referees for pointing out errors in earlier versions of what is now Lemma \ref{lem:primaryMult}.  They thank Daoji Huang and Matt Larson for helpful conversations regarding that lemma and especially Matt for ultimately showing Lemma 0H4J in the Stacks Project to the authors and explaining that \ref{lem:primaryMult} followed from it.  Finally, the authors thank the anonymous referees for their comments, which greatly improved this manuscript.

  The first author's travel was partially supported by an AMS-Simons Travel Grant.  The second author was partially supported by Bill Fulton's Oscar Zariski Distinguished University Professor Chair funds.  

\section{Background and preliminaries}

Throughout this paper, we will take $\kappa$ to be an arbitrary field. 

\subsection{Permutations}

Let $\mathbb N=\{0,1,2,\ldots\}$ and $\mathbb Z_+=\{1,2,3,\ldots\}$.  Given $m,n\in \mathbb Z_+$, let $[n]=\{1,2,\ldots,n\}$ and $[m,n]=\{i\in \mathbb Z_+:m\leq i\leq n\}$.  The \mydef{symmetric group} $S_n$ is the group of permutations of $n$ letters.  We often represent permutations in one-line notation.  It will sometimes also be convenient to represent permutations as \mydef{permutation matrices}.  We identify the permutation $w\in S_n$ with the matrix that has $1$'s in positions $(i,w(i))$ for all $i\in [n]$ and $0$'s in all other positions.  The transposition $t_{i,j}$ is the $2$-cycle $(ij)$, and we write $s_i$ for the simple reflection $(i\, i+1)$.  We use $\ell(w)=\#\{(i,j):i<j \text{ and } w(i)>w(j)\}$ to denote the \mydef{length} of $w \in S_n$.

The \mydef{(strong) Bruhat order} on $S_n$ is the transitive closure of covering relations of the form $w<wt_{i,j}$ if $\ell(w)+1=\ell(wt_{i,j})$.    There is another characterization of Bruhat order we will use: define the \mydef{rank function} of $w$ to be \[\rk_w(a,b)=\#\{(i,j)\in [a]\times [b]: w(i)=j\}.\]  Then $w\leq v$ if and only if $\rk_w(i,j)\geq \rk_v(i,j)$ for all $i,j\in[n]$.

In the tradition of \cite{KM05}, we will often use cardinal directions when describing relative positions of elements of an $n \times n$ grid representing $[n] \times [n]$.  Specifically, we will say that $(a,b)$ is \mydef{southeast} of $(c,d)$ if both $a\geq c$ and $b\geq d$.  We will say that $(a,b)$ is \mydef{maximally southeast} within a subset of $[n] \times [n]$ if there does not exist $(a',b') \neq (a,b)$ in the specified subset so that $a \leq a'$ and $b \leq b'$.  Use of the other cardinal directions occurs in the same manner.

Given $w\in S_n$, the \mydef{Rothe diagram} of $w$ is \[D(w)=\{(i,j):i,j\in [n], w(i)>j, \text{ and } w^{-1}(j)>i\}.\]  The length of $w$ satisfies $\ell(w)=\#D(w)$. The \mydef{essential set} of $w$ is 
\[\ess(w)=\{(i,j)\in D(w):(i+1,j),(i,j+1)\not \in D(w)\},\] i.e., the maximally southeast corners of the connected components of $D(w)$. 

We often visualize the Rothe diagram of the permutation $w$ by putting dots at each $(i,w(i))$ in an $n \times n$ grid and drawing a ray down and right from each such dot.  We call the boxes making up the $n \times n$ grid $\mydef{cells}$.  The set of cells without a dot or line through them makes up $D(w)$.  Borrowing from the ladder determinantal literature, we will call a maximally southeast element of $D(w)$ a \mydef{lower outside corner} of $D(w)$. See Figure \ref{fig:Rothe} for the visualization of $D(w)$ for $w=4721653$. 
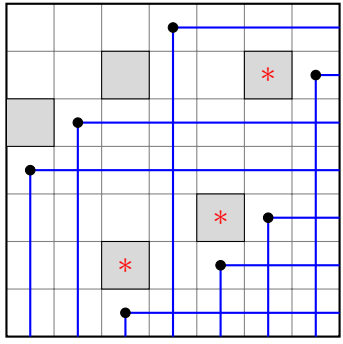
\begin{figure}[h]
\[
\begin{tikzpicture}[x=1.5em,y=1.5em]
\draw[step=1,gray, thin] (0,0) grid (7,7);
\draw[color=black, thick](0,0)rectangle(7,7);
\draw[thick, color=blue] (.5,0)--(.5,3.5)--(7,3.5);
\filldraw [black](.5,3.5)circle(.1);
\draw[thick, color=blue] (1.5,0)--(1.5,4.5)--(7,4.5);
\filldraw [black](1.5,4.5)circle(.1);
\draw[thick, color=blue] (2.5,0)--(2.5,.5)--(7,.5);
\filldraw [black](2.5,.5)circle(.1);
\draw[thick, color=blue] (3.5,0)--(3.5,6.5)--(7,6.5);
\filldraw [black](3.5,6.5)circle(.1);
\draw[thick, color=blue] (4.5,0)--(4.5,1.5)--(7,1.5);
\filldraw [black](4.5,1.5)circle(.1);
\draw[thick, color=blue] (5.5,0)--(5.5,2.5)--(7,2.5);
\filldraw [black](5.5,2.5)circle(.1);
\draw[thick, color=blue] (6.5,0)--(6.5,5.5)--(7,5.5);
\filldraw [black](6.5,5.5)circle(.1);
\filldraw[color=black, fill=gray!30](2,1)rectangle(3,2);
\filldraw[color=black, fill=gray!30](4,2)rectangle(5,3); 
\filldraw[color=black, fill=gray!30](5,5)rectangle(6,6);
\filldraw[color=black, fill=gray!30](2,5)rectangle(3,6);
\filldraw[color=black, fill=gray!30](0,4)rectangle(1,5);
\node at (2.5,1.5) {{\color{red}$\ast$}};
\node at (4.5,2.5) {{\color{red}$\ast$}};
\node at (5.5,5.5) {{\color{red}$\ast$}};
\end{tikzpicture}
\] \caption{Let $w=4721653$.  The cells of $\ess(w)$ have been shaded in gray, and the lower outside corners are ornamented with an {\color{red}$\ast$}. In particular, $\ess(w) = \{(2,3),(2,6),(3,1), (5,5),(6,3) \}$ and $ D(w) = \ess(w) \cup \{(1,1), (1,2), (1,3), (2,1), (2,2), (2,5), (5,3)\}$.  }
\label{fig:Rothe}
\end{figure}

A permutation $\pi\in S_n$ is \mydef{bigrassmannian} if $\#\ess(\pi)=1$.  A bigrassmannian permutation is uniquely determined by the position of its essential cell and the value of its rank function at this position. Explicitly, fix $a,b\in[n]$, and take $r$ so that $0\leq r<\min\{a,b\}$ and $a+b-r\leq n$.  Then we construct a bigrassmannian $\pi\in S_n$, which, in block matrix form, is given by \[\pi=\begin{pmatrix}
\mathbb{1}_r&\mathbf 0&\mathbf 0&\mathbf 0\\
\mathbf 0 &\mathbf 0& \mathbb{1}_{a-r}& \mathbf 0\\
\mathbf 0&\mathbb{1}_{b-r} &\mathbf 0 &\mathbf 0\\
\mathbf 0&\mathbf 0&\mathbf 0&\mathbb{1}_{n-a-b+r}
\end{pmatrix},\] where $\mathbb{1}_k$ denotes the identity matrix of size $k$.  By construction, $\ess(\pi)=\{(a,b)\}$ and $\rk_\pi(a,b)=r$.

\subsection{Alternating sign matrices}

An \mydef{alternating sign matrix} (ASM) is a square matrix with entries in $\{-1,0,1\}$ so that the entries in each row (and column) sum to $1$ and the nonzero entries in each row (and column) alternate in sign.  An ASM with no negative entries is a permutation matrix.  Write $\asm(n)$ for the set of $n\times n$ ASMs.

 The \mydef{corner sum function} of $A = (A_{i,j}) \in \asm(n)$ is defined by $ \rk_A(a,b)=\sum_{i=1}^a\sum_{j=1}^b A_{i,j}$ for $(a,b) \in [n] \times [n]$.   It will also be useful to define $\rk_A(i,j)=0$ whenever $i=0$ or $j=0$.  If $A\in S_{n}$, then $\rk_A$ agrees with the definition of the rank function of a permutation.  We may also use $\rk_A$ to denote the \mydef{corner sum matrix} of $A$, the $n \times n$ matrix whose $(i,j)^{th}$ entry is $\rk_A(i,j)$: that is,  $\rk_A = ((\rk_{A})_{i,j})  = (\rk_A(i,j))$.  When $A \in S_n$, the corner sum matrix is commonly called the \mydef{rank matrix} of $A$.

Corner sum functions induce a lattice structure on $\asm(n)$ defined by $A\geq B$  if and only if  $\rk_A(i,j)\leq \rk_B(i,j)$ for all $i,j\in [n]$.  Restricting to permutations recovers the (strong) Bruhat order on $S_n$; indeed, the ASM poset is the smallest lattice with this property \cite[Lemme~5.4]{LS96}. One computes the \mydef{join} (least upper bound) $A\vee B$ by taking entry-wise minima  of $\rk_A$ and $\rk_B$ and the \mydef{meet} (greatest lower bound) $A\wedge B$ by taking entry-wise maxima of $\rk_A$ and $\rk_B$.  The bigrassmannian permutations are the join-irreducibles of the lattice of ASMs. To compare a bigrassmannian with an ASM, it is enough to compare a single value of their corresponding corner sum functions.
\begin{lemma}[{\cite[Theorem~30]{BS17}}]
	\label{lemma:bigrasscompare}
	Fix a bigrassmannian $\pi \in S_n$ with $\ess(\pi)=\{(a,b)\}$.  Given $A\in \asm(n)$,  $\rk_A(a,b)\leq \rk_\pi (a,b)$ if and only if $\pi \leq A$.
\end{lemma}

Let $\perm(A)=\{w\in S_n:w\geq A, \text{ and, if } w\geq v\geq A \mbox{ for some $v \in S_n$, then }  w=v\}$.  We define $\hgt(A) =\min\{\ell(w):w\in \perm(A)\}$.  If $\ell(w)=\hgt(A) $ for all $w\in \perm(A)$, then we say that $A$ is \mydef{equidimensional}.

\begin{example}
\label{example:smallASM}
Let $A=\begin{pmatrix}0&1&0\\1&-1&1\\0&1&0\end{pmatrix}$.  Then $\rk_A=\begin{pmatrix}0&1&1\\1&1&2\\1&2&3\end{pmatrix}$.  The reader may verify that $\rk_{231}$, $\rk_{312}$, and $\rk_{321}$ are all entry-wise smaller than $\rk_A$; thus, $231,312,321>A$.  Furthermore, these are the only permutations in $S_3$ that are larger than $A$.  Since $321>231,312$, it follows that $\perm(A)=\{231,312\}$.  Because $\ell(231) = 2 = \ell(312)$, $A$ is equidimensional and $\hgt(A) =2$.
\end{example}

\subsection{Alternating sign matrix varieties}

Given a matrix $M$, let $M_{[i],[j]}$ be the submatrix of $M$ consisting of the first $i$ rows and $j$ columns.   Given $A\in \asm(n)$, we define the \mydef{ASM variety} of $A$ to be \[X_A=\{M\in \mat(n): \rk(M_{[i],[j]})\leq \rk_A(i,j) \text{ for all } i,j\in [n]\}.\]  When $w\in S_n$, we say $X_w$ is a \mydef{matrix Schubert variety}.  For background on matrix Schubert varieties, see \cite{Ful92, MS04}.

Fix an $n \times n$ generic matrix $Z=(z_{i,j})$, and let $R=\kappa[z_{1,1}, \ldots, z_{n,n}]$. We write $I_k(Z_{[i],[j]})$ for the ideal of $R$ generated by the $k$-minors in $Z_{[i],[j]}$.  As a convention, if $i=0$ or $j=0$, then define $I_k(Z_{[i],[j]})=(0)$. 
The \mydef{ASM ideal} of $A$ is \[
I_A=\sum_{i,j=1}^n I_{\rk_A(i,j)+1}(Z_{[i],[j]}).
\]

We call the union of the $(\rk_A(i,j)+1)$-minors in $Z_{[i],[j]}$, as $i$ and $j$ range from $1$ to $n$, the \mydef{natural generators} of $I_A$. 
 
If $w\in S_n$, $I_w$ is also called a \mydef{Schubert determinantal ideal}.  

\begin{proposition}[{\cite[Proposition~3.3]{Ful92}}]
\label{prop:matSchubFacts}
	Fix $w\in S_n$.
	\begin{enumerate}
		\item $I_w$ is prime.
		\item $\hgt (I_w)=\ell(w)$.
		\item $X_w$ is Cohen--Macaulay.
	\end{enumerate}
\end{proposition}
By \cite[Lemma~3.10]{Ful92}, a Schubert determinantal ideal can be generated by a (usually proper) subset of its natural generators: \[I_w=\sum_{(i,j)\in\ess(w)}I_{\rk_w(i,j)+1}(Z_{[i],[j]}).\]  We call these generators the \mydef{Fulton generators}.  There is a generalization of the Fulton generators for ASM ideals (see \cite[Lemma~5.9]{Wei}).

It is clear that $X_A=V(I_A)$, but it is not obvious that $I_A$ is radical.  We give a proof below of this and other fundamental facts about $I_A$. Before proceeding, we will need to recall the Stanley--Reisner correspondence.

Given a simplicial complex $\Delta$ on vertex set $[n]$, we define the \mydef{Stanley--Reisner ideal} $I_{\Delta}\subseteq \kappa[z_1,\dots, z_n]$ of $\Delta$ to be  
\[ I_\Delta = \left( \prod_{i \in U} z_i : U \subseteq [n], U \notin \Delta \right).\]  This map $\Delta\mapsto I_\Delta$ is a bijection from simplicial complexes on $[n]$ to squarefree monomial ideals of $\kappa[z_1,\dots, z_n]$. Let $\Delta(I)$ denote the simplicial complex associated to a squarefree monomial ideal $I$. For a subset $F \subseteq [n]$, observe that the prime ideal $P = (z_i : i \notin F)$ is a minimal prime of $I$ if and only if $F$ is a facet of $\Delta(I)$.  For further background, we refer the reader to \cite[Chapter 1]{MS04}.

Monomial ideals will typically occur for us as the initial ideals of Schubert determinantal ideals or ASM ideals.  For general background on term orders, initial ideals, and Gr\"obner bases, we refer the reader to \cite[Chapter 15]{Eis95}.   We will be especially interested in \emph{lexicographic} term orders.  For every ordering $\prec$ of the variables, there is a unique term order that is called the lexicographic term order on that ordering of the variables.  If, for example, $z_1 \prec z_2 \prec \cdots \prec z_n$ and $\mu$ and $\nu$ are two monomials in $\kappa[z_1, \ldots, z_n]$, then, to define the lexicographic term order $<$ on the ordering $\prec$ of the variables, we say that $\mu <\nu$ if there exists $i\in [n]$ such that $\max \{ k: z_j^k \mid \mu\} = \max\{ k: z_j^k \mid \nu\}$ for all $z_j \succ z_i$ and $\max\{k: z_i^k \mid \mu\} < \max\{ k: z_i^k \mid \nu\}$.

Because there are many orderings of the variables, there are many lexicographic term orders.  ``A lexicographic term order" should not be confused with ``the lexicographic ordering of the variables," which is sometimes used to mean $z_n \prec z_{n-1} \prec \cdots \prec z_1$.  Throughout this article, we will only use ``lexicographic" to refer to a full term order.  We will often describe orderings of the variables as ``reading orders" (for various languages).

For a term order $\sigma$ on a polynomial ring $R$ and an ideal $I$ of $R$, we will use $\init_\sigma(I)$ to denote the initial ideal of $I$ with respect to $\sigma$.  When $R$ is a polynomial ring in the entries of a generic matrix, we say that a term order is \mydef{diagonal} (respectively, \mydef{anti-diagonal}) if the leading term of each matrix minor is the product of the entries along the main diagonal (respectively, along the anti-diagonal).  

Each permutation $w\in S_n$ has an associated set of \notdefterm{(reduced) pipe dreams} $\pipes(w)$. We refer the reader to \cite{BB93} for background. Given $\mathcal D\in \pipes(w)$, we can record the locations of its crossing tiles $C(\mathcal D)\subseteq [n]\times [n]$. For any $E \in [n] \times [n]$, define the ideal $I_E = (z_{i,j}:(i,j)\in E)$.

\begin{theorem}[{\cite[Theorem~B]{KM05}}]
\label{theorem:KMTheoremB}
If $\sigma$ is any anti-diagonal term order on $\kappa[z_{1,1}, \ldots, z_{n,n}]$, then the Fulton generators of $I_w$ 
form a Gr\"obner basis.  In particular,  $\init_\sigma(I_w)$ is squarefree.  The facets of $\Delta(\init_\sigma(I_w))$ are \[\{[n]\times[n] - C(\mathcal D):\mathcal D\in \pipes(w)\},\]
and \[\init_\sigma(I_w)=\bigcap_{\mathcal D\in \pipes(w)} I_{C(\mathcal D)}.\]
\end{theorem}

\begin{remark}
Since the lattice of ASMs is finite (and hence complete), for any $w_1,\ldots,w_r\in S_n$,  $\vee\{w_1,\ldots,w_r\}\in \asm(n)$ is well defined.  In particular, though Lemma \ref{lemma:asmIdealFacts}, below, is phrased in terms of ASMs, it can just as well be thought of as a lemma about arbitrary intersections of matrix Schubert varieties.
\end{remark}

\begin{lemma}\label{lemma:asmIdealFacts}
Let $A\in \asm(n)$, and fix an anti-diagonal term order $\sigma$ on $ R = \kappa[z_{1,1},\ldots,z_{n,n}]$.   Then the following hold:
\begin{thmlist}
\item  If $w_1,\ldots,w_r\in S_n$ such that $A=\vee\{w_1,\ldots,w_r\}$, then
\[ \sum_{i=1}^r \init_\sigma(I_{w_i})=\init_\sigma(I_A)=\bigcap_{u\in \perm(A)} \init_\sigma(I_u).\]
\label{lempart:intersectInit}

    \item If $w_1,\ldots,w_r\in S_n$ such that $A=\vee\{w_1,\ldots,w_r\}$, then $ I_A={\displaystyle \sum_{i=1}^r} I_{w_i}.$ \label{lempart:addIwi}
    
     \item $I_A$ is radical. \label{lempart:asmRad}
    
   \item $I_A$ has the irredundant prime decomposition $\displaystyle I_A=\bigcap_{w\in \perm(A)}I_w$. \label{lempart:asmDecomp}
   
\item $\hgt (I_A)=\hgt(A)$. \label{lempart:asmHeight}

		\item $A$ is equidimensional if and only if $\spec(R/I_A)$ is equidimensional. \label{lempart:asmEquidim}
\end{thmlist}
\end{lemma}

\begin{proof}

\noindent (i) Since $A\geq w_{i}$  we have $\rk_A\leq \rk_{w_i}$, for all $i\in [r]$.  Thus, $I_{w_i}\subseteq I_{A}$ and so \[\sum_{i=1}^r I_{w_i}\subseteq I_A.\]
Similarly, since $A\leq u$ for all $u\in \perm(A)$, we know 
$ I_A\subseteq \bigcap\{ I_u:u\in \perm(A)\}$.

As such,
\begin{equation}
\label{eq:idealcontainments}
    \sum_{i=1}^r \init_\sigma(I_{w_i})\subseteq \init_\sigma\left(\sum_{i=1}^r I_{w_i}\right)\subseteq \init_\sigma(I_A)\subseteq \init_\sigma\left(\bigcap_{u\in \perm(A)}I_{u}\right)\subseteq \bigcap_{u\in \perm(A)}\init_\sigma(I_u).
\end{equation}

By Theorem~\ref{theorem:KMTheoremB}, for all $w\in S_n$, $\init_\sigma(I_w)$ is a squarefree monomial ideal with associated Stanley-Reisner complex $\Delta(\init_\sigma(I_w))$.  In particular, this implies that 
$ \sum_{i=1}^r\init_\sigma(I_{w_i})$ is squarefree with Stanley-Reisner complex 
$ \bigcap\{ \Delta(\init_\sigma(I_{w_i})):i \in [r]\}$
and that 
$\bigcap\{\init_\sigma(I_u):u\in \perm(A)\}$
is also squarefree with Stanley-Reisner complex 
$ \bigcup\{\Delta(\init_\sigma(I_u)):u\in \perm(A)\}$. 

By \cite[Proposition~4.8]{Wei}, \[\bigcap_{i \in [r]} \Delta(\init_\sigma(I_{w_i}))=\bigcup_{u\in \perm(A)}\Delta(\init_\sigma(I_u)).\]  Thus, all containments in Equation~\ref{eq:idealcontainments} are actually equalities.

\noindent (ii) Each ASM ideal is homogeneous with respect to the standard grading on $R$.
Because
\[\sum_{i=1}^r I_{w_i}\subseteq I_A\subseteq \bigcap_{u\in \perm(A)} I_u \mbox{ \hspace{1cm} and \hspace{1cm} } \displaystyle \init_\sigma\left(\sum_{i=1}^r I_{w_i}\right)=\init_\sigma\left (\bigcap_{u\in \perm(A)} I_u\right),
\]we know, by comparing Hilbert functions (see Subsection \ref{subsect:hilb}),
\[\sum_{i=1}^r I_{w_i}= I_A= \bigcap_{u\in \perm(A)} I_u.\]

\noindent (iii) First note that there exists a set $\{\pi_1,\ldots,\pi_m\}$ of bigrassmannian permutations in $S_n$ so that $A=\vee \{\pi_1,\ldots,\pi_m\}$ (this follows from \cite{LS96} and \cite[Proposition~9]{Rea02}). Thus, as a consequence of \ref{lempart:addIwi}, $\init_\sigma(I_A)$ is radical and so $I_A$ is also radical.

\noindent (iv) Again, noting that we can write $A= \vee \{\pi_1,\ldots,\pi_m\}$ for some set of bigrassmannian permutations, applying the argument in \ref{lempart:addIwi}, we conclude 
$ I_A= \bigcap\{ I_u:u\in \perm(A)\}$.

By Proposition~\ref{prop:matSchubFacts}, $I_u$ is prime for each $u\in \perm(A)$.  By definition, the elements of $\perm(A)$ are pairwise incomparable in Bruhat order.  Thus, applying \cite[Lemma~15.19]{MS04}, no $I_u$ properly contains any $I_{u'}$ for $u,u' \in \perm(A)$.  

\noindent (v)  We have $\hgt (I_A)=\min\{\hgt (I_u):u\in \perm(A)\}=\min\{\ell(u):u\in \perm(A)\}=\hgt(A)$.

\noindent (vi) This is immediate from \ref{lempart:asmDecomp} and \ref{lempart:asmHeight}.
\end{proof}

\begin{example}
Let $A$ be as in Example~\ref{example:smallASM}.  By examining $\rk_A$, we see that the only non-vacuous rank conditions come from $\rk_A(1,1)=0$ and $\rk_A(2,2)=1$.  Thus, \[I_A=\left(z_{11},\begin{vmatrix}z_{11}&z_{12}\\z_{21}&z_{22} \end{vmatrix}\right)=(z_{11}, z_{12}z_{21}) = (z_{11},z_{12})\cap(z_{11},z_{21}).\]
Recall that $\perm(A)=\{231,312\}$.  Furthermore, $I_{231}=(z_{11},z_{21})$ and $I_{312}=(z_{11},z_{12})$.  Hence, we have confirmed $I_A=\bigcap \{I_w : w\in \perm(A)\}$.
\end{example}

\subsection{Hilbert functions and multidegrees}
\label{subsect:hilb}
Fix a $\ZZ^d$-grading on a finitely generated $\kappa$-algebra $S$, and fix a finitely generated, $\ZZ^d$-graded $S$-module $M$.  
For $\mathbf{t} = (t_1, \ldots, t_d)$ and $\mathbf{a} = (a_1, \ldots, a_d)$, let $\mathbf{t}^\mathbf{a} = t_1^{a_1}\cdots t_d^{a_d}$ and $\langle\mathbf a,\mathbf t \rangle=a_1t_1+\cdots+a_dt_d$.  
Let $M_\mathbf{a}$ denote the $\mathbf{a}^{th}$ graded piece of $M$.  If $\dim_{\kappa\text{-vect}}(M_\mathbf{a})<\infty$
for all $\mathbf{a} \in \ZZ^d$, we define the \mydef{multigraded Hilbert series} of $M$, denoted $\hilb(M;\mathbf t)$, as \[
\hilb(M;\mathbf t) = \sum_{\mathbf{a} \in \ZZ^d}\dim_{\kappa\text{-vect}}(M_\mathbf{a}) \cdot \mathbf{t}^\mathbf{a}.
\]  
 Note that $\hilb(M;\mathbf t)$ is an element of the formal Laurent series ring $\ZZ((t_1, \ldots, t_d))$.  We refer the reader to \cite[Chapter 8]{MS04} for general background on multigraded Hilbert series.

We will often want to consider Hilbert series of finitely generated modules over polynomial rings, especially polynomial rings equipped with term orders.  We now restrict to that case.  For the remainder of this section, let $S = \kappa[z_1, \ldots, z_n]$, and assume that the degrees $\deg(z_{i})=(a_1,\ldots,a_d)$ of the algebra generators of $S$ 
all lie in a single open half-space of $\ZZ^d$.  
This assumption guarantees that, for each finitely generated $S$-module $M$, $\dim_{\kappa\text{-vect}}(M_\mathbf{a})<\infty$ for all $\mathbf{a} \in \ZZ^d$.  Recall that if $S$ is equipped with a term order $\sigma$ and $I$ is a homogeneous ideal of $S$, then $\hilb(S/I; \mathbf t) = \hilb(S/\init_\sigma(I); \mathbf t)$.  If $M$ is a finitely generated, $\ZZ^d$-graded $S$-module, then there is a unique Laurent polynomial $\mathcal{K}(M;\mathbf t) \in \ZZ[t_1, t_1^{-1}, \ldots, t_d, t_d^{-1}]$ so that \[
\hilb(M;\mathbf t) = \dfrac{\mathcal{K}(M;\mathbf t)}{ \displaystyle \prod_{i=1}^n (1-\mathbf t^{\deg(z_i)})}.
\]  We call $\mathcal{K}(M;\mathbf t)$ the \mydef{K-polynomial} of $M$.  The \mydef{multidegree} $\mathcal{C}(M;\mathbf t)$ consists of the lowest degree terms of  $\mathcal{K}(M;\mathbf 1-\mathbf t)$.  If $P$ is a minimal prime of $M$, let $\mult_P(M)$ denote the length of the finite length $S_P$-module $M_P$.  The map $M \mapsto \mathcal C(M;\mathbf t)$ satisfies three key properties (see \cite[Theorem~1.7.1]{KM05}):

\begin{thmlist}
	\item Additivity: If $P_1,\ldots, P_m$ are the associated primes of $M$ of minimal height, then \[\mathcal C(M;\mathbf t)=\sum_{i=1}^m \mult_{P_i}(M)\cdot \mathcal C(S/P_i;\mathbf t).\] \label{property:Additivity}
	\item Degeneration: If there is a Gr\"obner degeneration from $\spec(S/I)$ to $\spec(S/I')$, then \[\mathcal C(S/I;\mathbf t)=\mathcal C(S/I';\mathbf t).\] \label{property:Degeneration}
	\item Normalization: Given $D\subseteq [n]$, \[\mathcal C(S/(z_i:i\in D);\mathbf t)=\prod_{i\in D}\langle \deg(z_i),\mathbf t\rangle.\]\label{property:Normalization}
\end{thmlist}

If $S$ is standard graded and $I$ is a homogeneous ideal of height $h$, then, by degeneration and normalization, $\mathcal{C}(S/I;t) = e(S/I) \cdot t^h$ for some integer $e(S/I)$, which is called the \mydef{degree} of  the projective variety $\proj(S/I)$ (or the \emph{Hilbert--Samuel} multiplicity of $S/I$ on the homogeneous maximal ideal).  Property \ref{property:Additivity}  generalizes the ordinary associativity formula $e(M) =  \sum_{i =1}^m \mult_{P_i}(M) \cdot e(S/P_i)$ of the Hilbert--Samuel multiplicity studied by Lech \cite{Lec57}.

\section{Bumpless pipe dreams and transition equations}\label{sect:BPDandTransition}

\subsection{Bumpless pipe dreams}
A \mydef{bumpless pipe dream} (BPD) is a tiling of the $n\times n$ grid with the pictures in (\ref{eqn:sixpipes}) so that
\begin{enumerate}
	\item there are $n$ total pipes,
	\item pipes start at the bottom edge of the grid and end at the right edge, and
	\item pairwise, pipes cross at most one time.
\end{enumerate} 
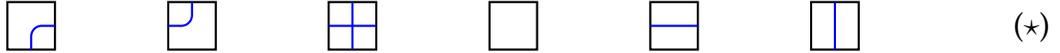
\begin{equation}\tag{$\star$}
\label{eqn:sixpipes}
\raisebox{-.5em}{
	\begin{tikzpicture}[x=1.5em,y=1.5em]
		\draw[color=black, thick](0,1)rectangle(1,2);
		\draw[thick,rounded corners,color=blue] (.5,1)--(.5,1.5)--(1,1.5);
	\end{tikzpicture}
	\hspace{3em}
	\begin{tikzpicture}[x=1.5em,y=1.5em]
		\draw[color=black, thick](0,1)rectangle(1,2);
		\draw[thick,rounded corners,color=blue] (.5,2)--(.5,1.5)--(0,1.5);
	\end{tikzpicture}
	\hspace{3em}
	\begin{tikzpicture}[x=1.5em,y=1.5em]
		\draw[color=black, thick](0,1)rectangle(1,2);
		\draw[thick,rounded corners,color=blue] (0,1.5)--(1,1.5);
		\draw[thick,rounded corners,color=blue] (.5,1)--(.5,2);
	\end{tikzpicture}
	\hspace{3em}
	\begin{tikzpicture}[x=1.5em,y=1.5em]
		\draw[color=black, thick](0,1)rectangle(1,2);
	\end{tikzpicture}
	\hspace{3em}
	\begin{tikzpicture}[x=1.5em,y=1.5em]
		\draw[color=black, thick](0,1)rectangle(1,2);
		\draw[thick,rounded corners,color=blue] (0,1.5)--(1,1.5);
	\end{tikzpicture}
	\hspace{3em}
	\begin{tikzpicture}[x=1.5em,y=1.5em]
		\draw[color=black, thick](0,1)rectangle(1,2);
		\draw[thick,rounded corners,color=blue] (.5,1)--(.5,2);
\end{tikzpicture}}
\end{equation}
Given a BPD $\mathcal B$, label its pipes $1,\ldots,n$ from left to right according to their starting columns.  We obtain a permutation $w_{\mathcal B}$ by defining $w_{\mathcal B}(i)$ to be the label of the pipe that terminates in row $i$.  Let $\bpd(w)$ be the set of BPDs of $w\in S_n$.  The name ``bumpless'' indicates that these tilings do not use the \mydef{bumping tile} \begin{tikzpicture}[x=1em,y=1em]
	\draw[color=black, thick](0,1)rectangle(1,2);
	\draw[thick,rounded corners,color=blue] (.5,1)--(.5,1.5)--(1,1.5);
	\draw[thick,rounded corners,color=blue](.5,2)--(.5,1.5)--(0,1.5);
\end{tikzpicture},
which appears in the pipe dreams of \cite{BB93}.

The \mydef{diagram} of $\mathcal B$ is the set \[
D(\mathcal B)=\{(i,j): \text{there is a blank tile in row } i \text{ and column } j \text{ of }\mathcal B\}.
\] 
We associate to $\mathcal B$ the weight \[ \wt(\mathcal B)=\prod_{(i,j)\in D(\mathcal B)}(x_i-y_j).\]   The \mydef{Rothe} BPD for $w$ is the (unique) BPD that has 	
		\begin{tikzpicture}[x=1em,y=1em]
			\draw[color=black, thick](0,1)rectangle(1,2);
			\draw[thick,rounded corners,color=blue] (.5,1)--(.5,1.5)--(1,1.5);
		\end{tikzpicture}
		tiles in cell $(i,w(i))$ for all $i\in [n]$ and has no 
		\begin{tikzpicture}[x=1em,y=1em]
			\draw[color=black, thick](0,1)rectangle(1,2);
			\draw[thick,rounded corners,color=blue] (.5,2)--(.5,1.5)--(0,1.5);
		\end{tikzpicture}
		tiles.  Notice that, if $\mathcal B$ is the Rothe BPD for $w$, then $D(w)=D(\mathcal B)$.   
		
		Fix a BPD $\mathcal B$.  Suppose the tile in cell $(i,j)$ is a downward elbow \begin{tikzpicture}[x=1em,y=1em]
			\draw[color=black, thick](0,1)rectangle(1,2);
			\draw[thick,rounded corners,color=blue] (.5,1)--(.5,1.5)--(1,1.5);
		\end{tikzpicture}.  Take $(a,b)\in D(\mathcal B)$ with $i<a$ and $j<b$.  Suppose further that the only \begin{tikzpicture}[x=1em,y=1em]
			\draw[color=black, thick](0,1)rectangle(1,2);
			\draw[thick,rounded corners,color=blue] (.5,1)--(.5,1.5)--(1,1.5);
		\end{tikzpicture} or \begin{tikzpicture}[x=1em,y=1em]
			\draw[color=black, thick](0,1)rectangle(1,2);
			\draw[thick,rounded corners,color=blue] (.5,2)--(.5,1.5)--(0,1.5);
		\end{tikzpicture} tile in the region $[i,a]\times[j,b]$ occurs in cell $(i,j)$.  Then we can take the pipe passing through $(i,j)$ and bend it within the rectangle so that there are downward elbows  $\begin{tikzpicture}[x=1em,y=1em]
			\draw[color=black, thick](0,1)rectangle(1,2);
			\draw[thick,rounded corners,color=blue] (.5,1)--(.5,1.5)--(1,1.5);
		\end{tikzpicture}$ in cells $(i,b)$ and $(a,j)$ and an upward elbow \begin{tikzpicture}[x=1em,y=1em]
			\draw[color=black, thick](0,1)rectangle(1,2);
			\draw[thick,rounded corners,color=blue] (.5,2)--(.5,1.5)--(0,1.5);
		\end{tikzpicture} in cell $(a,b)$ (see Example~\ref{example:transitionDroop}).  This move is called a \mydef{droop move}.  Applying a droop move to $\mathcal B\in \bpd(w)$ produces another element of $\bpd(w)$.  Furthermore, $\bpd(w)$ is connected by such moves:
		\begin{proposition}{{\cite[Proposition~5.3]{LLS21}}}
		Every $\mathcal B\in \bpd(w)$ can be reached from the Rothe BPD for $w$ by a sequence of droop moves.
		\end{proposition}
		
		\begin{corollary}
		\label{cor:BPDlowerOutsideCorner}
		If $(a,b)$ is a lower outside corner of $D(w)$ and $\mathcal B\in \bpd(w)$, then the tile in cell $(a,b)$ of $\mathcal B$ is either blank or is an upward elbow \begin{tikzpicture}[x=1em,y=1em]
			\draw[color=black, thick](0,1)rectangle(1,2);
			\draw[thick,rounded corners,color=blue] (.5,2)--(.5,1.5)--(0,1.5);
		\end{tikzpicture}.
		\end{corollary}

 Schubert polynomials are traditionally defined using divided difference operators.  The content of \cite[Theorem~5.13]{LLS21} is that those definitions are equivalent to the definitions we give here.   The \mydef{double Schubert polynomial} of $w \in S_n$ is the sum
\[\mathfrak S_w(\mathbf x, \mathbf y)=\sum_{\mathcal B\in \bpd(w)}\wt(\mathcal B).\]  We will also at times consider the \mydef{single Schubert polynomial} \[\mathfrak S_w(\mathbf x)=\sum_{\mathcal B\in \bpd(w)}\wt_{\mathbf x}(\mathcal B),\] where \[ \wt_{\mathbf x}(\mathcal B)=\prod_{(i,j)\in D(\mathcal B)}x_i.\]
Notice that $\mathfrak S_w(\mathbf x)=\mathfrak S_w(\mathbf x,\mathbf 0)$.

Write $\mathfrak S_w(\mathbf 1)$ for the principal specialization of $\mathfrak S_w(\mathbf x)$, (i.e., the result of substituting $x_i\mapsto 1$ for all $i\in[n]$).  From the BPD definition of $\mathfrak S_w(\mathbf x)$ given above, it is immediate that $\mathfrak S_w(\mathbf 1)=\#\bpd(w)$.

For the remainder of this section, we take $R=\kappa[z_{1,1},\ldots,z_{n,n}]$.  We will be interested in two gradings on $R$.  
The first is the $\mathbb Z^n$ grading that assigns generators the degrees $\deg(z_{i,j})=e_i$, the $i^{th}$ standard basis vector.  The second is the $\mathbb Z^{2n}$ grading for which $\deg(z_{i,j})=e_i+e_{n+j}$. Let $\mathbf x=(x_1,\ldots,x_n)$ and $\mathbf y=(y_1,\ldots,y_n)$.  When writing Hilbert functions with respect to the $\ZZ^n$ grading, we will have $\mathbf t=\mathbf x$ and, when with respect to the $\mathbb Z^{2n}$ grading, $\mathbf t=(\mathbf x,\mathbf y)$.

 \begin{theorem}[\cite{FR03,KM05}]
 \label{thm:SchubMultiDeg}
If $I_w$ is the Schubert determinantal ideal of $w \in S_n$, then, with respect to the $\mathbb Z^n$ and $\mathbb Z^{2n}$ gradings on $R$ given above, $\mathcal C(R/I_w;\mathbf x)=\mathfrak S_w(\mathbf x)$ and $\mathcal C(R/I_w;\mathbf x,\mathbf y)=\mathfrak S_w(\mathbf x, -\mathbf{y})$.
\end{theorem}

 \subsection{Transition equations}
 Fix $w \in S_n$ and a lower outside corner $(a,b)$ of $D(w)$.  
 Set $v=wt_{a,w^{-1}(b)}$.  Observe that $D(w) = D(v) \cup \{(a,b)\}$ and, in particular, that $\ell(w) = \ell(v)+1$.  \begin{notation} Let \[
 \phi(w,z_{a,b})=\{i\in [a-1]:vt_{i,a}>v \text{ and } \ell(vt_{i,a})=\ell(v)+1 = \ell(w)\}.
 \] Write $\Phi(w,z_{a,b})=\{vt_{i,a}:i\in \phi(w,z_{a,b})\}$.  
 \end{notation}

\begin{theorem}[{\cite[Proposition 4.1]{KV97}}]
	\label{thm:transition}
	Keeping the above notation, \[\mathfrak S_w(\mathbf x, \mathbf y)=(x_a-y_b)\cdot \mathfrak S_v(\mathbf x, \mathbf y)+\sum_{u\in \Phi(w,z_{a,b})}\mathfrak S_u(\mathbf x, \mathbf y).\]
\end{theorem}

If one takes the BPD formula of \cite{LLS21} as a definition for Schubert polynomials, Theorem~\ref{thm:transition} may be proved by appealing to the combinatorics of BPDs.  Indeed, there is a bijective explanation\footnote{See also \cite{KY04} for a diagrammatic interpretation of transition.}. Alternatively, one may define Schubert polynomials using transition equations and then recover the BPD formula as a consequence.  The following combinatorial lemma shows the two definitions are equivalent.

\begin{lemma}\label{lem:bpdBij}
 Fix $w \in S_n$ and a lower outside corner $(a,b)$ of $D(w)$.   
Set $v = wt_{a, w^{-1}(b)}$.  There is a bijection 
\[\psi:\displaystyle \bpd(w) \rightarrow \bpd(v) \hspace{1mm} \cup   \bigcup_{u \in \Phi(w, z_{a,b})} \bpd(u)\]
that respects the diagrams of the bumpless pipe dreams.  Specifically, if $\psi(\mathcal B)\in \bpd(v)$, then $D(\mathcal B)=D(\psi(\mathcal B))\cup \{(a,b)\}$.  Otherwise, $D(\mathcal B)=D(\psi(\mathcal B))$.
\end{lemma}
\begin{proof}
The Rothe BPD for $v$ is obtained from the Rothe BPD for $w$ by exchanging the rows in which pipe $b$ and pipe $w(a)$ exit.  
Because $\bpd(w)$ is connected by droop moves, tiles strictly below the maximally southeast
cells of $D(w)$ are the same in every element of $\bpd(w)$.  Thus, we can make the same modification of pipes $b$ and $w(a)$ in every BPD for $w$.  By Corollary~\ref{cor:BPDlowerOutsideCorner}, we have two cases.

\noindent Case 1: $\mathcal B\in \bpd(w)$ has a blank tile in cell $(a,b)$.

In this case, exchanging the rows in which pipes $b$ and $w(a)$ exit produces an element of $\bpd(v)$.  Define $\psi(\mathcal B)$ to be this element. It is clear from construction that $D(\mathcal B)=D(\psi(\mathcal B))\cup\{(a,b)\}.$

\noindent Case 2: $\mathcal B\in \bpd(w)$ has a \begin{tikzpicture}[x=1em,y=1em]
			\draw[color=black, thick](0,1)rectangle(1,2);
			\draw[thick,rounded corners,color=blue] (.5,2)--(.5,1.5)--(0,1.5);
		\end{tikzpicture} tile in cell $(a,b)$.

Since $(a,b)$ is an upward elbow tile, exchanging the exit rows of pipes $b$ and $w(a)$ introduces a bumping tile.  Define $\psi(\mathcal B)$ to be the BPD obtained by replacing this bumping tile with a crossing tile.  By construction, $D(\mathcal B)=D(\psi(\mathcal B))$.  Furthermore, 
we claim that
$\psi(\mathcal B)$ is an element of 
$ \bigcup\{ \bpd(u):u \in \Phi(w,z_{a,b})\}$.
		
That this map is well defined and bijective we leave as an exercise to the reader.  
The proof is nearly identical to the  arguments in \cite[Section~5]{Wei21}\footnote{The statement in \cite[Section~5]{Wei21} is K-theoretic; the cohomological version follows as a special case.}. 
 \end{proof}
 
 \begin{example}
 \label{example:transitionDroop}
Let $w=4721653$.  Fix the lower outside corner $(a,b)=(5,5)$ of $D(w)$.  The Rothe BPD of $w$ is pictured below with cell $(5,5)$ outlined.  Pictorially, we obtain the Rothe BPD for $v=wt_{5,6}=4721563$ by exchanging the rows in the Rothe BPD for $w$ in which pipe $b=5$ and pipe $w(a)=6$ exit.  See the green pipes below.
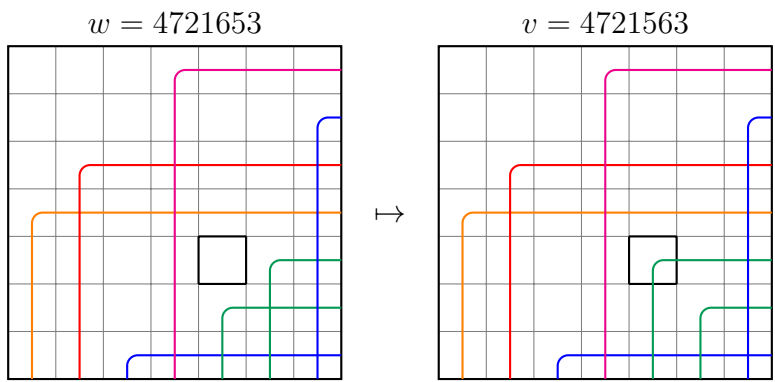
\[\begin{tikzpicture}[x=1.5em,y=1.5em]
\node at (3.5, 7.5) {$w=4721653$};
\draw[step=1,gray, thin] (0,0) grid (7,7);
\draw[step=1,black, thick] (4,2) grid (5,3);
\draw[color=black, thick](0,0)rectangle(7,7);
\draw[thick,rounded corners, color=orange] (.5,0)--(.5,3.5)--(7,3.5);
\draw[thick,rounded corners, color=red] (1.5,0)--(1.5,4.5)--(7,4.5);
\draw[thick,rounded corners, color=blue] (2.5,0)--(2.5,.5)--(7,.5);
\draw[thick,rounded corners, color=magenta] (3.5,0)--(3.5,6.5)--(7,6.5);
\draw[thick,rounded corners, color=ForestGreen] (4.5,0)--(4.5,1.5)--(7,1.5);
\draw[thick,rounded corners, color=ForestGreen] (5.5,0)--(5.5,2.5)--(7,2.5);
\draw[thick,rounded corners, color=blue] (6.5,0)--(6.5,5.5)--(7,5.5);
\end{tikzpicture}
\hspace{1em}
\raisebox{5em}{$\mapsto$}
\hspace{1em}
\begin{tikzpicture}[x=1.5em,y=1.5em]
\node at (3.5, 7.5) {$v=4721563$};
\draw[step=1,gray, thin] (0,0) grid (7,7);
\draw[step=1,black, thick] (4,2) grid (5,3);
\draw[color=black, thick](0,0)rectangle(7,7);
\draw[thick,rounded corners, color=orange] (.5,0)--(.5,3.5)--(7,3.5);
\draw[thick,rounded corners, color=red] (1.5,0)--(1.5,4.5)--(7,4.5);
\draw[thick,rounded corners, color=blue] (2.5,0)--(2.5,.5)--(7,.5);
\draw[thick,rounded corners, color=magenta] (3.5,0)--(3.5,6.5)--(7,6.5);
\draw[thick,rounded corners, color=ForestGreen] (4.5,0)--(4.5,2.5)--(7,2.5);
\draw[thick,rounded corners, color=ForestGreen] (5.5,0)--(5.5,1.5)--(7,1.5);
\draw[thick,rounded corners, color=blue] (6.5,0)--(6.5,5.5)--(7,5.5);
\end{tikzpicture}
\]
There are three pipes in the Rothe BPD for $w$ that are able to droop into $(5,5)$: the magenta, red, and orange pipes.  Each of these droops is pictured below.
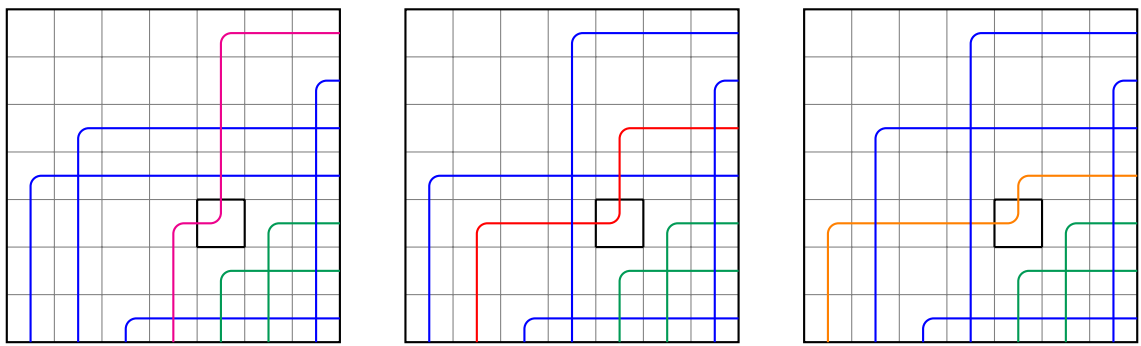
\[\begin{tikzpicture}[x=1.5em,y=1.5em]
\draw[step=1,gray, thin] (0,0) grid (7,7);
\draw[step=1,black, thick] (4,2) grid (5,3);
\draw[color=black, thick](0,0)rectangle(7,7);
\draw[thick,rounded corners, color=blue] (.5,0)--(.5,3.5)--(7,3.5);
\draw[thick,rounded corners, color=blue] (1.5,0)--(1.5,4.5)--(7,4.5);
\draw[thick,rounded corners, color=blue] (2.5,0)--(2.5,.5)--(7,.5);
\draw[thick,rounded corners, color=magenta] (3.5,0)--(3.5,2.5)--(4.5,2.5)--(4.5,6.5)--(7,6.5);
\draw[thick,rounded corners, color=ForestGreen] (4.5,0)--(4.5,1.5)--(7,1.5);
\draw[thick,rounded corners, color=ForestGreen] (5.5,0)--(5.5,2.5)--(7,2.5);
\draw[thick,rounded corners, color=blue] (6.5,0)--(6.5,5.5)--(7,5.5);
\end{tikzpicture} \hspace{2em}
\begin{tikzpicture}[x=1.5em,y=1.5em]
\draw[step=1,gray, thin] (0,0) grid (7,7);
\draw[step=1,black, thick] (4,2) grid (5,3);
\draw[color=black, thick](0,0)rectangle(7,7);
\draw[thick,rounded corners, color=blue] (.5,0)--(.5,3.5)--(7,3.5);
\draw[thick,rounded corners, color=red] (1.5,0)--(1.5,2.5)--(4.5,2.5)--(4.5,4.5)--(7,4.5);
\draw[thick,rounded corners, color=blue] (2.5,0)--(2.5,.5)--(7,.5);
\draw[thick,rounded corners, color=blue] (3.5,0)--(3.5,6.5)--(7,6.5);
\draw[thick,rounded corners, color=ForestGreen] (4.5,0)--(4.5,1.5)--(7,1.5);
\draw[thick,rounded corners, color=ForestGreen] (5.5,0)--(5.5,2.5)--(7,2.5);
\draw[thick,rounded corners, color=blue] (6.5,0)--(6.5,5.5)--(7,5.5);
\end{tikzpicture}  \hspace{2em}
\begin{tikzpicture}[x=1.5em,y=1.5em]
\draw[step=1,gray, thin] (0,0) grid (7,7);
\draw[step=1,black, thick] (4,2) grid (5,3);
\draw[color=black, thick](0,0)rectangle(7,7);
\draw[thick,rounded corners, color=orange] (.5,0)--(.5,2.5)--(4.5,2.5)--(4.5,3.5)--(7,3.5);
\draw[thick,rounded corners, color=blue] (1.5,0)--(1.5,4.5)--(7,4.5);
\draw[thick,rounded corners, color=blue] (2.5,0)--(2.5,.5)--(7,.5);
\draw[thick,rounded corners, color=blue] (3.5,0)--(3.5,6.5)--(7,6.5);
\draw[thick,rounded corners, color=ForestGreen] (4.5,0)--(4.5,1.5)--(7,1.5);
\draw[thick,rounded corners, color=ForestGreen] (5.5,0)--(5.5,2.5)--(7,2.5);
\draw[thick,rounded corners, color=blue] (6.5,0)--(6.5,5.5)--(7,5.5);
\end{tikzpicture}
\]
Taking these BPDs, again switch the rows in which the 5th and 6th pipes exit.  This produces diagrams that are \emph{not} BPDs, as they each have a bumping tile in cell $(5,5)$.
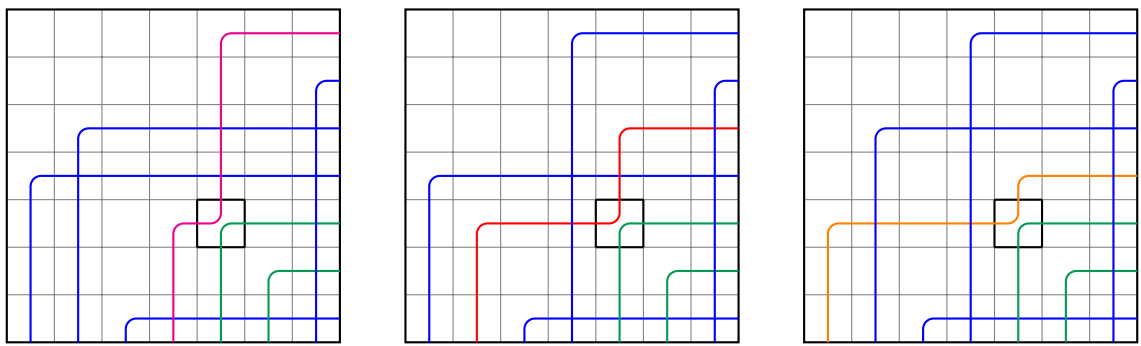
\[\begin{tikzpicture}[x=1.5em,y=1.5em]
\draw[step=1,gray, thin] (0,0) grid (7,7);
\draw[step=1,black, thick] (4,2) grid (5,3);
\draw[color=black, thick](0,0)rectangle(7,7);
\draw[thick,rounded corners, color=blue] (.5,0)--(.5,3.5)--(7,3.5);
\draw[thick,rounded corners, color=blue] (1.5,0)--(1.5,4.5)--(7,4.5);
\draw[thick,rounded corners, color=blue] (2.5,0)--(2.5,.5)--(7,.5);
\draw[thick,rounded corners, color=magenta] (3.5,0)--(3.5,2.5)--(4.5,2.5)--(4.5,6.5)--(7,6.5);
\draw[thick,rounded corners, color=ForestGreen] (4.5,0)--(4.5,2.5)--(7,2.5);
\draw[thick,rounded corners, color=ForestGreen] (5.5,0)--(5.5,1.5)--(7,1.5);
\draw[thick,rounded corners, color=blue] (6.5,0)--(6.5,5.5)--(7,5.5);
\end{tikzpicture} \hspace{2em}
\begin{tikzpicture}[x=1.5em,y=1.5em]
\draw[step=1,gray, thin] (0,0) grid (7,7);
\draw[step=1,black, thick] (4,2) grid (5,3);
\draw[color=black, thick](0,0)rectangle(7,7);
\draw[thick,rounded corners, color=blue] (.5,0)--(.5,3.5)--(7,3.5);
\draw[thick,rounded corners, color=red] (1.5,0)--(1.5,2.5)--(4.5,2.5)--(4.5,4.5)--(7,4.5);
\draw[thick,rounded corners, color=blue] (2.5,0)--(2.5,.5)--(7,.5);
\draw[thick,rounded corners, color=blue] (3.5,0)--(3.5,6.5)--(7,6.5);
\draw[thick,rounded corners, color=ForestGreen] (4.5,0)--(4.5,2.5)--(7,2.5);
\draw[thick,rounded corners, color=ForestGreen] (5.5,0)--(5.5,1.5)--(7,1.5);
\draw[thick,rounded corners, color=blue] (6.5,0)--(6.5,5.5)--(7,5.5);
\end{tikzpicture}
 \hspace{2em}
\begin{tikzpicture}[x=1.5em,y=1.5em]
\draw[step=1,gray, thin] (0,0) grid (7,7);
\draw[step=1,black, thick] (4,2) grid (5,3);
\draw[color=black, thick](0,0)rectangle(7,7);
\draw[thick,rounded corners, color=orange] (.5,0)--(.5,2.5)--(4.5,2.5)--(4.5,3.5)--(7,3.5);
\draw[thick,rounded corners, color=blue] (1.5,0)--(1.5,4.5)--(7,4.5);
\draw[thick,rounded corners, color=blue] (2.5,0)--(2.5,.5)--(7,.5);
\draw[thick,rounded corners, color=blue] (3.5,0)--(3.5,6.5)--(7,6.5);
\draw[thick,rounded corners, color=ForestGreen] (4.5,0)--(4.5,2.5)--(7,2.5);
\draw[thick,rounded corners, color=ForestGreen] (5.5,0)--(5.5,1.5)--(7,1.5);
\draw[thick,rounded corners, color=blue] (6.5,0)--(6.5,5.5)--(7,5.5);
\end{tikzpicture}
\]
We resolve this issue by replacing each bumping tile with a crossing tile.  This produces the Rothe BPDs for each permutation in $\Phi(w,z_{55}) = \{5721463, 4751263, 4725163\}$. 
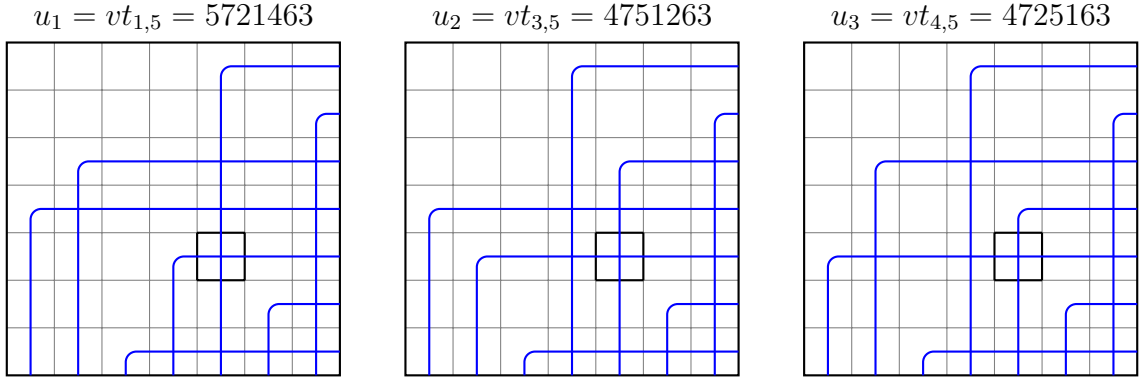
\[\begin{tikzpicture}[x=1.5em,y=1.5em]
\node at (3.5, 7.5) {$u_1=vt_{1,5}=5721463$};
\draw[step=1,gray, thin] (0,0) grid (7,7);
\draw[step=1,black, thick] (4,2) grid (5,3);
\draw[color=black, thick](0,0)rectangle(7,7);
\draw[thick,rounded corners, color=blue] (.5,0)--(.5,3.5)--(7,3.5);
\draw[thick,rounded corners, color=blue] (1.5,0)--(1.5,4.5)--(7,4.5);
\draw[thick,rounded corners, color=blue] (2.5,0)--(2.5,.5)--(7,.5);
\draw[thick,rounded corners, color=blue] (3.5,0)--(3.5,2.5)--(7,2.5);
\draw[thick,rounded corners, color=blue] (4.5,0)--(4.5,6.5)--(7,6.5);
\draw[thick,rounded corners, color=blue] (5.5,0)--(5.5,1.5)--(7,1.5);
\draw[thick,rounded corners, color=blue] (6.5,0)--(6.5,5.5)--(7,5.5);
\end{tikzpicture}
\hspace{2em}
\begin{tikzpicture}[x=1.5em,y=1.5em]
\node at (3.5, 7.5) {$u_2=vt_{3,5}=4751263$};
\draw[step=1,gray, thin] (0,0) grid (7,7);
\draw[step=1,black, thick] (4,2) grid (5,3);
\draw[color=black, thick](0,0)rectangle(7,7);
\draw[thick,rounded corners, color=blue] (.5,0)--(.5,3.5)--(7,3.5);
\draw[thick,rounded corners, color=blue] (1.5,0)--(1.5,2.5)--(7,2.5);
\draw[thick,rounded corners, color=blue] (2.5,0)--(2.5,.5)--(7,.5);
\draw[thick,rounded corners, color=blue] (3.5,0)--(3.5,6.5)--(7,6.5);
\draw[thick,rounded corners, color=blue] (4.5,0)--(4.5,4.5)--(7,4.5);
\draw[thick,rounded corners, color=blue] (5.5,0)--(5.5,1.5)--(7,1.5);
\draw[thick,rounded corners, color=blue] (6.5,0)--(6.5,5.5)--(7,5.5);
\end{tikzpicture}
\hspace{2em}
\begin{tikzpicture}[x=1.5em,y=1.5em]
\node at (3.5, 7.5) {$u_3=vt_{4,5}=4725163$};
\draw[step=1,gray, thin] (0,0) grid (7,7);
\draw[step=1,black, thick] (4,2) grid (5,3);
\draw[color=black, thick](0,0)rectangle(7,7);
\draw[thick,rounded corners, color=blue] (.5,0)--(.5,2.5)--(7,2.5);
\draw[thick,rounded corners, color=blue] (1.5,0)--(1.5,4.5)--(7,4.5);
\draw[thick,rounded corners, color=blue] (2.5,0)--(2.5,.5)--(7,.5);
\draw[thick,rounded corners, color=blue] (3.5,0)--(3.5,6.5)--(7,6.5);
\draw[thick,rounded corners, color=blue] (4.5,0)--(4.5,3.5)--(7,3.5);
\draw[thick,rounded corners, color=blue] (5.5,0)--(5.5,1.5)--(7,1.5);
\draw[thick,rounded corners, color=blue] (6.5,0)--(6.5,5.5)--(7,5.5);
\end{tikzpicture}
\]
Thus, \[\mathfrak S_w(\mathbf x, \mathbf y)=(x_5-y_5)\cdot \mathfrak S_v(\mathbf x, \mathbf y)+\mathfrak S_{u_1}(\mathbf x, \mathbf y)+\mathfrak S_{u_2}(\mathbf x, \mathbf y)+\mathfrak S_{u_3}(\mathbf x, \mathbf y),\] where $u_1=5721463$, $u_2=4751263$, and $u_3=4725163$.
\end{example}

In Section~\ref{section:additionalConsequences}, we give a proof of Theorem~\ref{thm:transition} based on a geometric recurrence we develop in Sections \ref{section:GVD} and \ref{section:proofs}.

\section{Bumpless pipe dreams and geometric vertex decomposition}\label{section:GVD}

Geometric vertex decomposition, which was introduced by Knutson, Miller, and Yong \cite{KMY09} in the study of vexillary matrix Schubert varieties, will be one of our main tools for understanding Schubert determinantal ideals in the context of BPDs.  We will use it to obtain an algebro-geometric recurrence that mirrors the combinatorial recurrences we have seen in BPDs and transition.  Throughout this section, we will take $R = \kappa[z_{1,1},\dots, z_{n,n}]$ and assume that all ideals are ideals of $R$ unless otherwise stated.

Fix one of the algebra generators $z_{a,b}$ of $R$, and set $y = z_{a,b}$.   For a polynomial $f =  \sum_{i = 0}^m \alpha_i y^i \in R$ with each $\alpha_i \in \kappa[z_{1,1},\dots,\widehat{y},\dots, z_{n,n}]$ and $\alpha_m \neq 0$, define the \mydef{initial $y$-form of $f$} to be $\init_y (f) =\alpha_m y^m$.  Given an ideal $I$, define $\init_y (I)$ to be the ideal generated by the initial $y$-forms of $I$; that is, $\init_y (I) = ( \init_y (f) : f \in I )$.   We say that a term order $\sigma$ on $R$ is \mydef{$y$-compatible} if it satisfies $\init_\sigma (f) = \init_\sigma(\init_y (f))$  for every $f\in R$. Notice that whenever a term order $\sigma$ is $y$-compatible, $\init_\sigma(I) = \init_\sigma(\init_y(I))$.  Important examples of $y$-compatible term orders are lexicographic term orders in which $y$ is the largest variable.  

\begin{defn}\label{def:gvd}\cite[Section 2.1]{KMY09}
Suppose that $I$ is an ideal of a polynomial ring that is equipped with  a $y$-compatible term order $\sigma$ and that $I$ has a Gr\"obner basis of the form $\mathcal{G} = \{yq_1+r_1, \ldots, yq_k+r_k, h_1, \ldots, h_\ell\}$ where $y$ does not divide any $q_i$, $r_i$, or $h_i$.  We define $C_{y,I} = (q_1, \ldots, q_k, h_1, \ldots, h_\ell)$ and $
N_{y,I} = (h_1, \ldots, h_\ell)$.  
Then $\init_y (I) = C_{y,I} \cap (N_{y,I}+(y))$, and this decomposition is called a \mydef{geometric vertex decomposition of $I$ with respect to $y$}.  
\end{defn}

Throughout this paper, for an ideal $I$ and variable $y$, we will only use the notation $C_{y,I}$ and $N_{y,I}$ to denote the ideals constructed as in Definition \ref{def:gvd}.

It follows from \cite[Theorem 2.1]{KMY09} that the ideals $C_{y,I}$ and $N_{y,I}$ do not depend on the choice of $y$-compatible term order and can be computed from any Gr\"obner basis for any such term order (rather than requiring knowledge of the reduced Gr\"obner basis).  If $y \in I$, then $C_{y,I} = R$ and $N_{y,I}+(y) = I$.  We will be performing geometric vertex decompositions only  when $\spec (R/I)$ is reduced and equidimensional.  In this case and when $y \notin I$, then, by \cite[Proposition 2.4]{KR}, either $N_{y,I} = C_{y,I}$, in which case $I$ has a generating set that does not involve $y$, or, by  \cite[Lemma 2.8]{KR}, $\spec (R/C_{y,I})$ is also equidimensional, and \[
\dim(\spec (R/I) )= \dim(\spec (R/C_{y,I})) = \dim(\spec (R/N_{y,I}))-1.
\]  For more information on geometric vertex decomposition, we refer the reader to \cite{KMY09, KR}. 

When an ideal $I$ has a generating set of the form $\{yq_1+r_1, \ldots, yq_k+r_k, h_1, \ldots, h_\ell\}$ where $y$ does not divide any $q_i$, $r_i$, or $h_i$, we will say that $I$ is \mydef{linear in $y$}.  Note that, if $I$ is linear in $y$, then the reduced Gr\"obner basis $\mathcal{G}$ of $I$ with respect to any $y$-compatible term order $\sigma$ will satisfy $\init_\sigma(g) \notin (y^2)$ for all $g \in \mathcal{G}$.  Hence, to perform a geometric vertex decomposition of an ideal $I$ with respect to $y$, it suffices to know that $I$ is linear in $y$.

We begin with a lemma that we will use to understand the geometric vertex decomposition of the Schubert determinantal ideal $I_w$ at a lower outside corner of $D(w)$.

\begin{lemma}\label{lem:NisSchub}
Fix $w \in S_n$ and a lower outside corner $(a,b)$ of $D(w)$.
 Write the Fulton generators of $I_w$ as $\{z_{a,b}q_1+r_1, \ldots, z_{a,b}q_k+r_k, h_1, \ldots, h_\ell\}$, where $z_{a,b}$ does not divide any $q_i, r_i,$ or $h_j$.  If $N=(h_1, \ldots, h_\ell)$, then $N$ is the Schubert determinantal ideal $I_v$ for $v=wt_{a,w^{-1}(b)}$.  
 The Fulton generators  of $I_v$ form a subset  of $\{h_1,\ldots,h_\ell\}$.
\end{lemma}
\begin{proof}
\noindent We first claim that  $N\subseteq I_v$.  We will show that each $h_j$ is a natural generator of $I_v$.
Because $h_j$ is a Fulton generator for $I_w$, we know $h_j\in I_{\rk_w(c,d)+1}(Z_{[c],[d]})$ for some $(c,d)\in \ess(w)$.  If $(c,d)=(a,b)$, since $h_j$ does not involve the variable $z_{a,b}$, we must have $h_j\in I_{\rk_w(c,d)+1}(Z_{[c-1],[d]})$ or $h_j\in I_{\rk_w(c,d)+1}(Z_{[c],[d-1]})$.

Because $wt_{a,w^{-1}(b)}=v$ and $w>v$, we have \[\rk_v(c,d)-\rk_{w}(c,d)=
\begin{cases}
1& \text{if } a\leq c<w^{-1}(b) \text{ and } b\leq d <w(a), \\
0& \text{otherwise.}
\end{cases}
\]  Since $(c,d)\in D(w)$, we have $\rk_w(c,d)=\rk_w(c-1,d)=\rk_w(c,d-1)$ (see, e.g., \cite[Lemmas 3.3 and 3.5]{Wei21}).  Thus, $\rk_w(c,d)=\rk_v(c-1,d)=\rk_v(c,d-1)$.  
Therefore, 
\[h_j\in I_{\rk_w(c,d)+1}(Z_{[c-1],[d]})=I_{\rk_v(c-1,d)+1}(Z_{[c-1],[d]})\]
or
\[h_j\in I_{\rk_w(c,d)+1}(Z_{[c],[d-1]})=I_{\rk_v(c,d-1)+1}(Z_{[c],[d-1]}).\]
In both cases, $h_j$ is a natural generator of $I_v$.

Now suppose $(c,d)\neq (a,b)$.  In this case, $\rk_w(c,d)=\rk_v(c,d)$, and so $h_j$ is a natural generator of $I_v$.  Since each $h_j$ is a natural generator of $I_v$, we conclude $N\subseteq I_v$.

We next claim that each Fulton generator for $I_v$ is an element of the set $\{h_1,\ldots, h_\ell\}$, from which it follows that $I_v\subseteq N$.  Because $D(v)=D(w)-\{(a,b)\}$, we have 
\[\ess(v)\subseteq (\ess(w)-\{(a,b)\})\cup \{(a-1,b),(a,b-1)\}.\]  Take a Fulton generator $h\in I_{\rk_v(c,d)+1}(Z_{[c],[d]})$ for some $(c,d) \in \ess(v)$.
If $(c,d)\in \ess(w)-\{(a,b)\}$, then, because $\rk_v(c,d)=\rk_w(c,d)$, $h$ is a Fulton generator of $I_w$.  
Otherwise, $(c,d)\in \{(a-1,b),(a,b-1)\}$.  In that case, $\rk_v(c,d)=\rk_w(c,d) = \rk_w(a,b)$, which implies that $h$ is a Fulton generator of $I_w$ that is among the natural generators of $I_{\rk_w(a,b)+1}(Z_{[a],[b]})$.  In both cases, $h$ is a Fulton generator of $I_w$ that does not involve the variable $z_{a,b}$, and so $h\in \{h_1,\ldots, h_\ell\}$.  Hence, $I_v\subseteq N$.

Thus, we conclude $I_v=N$.
\end{proof}

There is one uncommon family of term orders that will be of use to us throughout the remainder of this paper.  We now describe and name those term orders for later repeated use.
\begin{notation}[The term orders $\tau_{a,b}$]\label{not:tau}
 Consider the order on the variables appearing in the $n \times n$ matrix $Z = (z_{i,j})$ starting from the northeast corner of $Z$, then reading left across the top row, and then right to left along the second row, and so on.  We will call this right-to-left reading order.  Note that the lexicographic term order on right-to-left reading order is an anti-diagonal term order. 
 Throughout this paper, we will use  \mydef{$\tau_{a,b}$} to denote the lexicographic order in which $z_{a,b}$ is the largest variable and the remaining variables appearing in $Z$ are ordered in right-to-left reading order (skipping over $z_{a,b}$ as it would appear later in right-to-left reading order). 
\end{notation}

\begin{example}
Let $w=1243$.  Then 
\[I_w=\left(\begin{vmatrix} z_{11} & z_{12} & z_{13} \\
z_{21} & z_{22} & z_{23} \\
z_{31} & z_{32} & z_{33} 
 \end{vmatrix}\right).\]
 The cell $(3,3)$ is a lower outside corner of $D(w)$.  The term order $\tau_{3,3}$ is the lexicographic order with $z_{33}>z_{14}>z_{13}>z_{12}>z_{11}>z_{24}>z_{23}>z_{22}>z_{21}>z_{34}>z_{32}>z_{31}>z_{44}>z_{43}>z_{42}>z_{41}.$   In particular, \[\init_{\tau_{3,3}}\begin{vmatrix} z_{11} & z_{12} & z_{13} \\
z_{21} & z_{22} & z_{23} \\
z_{31} & z_{32} & z_{33} 
 \end{vmatrix}=z_{33}z_{12}z_{21}. \qedhere\] 
\end{example}

\begin{proposition}\label{prop:naturalGensGrobner}
Fix $w \in S_n$ and a lower outside corner $(a,b)$ of $D(w)$.  The Fulton generators of $I_w$ 
 form a Gr\"obner basis
with respect to $\tau_{a,b}$ (as in Notation \ref{not:tau}).
\end{proposition}

\begin{proof}
Write $I_w = (z_{a,b}q_1+r_1, \ldots, z_{a,b}q_k+r_k, h_1, \ldots, h_\ell)$, where $z_{a,b}$ does not divide any term of any of the $q_i, r_i$, or $h_i$ and the given generators are the Fulton generators.  If $\rk_w(a,b) = 0$, then $z_{a,b} \in I_w$, and the result is obvious 
 as the Fulton generators form an anti-diagonal Gr\"obner basis \cite[Theorem~B]{KM05}.  
 When
$\rk_w(a,b)\geq 1$, we will use \cite[Corollary 4.13]{KR}. 

Set $C = (q_1, \ldots, q_k, h_1, \ldots, h_\ell)$ and $N = (h_1, \ldots, h_\ell)$.  (Recall from Definition \ref{def:gvd} that, once we have shown that the generators named above for $I_w$ form a Gr\"obner basis under a $z_{a,b}$-compatible term order, then we will have shown $C = C_{z_{a,b},I_w}$ and $N = N_{z_{a,b},I_w}$.  We have named $C$ and $N$ in anticipation of this result.)

By Lemma \ref{lem:NisSchub}, $N$ is the Schubert determinantal ideal of the permutation $v = wt_{a, w^{-1}(b)}$, and the given generators of $N = I_v$ contain the Fulton generators. Because $N$ does not involve $z_{a,b}$, $\tau_{a,b}$ restricts to the lexicographic term order on right-to-left reading order on the variables involved in some Fulton generator of $N$.  Because the lexicographic term order on right-to-left reading order is an anti-diagonal order, the Fulton generators of $N$
 form a Gr\"obner basis
under $\tau_{a,b}$ \cite[Theorem~B]{KM05}; hence, so too does the set $\{h_1, \ldots, h_\ell\}$.  

Let $\pi$ be the bigrassmannian permutation whose unique essential cell is $(a-1,b-1)$ with rank condition $\rk_\pi(a-1,b-1) = \rk_w(a,b)-1$. Let $H$ denote the subset of $\{h_1, \ldots, h_\ell\}$ consisting of the $h_i$ that are also Fulton generators of $I_\pi$.  Then $\{q_1, \ldots, q_k\} \cup H$ is the set of Fulton generators of $I_\pi$, which form an anti-diagonal Gr\"obner basis.  Now $C = I_\pi + I_v=I_{\pi \vee v}$ by Lemma~\ref{lemma:asmIdealFacts}\ref{lempart:addIwi}.   In particular, the generating set
$\{q_1, \ldots, q_k, h_1, \ldots, h_\ell\}$ for $C$ is the concatenation of a Gr\"obner basis for $I_v$ and a Gr\"obner basis for $I_\pi$ under $\tau_{a,b}$.   Since $C$, $I_v$, and $I_\pi$ all have generating sets that do not involve $z_{a,b}$, we may apply Lemma~\ref{lemma:asmIdealFacts}\ref{lempart:intersectInit}, which tells us $\init_{\tau_{a,b}}(C)=\init_{\tau_{a,b}}(I_v)+\init_{\tau_{a,b}}(I_\pi)$.  Thus, $\{q_1, \ldots, q_k, h_1, \ldots, h_\ell\}$ is a Gr\"obner basis for $C$ under $\tau_{a,b}$.

Because $\ell(v) = \ell(w)-1$, the height of $N$ is one less than the height of $I_w$.  Because $N$ is a Schubert determinantal ideal, it is prime and so, in particular, height unmixed, (i.e., all associated primes of $N$ have the same height).
Because $C$ properly contains the prime ideal $N$, the height of $C$ is strictly greater than the height of $N$.  Because $\tau_{a,b}$ is a lexicographic order with $z_{a,b}$ 
the largest variable, $\init_{\tau_{a,b}}(z_{a,b}q_i+r_i) = z_{a,b} \cdot \init_{\tau_{a,b}}(q_i)$ for all $i \in [k]$.  

Finally, we claim that $q_i r_j-q_jr_i \in N$ for all $i,j \in [k] \times [k]$.  Fix a diagonal term order $\sigma$ with $z_{a,b}$ largest among variables appearing in at least one Fulton generator of $I_w$.  Consider the ideal $J =I_{\rk_w(a,b)+1}(Z_{[a],[b]})$, for which the natural generators form a diagonal Gr\"obner basis.  Then \[
q_i r_j-q_jr_i = (z_{a,b}q_j+r_j)q_i-(z_{a,b}q_i+r_i)q_j \in J.
\]
Because $q_i r_j-q_jr_i$ does not involve $z_{a,b}$, which is the lexicographically  largest variable, 
it must have remainder $0$ on division by the generators of $J$ that do not involve $z_{a,b}$, each of which is one of the natural generators of $N$.  Hence, $q_i r_j-q_jr_i \in N$ for all $i,j \in [k] \times [k]$, and so  the result follows from \cite[Corollary 4.13]{KR}.
\end{proof}

In the proof of Proposition \ref{prop:naturalGensGrobner}, one may alternatively use the Frobenius splitting described in \cite[Section~7.2]{Knu09} to show that the given generators of $C$
form a Gr\"obner basis 
under $\tau_{a,b}$.

With notation and assumptions as in Proposition \ref{prop:naturalGensGrobner}, we are now entitled to write $C = C_{z_{a,b},I_w}$ and $N = N_{z_{a,b},I_w} = I_v$. For later convenience, we record as a corollary the relationship between $C_{z_{a,b},I_w}$, $I_v$, and $I_\pi$ discussed in the proof of Proposition \ref{prop:naturalGensGrobner}.

\begin{corollary}\label{cor:formOfC}
Fix $w \in S_n$ and a lower outside corner $(a,b)$ of $D(w)$.  Assume $\rk_w(a,b)\geq 1$. 
Write $v = wt_{a, w^{-1}(b)}$, and let $\pi$ be the bigrassmannian permutation so that $\ess(\pi)=\{(a-1,b-1)\}$ and $\rk_\pi(a-1,b-1)=\rk_w(a,b)-1$.   Then $C_{z_{a,b},I_w} = I_v+I_\pi = I_{v \vee \pi}$ and $\init_{z_{a,b}} (I_w) = I_{v \vee \pi} \cap (I_v+(z_{a,b}))$.
\end{corollary} 

We now give a combinatorial lemma that allows us to use Corollary \ref{cor:formOfC} to identify the associated primes of $C_{z_{a,b},I_w}$ if $w \in S_n$, $(a,b)$ is a lower outside corner of $D(w)$, and $\rk_w(a,b)\geq 1$.

 \begin{proposition}
\label{prop:ASMdecomp}
Fix $w \in S_n$ and a lower outside corner $(a,b)$ of $D(w)$.  
 Assume $\rk_w(a,b)\geq 1$.
 Set $v=wt_{a,w^{-1}(b)}$, and let $\pi$ be the bigrassmannian permutation so that $\ess(\pi)=\{(a-1,b-1)\}$ and $\rk_\pi(a-1,b-1)=\rk_w(a,b)-1$. 
	Then $\perm(v \vee \pi)=\Phi(w,z_{a,b})$ and $\hgt(v \vee \pi)=\ell(w)$.
\end{proposition}

\begin{proof}
For convenience, write $A=v \vee \pi$.
Write $\init_{z_{a,b}} (I_w) = C_{z_{a,b},I_w} \cap (N_{z_{a,b},I_w}+(z_{a,b}))$.  By \cite[Lemma 2.8]{KR}, $\spec(R/C_{z_{a,b},I_w})$ is equidimensional, and $\hgt (C_{z_{a,b},I_w}) = \hgt (I_w)$.  By Corollary \ref{cor:formOfC}, $C_{z_{a,b},I_w} = I_{\pi} + I_v$, and, by Lemma~\ref{lemma:asmIdealFacts}~\ref{lempart:addIwi}, $I_{\pi} + I_v = I_{\pi \vee v} = I_A$.  Hence, by Lemma~\ref{lemma:asmIdealFacts}\ref{lempart:asmHeight} and \ref{lempart:asmEquidim}, $\hgt(A)   = \hgt (C_{z_{a,b},I_w}) = \hgt (I_w) = \ell(w)$, and $A$ is equidimensional.  

We now claim that $\Phi(w,z_{a,b}) \subseteq \perm(A)$.  Fix some $w'\in \Phi(w,z_{a,b})$.  By definition of $\Phi(w,z_{a,b})$, $\ell(w') = \ell(w)$.  It follows that, if $w' \geq A$, then $w' \in \perm(A)$.  Also by definition of $\Phi(w,z_{a,b})$, $w'>v$.  By construction, \[\rk_{w'}(a-1,b-1)=\rk_v(a-1,b-1)-1=\rk_\pi(a-1,b-1).\]  Thus, by Lemma~\ref{lemma:bigrasscompare}, $w' \geq \pi$.  Since $w' \geq \pi,v$, we know $w' \geq A=\pi\vee v$, and so $w' \in \perm(A)$.

  Conversely, fix $\widetilde{w}\in \perm(A)$.  We seek to show  $\widetilde{w}\in \Phi(w,z_{a,b})$.  Because $A$ is equidimensional, all elements of $\perm(A)$ are of length $\ell(w)$. Since $\widetilde{w}\geq A>v$, it is enough to consider covers of $v$; i.e., we may assume $\widetilde{w}=vt_{i,j}$ for some $1\leq i<j\leq n$ such that $v(i)<v(j)$ and there does not exist $k$ so that $i<k<j$ and $v(i)<v(k)<v(j)$ \cite[Lemma 2.1.4]{BB05}. 
  By Lemma \ref{lemma:bigrasscompare}, \[\rk_{\tilde{w}}(a-1,b-1)\leq \rk_\pi(a-1,b-1)  = \rk_v(a-1,b-1)-1.\]  Because $\rk_{\tilde{w}}(a-1,b-1)<\rk_v(a-1,b-1)$, we must have $i\leq a-1$ and $v(i) \leq b-1$ as well as $j\geq a$ and $v(j) \geq b$.  Suppose that $j>a$.  Because $v(a) = b$ and $v(j) \geq b$, it must also be that $v(j)>b$.  But then $i<a<j$ and $v(i)<v(a)=b<v(j)$, contradicting the assumption that $\widetilde{w}$ is a cover of $v$.  Hence, $j=a$.  
  
  Thus, by definition of $\phi(w,z_{a,b})$, we have $i\in \phi(w,z_{a,b})$, and so $\tilde{w}\in \Phi(w,z_{a,b})$.
\end{proof}

\begin{example}
Let $w=4721653$ as in Example~\ref{example:transitionDroop}.  Consider the lower outside corner $(5,5)$ of $D(w)$.  Since $\rk_w(5,5)=3$, we set $\pi \in S_7$ to be the bigrassmannian permutation with $\ess(\pi)=\{(4,4)\}$ and $\rk_\pi(4,4)=2$, i.e.,\ $\pi=1256347$.  As before, set $v=4721563 = wt_{5,6}$.  Thus, we have 
\[v\vee \pi=
\begin{pmatrix}
0&0&0&1&0&0&0\\
0&0&0&0&0&0&1\\
0&1&0&0&0&0&0\\
1&0&0&-1&1&0&0\\
0&0&0&1&0&0&0\\
0&0&0&0&0&1&0\\
0&0&1&0&0&0&0\\
\end{pmatrix}.
\]
We leave it as an exercise to verify directly that $\perm(v\vee \pi)=\Phi(w,z_{55}).$ 
\end{example}

In light of Proposition~\ref{prop:ASMdecomp}, we may interpret the transition equations using ASMs: 

\begin{corollary}  \label{cor:transitionWithPerm} Fix $w \in S_n$ and a lower outside corner $(a,b)$ of $D(w)$.
Assume $\rk_w(a,b)\geq 1$, and let $\pi$ be the bigrassmannian permutation so that $\ess(\pi)=\{(a-1,b-1)\}$ and $\rk_\pi(a-1,b-1)=\rk_w(a,b)-1$.  Let $v=wt_{a,w^{-1}(b)}$.  Then \[\mathfrak S_w(\mathbf x, \mathbf y)=(x_a-y_b)\cdot \mathfrak S_v(\mathbf x, \mathbf y)+\sum_{u\in \perm(v\vee \pi)} \mathfrak S_{u}(\mathbf x, \mathbf y).\]
\end{corollary}

\begin{proposition}
\label{prop:linkDecomp}
Fix $w \in S_n$ and a lower outside corner $(a,b)$ of $D(w)$.   
Write $\init_{z_{a,b}} (I_w) = C_{z_{a,b},I_w} \cap ( N_{z_{a,b},I_w}+(z_{a,b}))$ for the geometric vertex decomposition of $I_w$ at $z_{a,b}$. Then \[C_{z_{a,b},I_w} =  \bigcap_{u \in \Phi(w,z_{a,b})} I_{u},\] and $N_{z_{a,b},I_w} = I_v$ where $v = wt_{a, w^{-1}(b)}$.
\end{proposition}
\begin{proof}
Because the geometric vertex decomposition does not depend on the choice of $z_{a,b}$-compatible term order, we may assume, for the purposes of computing $C_{z_{a,b},I_w}$ and $N_{z_{a,b},I_w}$, that $R$ is equipped with the term order $\tau_{a,b}$ (as in Notation \ref{not:tau}).  Then, by Proposition \ref{prop:naturalGensGrobner}, the Fulton generators, which have the form $\{z_{a,b}q_1+r_1, \ldots, z_{a,b}q_k+r_k, h_1, \ldots, h_\ell\}$ where $z_{a,b}$ does not divide any $q_i, r_i$, or $h_i$, 
form a Gr\"obner basis for $I_w$ under $\tau_{a,b}$.  Hence, $C_{z_{a,b},I_w} = (q_1, \ldots, q_k, h_1, \ldots, h_\ell)$ and $N_{z_{a,b},I_w} = (h_1, \ldots, h_\ell)$.  Then, by Lemma \ref{lem:NisSchub}, $N_{z_{a,b},I_w}=I_v$ for $v = wt_{a, w^{-1}(b)}$.

We now break our argument into cases.  First suppose $\rk_w(a,b)=0$.  Then $\Phi(w,z_{a,b})=\emptyset$, and so 
$\bigcap\{ I_{u}:u \in \Phi(w,z_{a,b})\}$
is the empty intersection of ideals, which is $R$.  Also, if $\rk_w(a,b)=0$, then $z_{a,b}$ is a Fulton generator of $I_w$, and so $1 = q_i$ for some $i \in [k]$, which is to say that $C_{z_{a,b},I_w} = R$, as well.

Alternatively, suppose $\rk_w(a,b)\geq 1$.  Let $\pi$ be the bigrassmannian permutation with $\ess(\pi)=\{(a-1,b-1)\}$ and $\rk_\pi(a-1,b-1)=\rk_w(a,b)-1$.  By Corollary~\ref{cor:formOfC}, $I_{v \vee \pi}=C_{z_{a,b},I_w}$.  We know from Proposition~\ref{prop:ASMdecomp} that $\perm(v \vee \pi)=\Phi(w,z_{a,b})$. Thus, by Lemma~\ref{lemma:asmIdealFacts}\ref{lempart:asmDecomp}, \[C_{z_{a,b},I_w} =  \bigcap_{u \in \Phi(w,z_{a,b})} I_{u}. \qedhere\]
\end{proof}

We give an example to illustrate Propositions \ref{prop:naturalGensGrobner} and \ref{prop:linkDecomp} and to see the recursion on Schubert determinantal ideals they give rise to in terms of the recursion on BPDs of Lemma \ref{lem:bpdBij}.

\begin{example}\label{ex:idealRecursion}
Set $w = 214365$ and $(a,b) = (5,5)$.  Then \[
I_w = \left(z_{11}, \begin{vmatrix} z_{11} & z_{12} & z_{13} \\
z_{21} & z_{22} & z_{23} \\
z_{31} & z_{32} & z_{33} 
 \end{vmatrix}, \begin{vmatrix} z_{11} & z_{12} & z_{13} & z_{14} & z_{15} \\
 z_{21} & z_{22} & z_{23} & z_{24} & z_{25}  \\
 z_{31} & z_{32} & z_{33} & z_{34} & z_{35} \\
  z_{41} & z_{42} & z_{43} & z_{44} & z_{45}  \\
 z_{51} & z_{52} & z_{53} & z_{54} & z_{55} 
 \end{vmatrix}\right).
\]  Using Proposition \ref{prop:naturalGensGrobner}, \[
C_{y,I_w} = \left(z_{11}, \begin{vmatrix} z_{11} & z_{12} & z_{13} \\
z_{21} & z_{22} & z_{23} \\
z_{31} & z_{32} & z_{33} 
 \end{vmatrix}, \begin{vmatrix} z_{11} & z_{12} & z_{13} & z_{14} \\
 z_{21} & z_{22} & z_{23} & z_{24}\\
 z_{31} & z_{32} & z_{33} & z_{34} \\
  z_{41} & z_{42} & z_{43} & z_{44}  \\
 \end{vmatrix}\right) \mbox{ and } N_{y,I_w} = \left(z_{11}, \begin{vmatrix} z_{11} & z_{12} & z_{13} \\
z_{21} & z_{22} & z_{23} \\
z_{31} & z_{32} & z_{33} 
 \end{vmatrix}\right).
\]

The Rothe BPDs of $w$ and $v = wt_{5,6} = 214356$ are below, with the green pipes noting those whose exiting rows are exchanged, as described in Lemma \ref{lem:bpdBij}.  
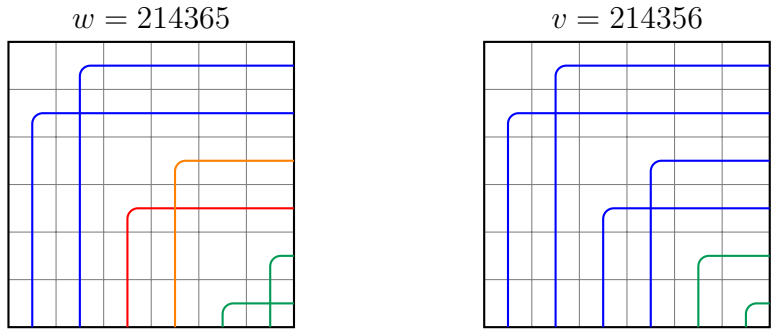
\[
\begin{tikzpicture}[x=1.5em,y=1.5em]
\node at (3, 7.5) {$w = 214365$};
\draw[step=1,gray, thin] (0,1) grid (6,7);
\draw[color=black, thick](0,1)rectangle(6,7);
\draw[thick,rounded corners, color=blue] (.5,1)--(.5,5.5)--(6,5.5);
\draw[thick,rounded corners, color=blue] (1.5,1)--(1.5,6.5)--(6,6.5);
\draw[thick,rounded corners, color = red] (2.5,1)--(2.5,3.5)--(6,3.5);
\draw[thick,rounded corners, color = orange] (3.5,1)--(3.5,4.5)--(6,4.5);
\draw[thick,rounded corners, color=ForestGreen] (4.5,1)--(4.5,1.5)--(6,1.5);
\draw[thick,rounded corners, color=ForestGreen] (5.5,1)--(5.5,2.5)--(6,2.5);
\end{tikzpicture} \hspace{2.5cm} 
\begin{tikzpicture}[x=1.5em,y=1.5em]
\node at (3, 7.5) {$v = 214356$};
\draw[step=1,gray, thin] (0,1) grid (6,7);
\draw[color=black, thick](0,1)rectangle(6,7);
\draw[thick,rounded corners, color=blue] (.5,1)--(.5,5.5)--(6,5.5);
\draw[thick,rounded corners, color=blue] (1.5,1)--(1.5,6.5)--(6,6.5);
\draw[thick,rounded corners, color=blue] (2.5,1)--(2.5,3.5)--(6,3.5);
\draw[thick,rounded corners, color=blue] (3.5,1)--(3.5,4.5)--(6,4.5);
\draw[thick,rounded corners, color=ForestGreen] (4.5,1)--(4.5,2.5)--(6,2.5);
\draw[thick,rounded corners, color=ForestGreen] (5.5,1)--(5.5,1.5)--(6,1.5);
\end{tikzpicture}
\] 
Drooping the red pipe of the Rothe BPD of $w$ into $(5,5)$ gives  $u_1 = 214536$ and drooping the orange pipe produces $u_2 = 215346$, as pictured below. 
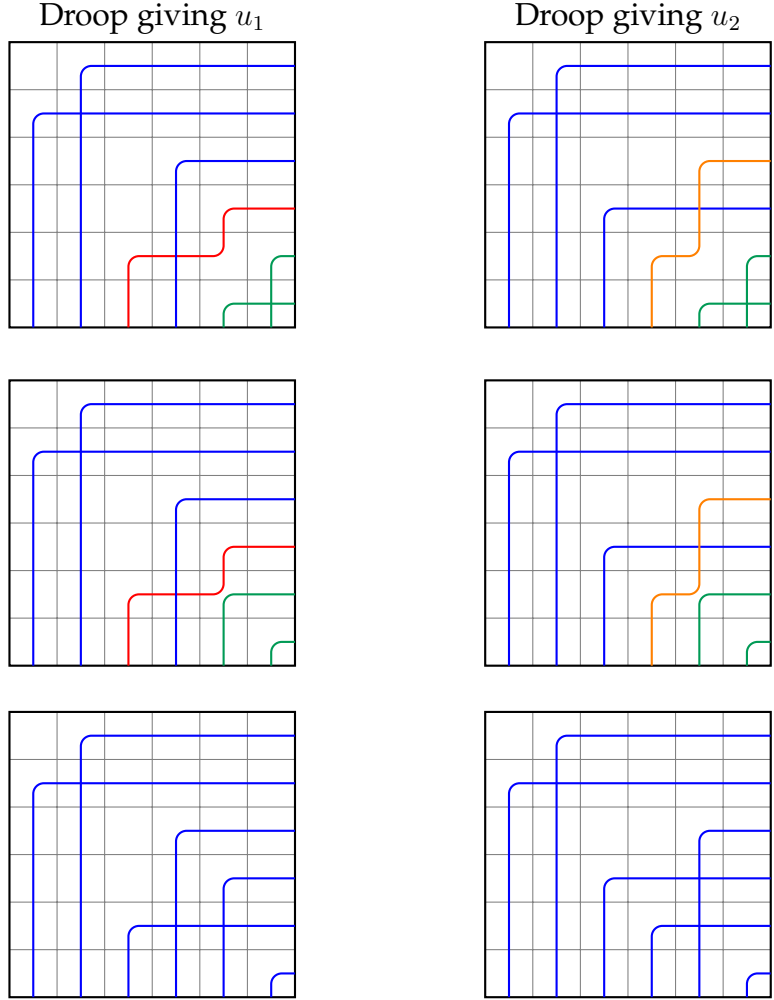
\[
\begin{tikzpicture}[x=1.5em,y=1.5em]
\node at (3, 7.5) {Droop giving $u_1$};
\draw[step=1,gray, thin] (0,1) grid (6,7);
\draw[color=black, thick](0,1)rectangle(6,7);
\draw[thick,rounded corners, color=blue] (.5,1)--(.5,5.5)--(6,5.5);
\draw[thick,rounded corners, color=blue] (1.5,1)--(1.5,6.5)--(6,6.5);
\draw[thick,rounded corners, color = red] (2.5,1)--(2.5,2.5)--(4.5,2.5)--(4.5,3.5)--(6,3.5);
\draw[thick,rounded corners, color = blue] (3.5,1)--(3.5,4.5)--(6,4.5);
\draw[thick,rounded corners, color=ForestGreen] (4.5,1)--(4.5,1.5)--(6,1.5);
\draw[thick,rounded corners, color=ForestGreen] (5.5,1)--(5.5,2.5)--(6,2.5);
\end{tikzpicture}  \hspace{2.5cm} \begin{tikzpicture}[x=1.5em,y=1.5em]
\node at (3, 7.5) {Droop giving $u_2$};
\draw[step=1,gray, thin] (0,1) grid (6,7);
\draw[color=black, thick](0,1)rectangle(6,7);
\draw[thick,rounded corners, color=blue] (.5,1)--(.5,5.5)--(6,5.5);
\draw[thick,rounded corners, color=blue] (1.5,1)--(1.5,6.5)--(6,6.5);
\draw[thick,rounded corners, color = blue] (2.5,1)--(2.5,3.5)--(6,3.5);
\draw[thick,rounded corners, color = orange] (3.5,1)--(3.5,2.5)--(4.5,2.5)--(4.5,4.5)--(6,4.5);
\draw[thick,rounded corners, color=ForestGreen] (4.5,1)--(4.5,1.5)--(6,1.5);
\draw[thick,rounded corners, color=ForestGreen] (5.5,1)--(5.5,2.5)--(6,2.5);
\end{tikzpicture}
\]   
\[
\begin{tikzpicture}[x=1.5em,y=1.5em]
\draw[step=1,gray, thin] (0,1) grid (6,7);
\draw[color=black, thick](0,1)rectangle(6,7);
\draw[thick,rounded corners, color=blue] (.5,1)--(.5,5.5)--(6,5.5);
\draw[thick,rounded corners, color=blue] (1.5,1)--(1.5,6.5)--(6,6.5);
\draw[thick,rounded corners, color = red] (2.5,1)--(2.5,2.5)--(4.5,2.5)--(4.5,3.5)--(6,3.5);
\draw[thick,rounded corners, color = blue] (3.5,1)--(3.5,4.5)--(6,4.5);
\draw[thick,rounded corners, color=ForestGreen] (4.5,1)--(4.5,2.5)--(6,2.5);
\draw[thick,rounded corners, color=ForestGreen] (5.5,1)--(5.5,1.5)--(6,1.5);
\end{tikzpicture}  \hspace{2.5cm} \begin{tikzpicture}[x=1.5em,y=1.5em]
\draw[step=1,gray, thin] (0,1) grid (6,7);
\draw[color=black, thick](0,1)rectangle(6,7);
\draw[thick,rounded corners, color=blue] (.5,1)--(.5,5.5)--(6,5.5);
\draw[thick,rounded corners, color=blue] (1.5,1)--(1.5,6.5)--(6,6.5);
\draw[thick,rounded corners, color = blue] (2.5,1)--(2.5,3.5)--(6,3.5);
\draw[thick,rounded corners, color = orange] (3.5,1)--(3.5,2.5)--(4.5,2.5)--(4.5,4.5)--(6,4.5);
\draw[thick,rounded corners, color=ForestGreen] (4.5,1)--(4.5,2.5)--(6,2.5);
\draw[thick,rounded corners, color=ForestGreen] (5.5,1)--(5.5,1.5)--(6,1.5);
\end{tikzpicture}
\]   
\[
\begin{tikzpicture}[x=1.5em,y=1.5em]
\draw[step=1,gray, thin] (0,1) grid (6,7);
\draw[color=black, thick](0,1)rectangle(6,7);
\draw[thick,rounded corners, color=blue] (.5,1)--(.5,5.5)--(6,5.5);
\draw[thick,rounded corners, color=blue] (1.5,1)--(1.5,6.5)--(6,6.5);
\draw[thick,rounded corners, color = blue] (2.5,1)--(2.5,2.5)--(6,2.5);
\draw[thick,rounded corners, color = blue] (3.5,1)--(3.5,4.5)--(6,4.5);
\draw[thick,rounded corners, color=blue] (4.5,1)--(4.5,3.5)--(6,3.5);
\draw[thick,rounded corners, color=blue] (5.5,1)--(5.5,1.5)--(6,1.5);
\end{tikzpicture}  \hspace{2.5cm} \begin{tikzpicture}[x=1.5em,y=1.5em]
\draw[step=1,gray, thin] (0,1) grid (6,7);
\draw[color=black, thick](0,1)rectangle(6,7);
\draw[thick,rounded corners, color=blue] (.5,1)--(.5,5.5)--(6,5.5);
\draw[thick,rounded corners, color=blue] (1.5,1)--(1.5,6.5)--(6,6.5);
\draw[thick,rounded corners, color = blue] (2.5,1)--(2.5,3.5)--(6,3.5);
\draw[thick,rounded corners, color = blue] (3.5,1)--(3.5,2.5)--(6,2.5);
\draw[thick,rounded corners, color=blue] (4.5,1)--(4.5,4.5)--(6,4.5);
\draw[thick,rounded corners, color=blue] (5.5,1)--(5.5,1.5)--(6,1.5);
\end{tikzpicture}
\]   
We can determine the Fulton generators of the $I_{u_i}$ from the corresponding drooped BPDs of $w$, pictured above, and verify the equality $C_{z_{55},I_w} = I_{u_1} \cap I_{u_2}$.  
Here $\Phi(w,z_{55}) = \{u_1,u_2\}$, and \[
C_{y,I_w} = I_{u_1} \cap I_{u_2} =  \left((z_{11})+ I_3\begin{pmatrix} z_{11} & z_{12} & z_{13} \\
z_{21} & z_{22} & z_{23} \\
z_{31} & z_{32} & z_{33} \\
z_{41} & z_{42} & z_{43} 
 \end{pmatrix}
 \right) \cap \left((z_{11})+ I_3\begin{pmatrix} z_{11} & z_{12} & z_{13} & z_{14} \\
z_{21} & z_{22} & z_{23} & z_{24} \\
z_{31} & z_{32} & z_{33} & z_{34}
 \end{pmatrix}
 \right).
\]

In terms of the transition equations of Theorem \ref{thm:transition} (and Corollary \ref{cor:transitionWithPerm}), we have
\begin{align*}\mathfrak S_w(\mathbf x, \mathbf y) = (x_5-y_5)\cdot \mathfrak S_v(\mathbf x, \mathbf y)+\mathfrak S_{u_1}(\mathbf x, \mathbf y)+\mathfrak S_{u_2}(\mathbf x, \mathbf y). &  \qedhere
\end{align*} 
\end{example}

For $w \in S_n$ and $(a,b)$ a lower outside corner of $D(w)$, we conclude this section with an explicit description of $\init_{\tau_{a,b}}(I_w)$ in terms of pipe dreams (where $\tau_{a,b}$ is as in Notation \ref{not:tau}).  Recall that, for each pipe dream $\mathcal{D} \in \pipes(w)$, $C(\mathcal D) \subseteq [n] \times [n]$ denotes the set of crossing tiles of $\mathcal D$ and that $I_{C(\mathcal D)} = (z_{i,j} : (i,j) \in C(\mathcal D))$.

\begin{proposition}
Fix $w \in S_n$ and a lower outside corner $(a,b)$ of $D(w)$.  As before, let $v = wt_{a,w^{-1}(b)}$ and $\tau_{a,b}$ be the term order of Notation~\ref{not:tau}. 
Then \[\init_{\tau_{a,b}}(I_w)= \left ( \displaystyle \bigcap_{u\in \Phi(w,z_{a,b})}\left(\bigcap_{\mathcal D \in \pipes(u)} I_{C(\mathcal D)} \right)\right )\cap\left(\bigcap_{\mathcal D\in \pipes(v)} ( I_{C(\mathcal D)}+(z_{a,b}))\right).\]
\end{proposition}

\begin{proof}
Let $\sigma$ be any anti-diagonal term order on $R$.  First assume $\rk_w(a,b)=0$, in which case $z_{a,b} \in I_w$.  Then $\init_{z_{a,b}}(I_w)= \init_{z_{a,b}}(I_v+(z_{a,b}))= I_v+(z_{a,b})$ because $z_{a,b}$ does not divide any term of any Fulton generator of $I_v$.   Because $\tau_{a,b}$ is $z_{a,b}$-compatible, \[
\init_{\tau_{a,b}}(I_w) = \init_{\tau_{a,b}}(\init_{z_{a,b}}(I_w)) = \init_{\tau_{a,b}}(I_v+(z_{a,b})) = \init_{\tau_{a,b}}(I_v)+(z_{a,b}).
\]
 Hence \[\init_{\tau_{a,b}}(I_w)=\init_{\tau_{a,b}}(I_v)+(z_{a,b})=\init_\sigma(I_v)+(z_{a,b})=\bigcap_{\mathcal D\in \pipes(v)} ( I_{C(\mathcal D)}+(z_{a,b})).\]  Furthermore, $\Phi(w,z_{a,b})=\emptyset$, and so  we have the empty intersection of ideals \[\displaystyle \bigcap_{u\in \Phi(w,z_{a,b})}\left(\bigcap_{\mathcal D \in \pipes(u)} I_{C(\mathcal D)} \right),\] which is (by convention) $R$.  So the statement holds.

Now suppose $\rk_w(a,b)\geq 1$.  Let $A=v\vee \pi$, where $\pi$ is the bigrassmannian permutation with essential cell $(a-1,b-1)$ and $\rk_\pi(a-1,b-1)=\rk_w(a,b)-1$.  Then by Corollary \ref{cor:formOfC}
\[\init_{z_{a,b}}(I_w)= I_A
\cap(I_v+(z_{a,b})).\]  

If $\{z_{a,b}q_1+r_1, \ldots, z_{a,b}q_k+r_k, h_1, \ldots, h_\ell\}$ are the Fulton generators of $I_w$, which 
form a Gr\"obner basis for $I_w$ under the $z_{a,b}$-compatible term order $\tau_{a,b}$ by Proposition \ref{prop:naturalGensGrobner}, then, by \cite[Theorem 2.1(a)]{KMY09}, $\{q_1, \ldots, q_k, h_1, \ldots, h_\ell\}$ is a Gr\"obner basis for $C_{z_{a,b},I_w}=I_A$ (using Corollary \ref{cor:formOfC}), and $\{h_1, \ldots, h_\ell\}$ is a Gr\"obner basis for $N_{z_{a,b},I_w}=I_v$ (using Proposition \ref{prop:linkDecomp}).  Hence, \begin{align*}
\init_{\tau_{a,b}}(I_w) &= (\init_{\tau_{a,b}}(z_{a,b}q_1), \ldots, \init_{\tau_{a,b}}(z_{a,b}q_k), \init_{\tau_{a,b}}(h_1), \ldots, \init_{\tau_{a,b}}(h_\ell))\\
& = (\init_{\tau_{a,b}}(I_A) \cap (z_{a,b}) )+\init_{\tau_{a,b}}(I_v)\\
& = \init_{\tau_{a,b}}(I_A) \cap (\init_{\tau_{a,b}}(I_v)+(z_{a,b})),
\end{align*} where the final equality holds because 
$\init_{\tau_{a,b}}(I_v) \subseteq \init_{\tau_{a,b}}(I_A)$.

Because both $I_v$ and $I_A$ have generating sets that do not involve $z_{a,b}$, $\init_{\tau_{a,b}}(I_v) = \init_\sigma(I_v)$ and $\init_{\tau_{a,b}}(I_A) = \init_\sigma(I_A)$.  Putting these equalities together,
\begin{align*}
    \init_{\tau_{a,b}}(I_w)&=\init_{\tau_{a,b}}(I_A) \cap \init_{\tau_{a,b}}(I_v+(z_{a,b})) \\
    &=\init_\sigma(I_A) \cap \init_\sigma((I_v)+(z_{a,b})) \\
    &=\left(\bigcap_{u\in \Phi(w,z_{a,b})} \init_\sigma(I_u) \right) \cap (\init_\sigma(I_v)+(z_{a,b})) & \text{(by Lemma~\ref{lemma:asmIdealFacts}\ref{lempart:intersectInit} and Proposition~\ref{prop:ASMdecomp}).}
    \end{align*}
    Thus, applying Theorem~\ref{theorem:KMTheoremB} yields
    \begin{align*}
       \init_{ \tau_{a,b}}(I_w)=\left (  \bigcap_{u\in \Phi(w,z_{a,b})}\left(\bigcap_{\mathcal D \in \pipes(u)} I_{C(\mathcal D)} \right)\right )\cap \bigcap_{\mathcal D\in \pipes(v)} ( I_{C(\mathcal D)}+(z_{a,b})) .&\qedhere 
    \end{align*}
\end{proof}

\section{The main result}
\label{section:proofs}

The main goal of this section is to prove Theorem \ref{thm:main}.  We also obtain, as a consequence of the proof of this theorem, the Cohen--Macaulayness of certain families of  equidimensional unions of matrix Schubert varieties.  As in previous sections, we will take $R = \kappa[z_{1,1}, \ldots, z_{n,n}]$ and assume that all ideals are ideals of $R$ unless otherwise stated.  

\subsection{Proof of main result}
\label{subsect:mainProofs}

We will use $\min(I)$ to denote the set of minimal primes of an ideal $I$ and $\ell(M)$ to denote the length of a finite length $R$-module $M$.  Suppose $P$ is a minimal prime of $I$, in which case $\spec(R/P)$ is an irreducible component of $\spec(R/I)$.  Recall that the multiplicity of $\spec(R/P)$ along $\spec(R/I)$ is defined to be the length $\ell(R_P/IR_P)$ (equivalently, $\ell((R/I)_P)$) and is denoted $\mult_P(R/I)$.

We will use geometric vertex decomposition to develop a recurrence on unions of matrix Schubert varieties that mirrors the recurrence on bumpless pipe dreams from Section \ref{sect:BPDandTransition}.  Before beginning the proofs in this section, we give an example to illustrate the structure of the induction. 

\begin{example}
Assume $w = 214365$ and $(a,b) = (5,5)$ as in Example \ref{ex:idealRecursion}.   Keeping notation from that example for $u_1$, $u_2$, and $v$, we have seen that the equality of ideals \[
\init_{z_{55}}(I_w) = (I_{u_1} \cap I_{u_2}) \cap (I_{v}+(z_{55}))
\] is reflected in a bijection of the BPDs of $w$ with the union of the BPDs of $u_1$, $u_2$, and $v$.

 Fix a $z_{55}$-compatible term order $\sigma$. The minimal primes of $\init_\sigma (I_w)$ that contain $z_{55}$ are of the form $Q+(z_{55})$ for a minimal prime $Q$ of $\init_\sigma(I_v)$, and the minimal primes of $\init_\sigma (I_w)$ that do not contain $z_{55}$ are the minimal primes of $\init_\sigma(C_{z_{55},I_w}) = \init_\sigma(I_{u_1} \cap I_{u_2})$, which are in turn the union of the minimal primes of $\init_\sigma(I_{u_1})$ and $\init_\sigma(I_{u_2})$.  If $P = (z_{11},z_{12},z_{21})$, which is a minimal prime both of $\init_\sigma(I_{u_1})$ and of $\init_\sigma(I_{u_2})$, then $\mult_P(R/\init_\sigma(I_w)) = 2=\mult_P(R/\init_\sigma(C_{z_{55},I_w}))$.  There are no other primes which are minimal over both $\init_\sigma(I_{u_1})$ and $\init_\sigma(I_{u_2})$.  The multiplicity of every other irreducible component along $\spec(R/\init_\sigma(I_w))$ is $1$.
\end{example}

The following lemma will facilitate the inductive argument in our main theorem.  A geometric vertex decomposition of $I_w$ at a lower outside corner $D(w)$ will allow us to understand $\init_\sigma(I_w)$ in terms of an equidimensional union of matrix Schubert varieties, each of which is in an appropriate sense simpler than $X_w$.  The following lemma will allow us to understand that union by studying each component individually.

Because we are interested both in full, that is, monomial, initial ideals and also in ideals arising from geometric vertex decomposition, we will prefer to study degenerations determined by weight orders in addition to those determined by monomial orders.  For a general background on weight orders, we refer the reader to \cite[Chapter 15]{Eis95}. To construct a weight order, one assigns to each variable $z_{i,j} \in R$ an integer $\lambda_{i,j}$.   Let $\lambda$ be the vector of the $\lambda_{i,j}$. For exponents $a_{i,j}$, the weight of a monomial $\prod \{z_{i,j}^{a_{i,j}}: (i,j) \in [n] \times [n]\}$ is $\sum \{a_{i,j}\lambda_{i,j}: (i,j) \in [n] \times [n]\}$.  The initial term of a polynomial $f$, denoted $\init_{\lambda}(f)$ is the sum of the terms of $f$ of highest weight.  If $I$ is an ideal, then $\init_\lambda(I) = (\init_\lambda(f): f \in I)$. For any finite collection of ideals $I_1, \ldots, I_k$ and any term order $\sigma$, there exists a weight vector $\lambda$ so that $\init_{\sigma}(I_\ell) = \init_\lambda(I_\ell)$ for all $\ell \in [k]$.  A geometric vertex decomposition at $z_{i,j}$ corresponds to the weight vector with $\lambda_{i,j} = 1$ and all other entries equal to $0$.

\begin{lemma}\label{lem:primaryMult} Let $J$ be an ideal of $R$ defining an equidimensional scheme with $\min(J) = \{P_1, \ldots, P_r\}$. Let $\lambda$ be any weight order.  For all $\mathcal{P} \in \min(\init_\lambda(J))$, 
\[
\mult_\mathcal{P}(R/\init_\lambda(J)) = \displaystyle \sum_{i=1}^r  \mult_{P_i}(R/J) \cdot \mult_{\mathcal{P}}(R/\init_\lambda(P_i) ).
\]
\end{lemma}

We thank Matt Larson for pointing us to Lemma 0H4J in the Stacks Project \cite{stacks-project} and for suggesting this proof.  We refer the reader to the Stacks Project for background on cycles.

\begin{proof}
Let $X = \spec(R/J)$, and let $d = \dim(X)$. Let $X_\init = \spec(R/\init_\lambda(J))$, which is then also of dimension $d$.  View the Gr\"obner degeneration determined by $\lambda$ as a flat family over the discrete valuation ring $\kappa[t]_{(t)}$, whose generic fiber we identify with $X$ and whose special fiber we identify with $X_\init$.  

For each $i \in [r]$, let $X_i = \spec(R/P_i)$, $(X_i)_\init = \spec(R/\init_\lambda(P_i))$, and $X_\mathcal{P} = \spec(R/\mathcal{P})$ for each $\mathcal{P} \in \min(\init_\lambda(J))$.  Note that $\min(\init_\lambda(J)) = \bigcup \{\min(\init_\lambda(P_i)) :i \in [r] \}$, and that all such primes are of height $\hgt(J)$.  

Let $Z_d(X)$ (respectively, $Z_d(X_\init)$) denote the group of cycles of dimension $d$ on $X$, which is free abelian on the classes of the integral closed subschemes of $X$ (respectively, of $X_\init$) of dimension $d$.  For a closed subscheme $Y$ of $X$ or of $X_\init$, we will let $[Y]$ denote its image in the appropriate group of cycles.

Now \begin{equation}\label{bigCycles}
[X] = \sum_{i \in [r]} \mult_{P_i}(R/J) [X_i],
\end{equation}
 \begin{equation}\label{smallCycles}
 [X_\init] = \sum_{\mathcal{P} \in \min(R/\init_\lambda(J))} \mult_{\mathcal{P}}(R/\init_\lambda(J)) [X_\mathcal{P}],
\end{equation}
and \begin{equation}\label{cycles-components}
 [(X_i)_\init] = \sum_{\mathcal{P} \in \min(R/\init_\lambda(J))} \mult_{\mathcal{P}}(R/\init_\lambda(P_i)) [X_\mathcal{P}].
\end{equation}

By Lemma 0H4J in the Stacks Project \cite{stacks-project}, there is a group homomorphism taking $[X]$ to $[X_\init]$ and each $[X_i]$ to $[(X_i)_\init]$.  Applying this map to both sides of Equation \ref{bigCycles} and making substitutions using Equations \ref{smallCycles} and \ref{cycles-components}, we obtain
 \begin{equation}\label{total-mults}
\sum_{\mathcal{P} \in \min(R/\init_\lambda(J))} \mult_{\mathcal{P}}(R/\init_\lambda(J)) [X_\mathcal{P}] = \sum_{i \in [r]} \mult_{P_i}(R/J) \sum_{\mathcal{P} \in \min(R/\init_\lambda(J))} \mult_{\mathcal{P}}(R/\init_\lambda(P_i)) [X_\mathcal{P}].
\end{equation}

Because the $[X_{\mathcal{P}}]$ are among the free generators of $Z_d(X_\init)$, the coefficients of each $[X_\mathcal{P}]$ on the left and right hand sides of Equation \ref{total-mults} must agree, yielding the desired result.
\end{proof}

To facilitate an inductive argument in the main theorem of this paper, we define a relation on $S_n$.  The purpose of this relation is to describe a manner in which, for $(a,b)$ a lower outside corner of $D(w)$, the elements of $\Phi(w,z_{a,b})$ and $wt_{a,w^{-1}(b)}$ are all appropriately understood to be smaller than $w$.  This will allow a geometric vertex decomposition of $I_w$ at $z_{a,b}$ to be the key step in our inductive argument.

\begin{notation}\label{not:indOrder} If $u,w \in S_n$, we will say that $u \prec w$ if, \begin{enumerate} 
\item for every lower outside corner $(\alpha,\beta)$ of $D(u)$, there exists some $(\gamma,\delta)$ lower outside corner of $D(w)$ so that $(\alpha,\beta) \leq (\gamma,\delta)$ (i.e., $\alpha \leq \gamma$ and $\beta \leq \delta$), and 
\item there exists some lower outside corner $(\epsilon, \zeta)$ of $D(w)$ so that there does \emph{not} exist any $(\alpha,\beta) \in D(u)$ satisfying $(\epsilon, \zeta) \leq (\alpha, \beta)$.
\end{enumerate} Note that the relation $\prec$ is transitive.    
\end{notation}

\begin{example}
Set $w = 52143$, $u_1 = 52314$, and $u_2 = 53124$.  Then $u_1, u_2 \prec w$.  To see this, note that  $(2,2),(3,1) \leq (4,3)$ (corresponding to the light gray diagram boxes in the Rothe diagrams below) and $(1,4) \leq (1,4)$ (corresponding to the dark gray diagram boxes in the Rothe diagrams below).
\[
\begin{tikzpicture}[x=1.5em,y=1.5em]
\node at (2.5,5.5) {{$u_1 = 52314$}};
\draw[step=1,gray, thin] (0,0) grid (5,5);
\draw[color=black, thick](0,0)rectangle(5,5);
\filldraw [black](4.5,4.5)circle(.1);
\filldraw [black](1.5,3.5)circle(.1);
\filldraw [black](2.5,2.5)circle(.1);
\filldraw [black](.5,1.5)circle(.1);
\filldraw [black](3.5,.5)circle(.1);
\draw[thick, color=blue] (.5,0)--(.5,1.5)--(5,1.5);
\draw[thick, color=blue] (1.5,0)--(1.5,3.5)--(5,3.5);
\draw[thick, color=blue] (2.5,0)--(2.5,2.5)--(5,2.5);
\draw[thick, color=blue] (3.5,0)--(3.5,.5)--(5,.5);
\draw[thick, color=blue] (4.5,0)--(4.5,4.5)--(5,4.5);
\filldraw[color=black, fill=gray](3,4)rectangle(4,5);
\filldraw[color=black, fill=gray!30](0,2)rectangle(1,3);
\end{tikzpicture} \hspace{2cm} \begin{tikzpicture}[x=1.5em,y=1.5em]
\node at (2.5,5.5) {{$u_2 = 53124$}};
\draw[step=1,gray, thin] (0,0) grid (5,5);
\draw[color=black, thick](0,0)rectangle(5,5);
\filldraw [black](4.5,4.5)circle(.1);
\filldraw [black](2.5,3.5)circle(.1);
\filldraw [black](.5,2.5)circle(.1);
\filldraw [black](1.5,1.5)circle(.1);
\filldraw [black](3.5,.5)circle(.1);
\draw[thick, color=blue] (.5,0)--(.5,2.5)--(5,2.5);
\draw[thick, color=blue] (1.5,0)--(1.5,1.5)--(5,1.5);
\draw[thick, color=blue] (2.5,0)--(2.5,3.5)--(5,3.5);
\draw[thick, color=blue] (3.5,0)--(3.5,.5)--(5,.5);
\draw[thick, color=blue] (4.5,0)--(4.5,4.5)--(5,4.5);
\filldraw[color=black, fill=gray](3,4)rectangle(4,5);
\filldraw[color=black, fill=gray!30](1,3)rectangle(2,4);
\end{tikzpicture} \hspace{2cm} \begin{tikzpicture}[x=1.5em,y=1.5em]
\node at (2.5,5.5) {{$w = 52143$}};
\draw[step=1,gray, thin] (0,0) grid (5,5);
\draw[color=black, thick](0,0)rectangle(5,5);
\filldraw [black](4.5,4.5)circle(.1);
\filldraw [black](1.5,3.5)circle(.1);
\filldraw [black](.5,2.5)circle(.1);
\filldraw [black](3.5,1.5)circle(.1);
\filldraw [black](2.5,.5)circle(.1);
\draw[thick, color=blue] (.5,0)--(.5,2.5)--(5,2.5);
\draw[thick, color=blue] (1.5,0)--(1.5,3.5)--(5,3.5);
\draw[thick, color=blue] (2.5,0)--(2.5,.5)--(5,.5);
\draw[thick, color=blue] (3.5,0)--(3.5,1.5)--(5,1.5);
\draw[thick, color=blue] (4.5,0)--(4.5,4.5)--(5,4.5);
\filldraw[color=black, fill=gray](3,4)rectangle(4,5);
\filldraw[color=black, fill=gray!30](2,1)rectangle(3,2);
\end{tikzpicture}
\]
Meanwhile, there does not exist any $(\alpha, \beta)$ in either $D(u_i)$ so that $(4,3) \leq (\alpha,\beta)$. We remark that $\{u_1,u_2\} = \Phi(w, z_{43})$. 
\end{example}

\begin{lemma}
\label{lem:indSmaller}
Fix $w \in S_n$, and suppose  that $(a,b)$ is a lower outside corner of $D(w)$. Then $u \prec w$ for each $u \in \Phi(w,z_{a,b}) \cup \{wt_{a,w^{-1}(b)}\}$.
\end{lemma}
\begin{proof}
Condition (1) in the definition of $\prec$ follows from the bijection from the proof of Lemma~\ref{lem:bpdBij} and the definition of lower outside corner. (Alternatively, one may phrase this argument in terms of the ``marching" operation of \cite{KY04}.) Condition (2) is satisfied by the lower outside corner $(a,b)$ of $D(w)$.\end{proof}

We now define the family of term orders that will appear in this paper's main theorem.  

\begin{defn}\label{def:lexFromSE}
Fix permutations $w_1, \ldots, w_r \in S_n$.  We call a lexicographic term order $\sigma$ \mydef{lexicographic from southeast} with respect to $\{w_1, \ldots, w_r\}$ if $\{w_1, \ldots, w_r\}$ is the singleton set consisting of the identity permutation or if the following conditions hold on the lexicographically largest variable $z_{a,b}$ involved in any Fulton generator of any $w_i$:
 \begin{enumerate}
    \item For each $i \in [r]$, if $z_{a,b}$ is involved in a Fulton generator of $I_{w_i}$, then $(a,b)$ is a lower outside corner of $D(w_i)$, and
    \item $\sigma$ is lexicographic from southeast with respect to \begin{align*}
   & \{w_i t_{a,w_i^{-1}(b)} : i \in [r], (a,b) \mbox{ is a lower outside corner of } D(w_i)\} \\
   \cup  &\{u \in \Phi(w_i,z_{a,b}) : i \in [r], (a,b) \mbox{ is a lower outside corner of } D(w_i)\} \\
   \cup &\{w_i : i \in [r], (a,b) \mbox{ is not a lower outside corner of } D(w_i)\}.
    \end{align*}
\end{enumerate}
\end{defn}

We consider two examples of term orders that are lexicographic from southeast with respect to any subset of $S_n$.  We will return to these examples several times. Let $\sigma$ be the lexicographic order on the variables ordered starting from $z_{n,n}$ and progressing up column $n$, then from $z_{n, n-1}$ up column $n-1$ and so on.  Similarly, let $\sigma'$ be the lexicographic order on the variables ordered starting from $z_{n,n}$ and progressing left along row $n$, then from $z_{n-1, n}$ left along row $n-1$ and so on.   Note that both $\sigma$ and $\sigma'$, in addition to being lexicographic from southeast, are diagonal term orders.

 Recall that, to a subset $E$ of $[n] \times [n]$, we associate the ideal $I_E = (z_{i,j} : (i,j) \in E)$.  In particular, if $\mathcal{B}$ is a bumpless pipe dream, then $I_{D(\mathcal{B})}$ is the ideal generated by the $z_{i,j}$ where $(i,j)$ is a blank tile of $\mathcal{B}$.

We will now prove the main result of this paper. The case $r=1$ of its immediate corollary (Corollary \ref{cor:mainCor}) was conjectured by Hamaker, Pechenik, and Weigandt in \cite[Conjecture 1.2]{HPW}.

\begin{theorem}
\label{thm:main}
Fix distinct permutations $w_1, \ldots, w_r \in S_n$ of the same length and a term order $\sigma$ that is lexicographic from southeast with respect to $\{w_i:i\in[r]\}$. Set $J = \bigcap \{I_{w_i}:i\in[r]\}$.  Then the irreducible components of $\spec(R/\init_\sigma(J))$, counted with multiplicity, are indexed by 
$ \bigcup\{\bpd(w_i):i\in[r]\}$.  
Precisely, the multiplicity of $\spec(R/P)$ along $\spec(R/\init_\sigma(J))$ is \[
\#\left\{\mathcal{B} \in \bpd(w_1)\cup \cdots \cup \bpd(w_r) : P = I_{D(\mathcal{B})} \right\}.
\]
\end{theorem}

\begin{proof}
By Lemma \ref{lem:primaryMult}, it suffices to consider $r=1$.  Set $w = w_1$.  We will proceed by induction on $\prec$.  The unique minimal element of $S_n$ under $\prec$ is the identity permutation, for which the claim is clear.

Now take an arbitrary non-identity $w \in S_n$.  Let $z_{a,b}$ be the lexicographically largest variable involved in a Fulton generator of $I_w$.  Then $(a,b)$ is a lower outside corner of $D(w)$.  

Set $v = wt_{a,w^{-1}(b)}$.   By Proposition \ref{prop:linkDecomp}, $C_{z_{a,b},I_w} = \bigcap\{I_u : u \in \Phi(w,z_{a,b})\}$ and $N_{z_{a,b},I_w} = I_v$.  By Lemma \ref{lem:indSmaller}, $u \prec w$ for each $u \in \Phi(w,z_{a,b})$ and $v\prec w$.   Fix $P \in \min(\init_\sigma(I_w))$.  Now $\init_\sigma(I_w) = \init_\sigma(\init_{z_{a,b}}(I_w))$ and $\init_{z_{a,b}}(I_w) = C_{z_{a,b},I_w} \cap (N_{z_{a,b},I_w}+(z_{a,b}))$.

Hence, by  Lemma \ref{lem:primaryMult} and the inductive hypothesis  \begin{align*}
\mult_P(\init_\sigma(I_w)) &= \mult_P(\init_\sigma(\init_{z_{a,b}}(I_w)))\\
& = \sum_{u \in \Phi(w,z_{a,b})} \mult_P(\init_\sigma(I_u)) + \mult_P(\init_\sigma(I_v)+(z_{a,b}))\\
 = &\#\left\{\mathcal{B} \in \bigcup_{u \in \Phi(w,z_{a,b})} \bpd(u): P = I_{D(\mathcal{B})}\right\}+\#\left\{\mathcal{B} \in \bpd(v): P = I_{D(\mathcal{B} \cup \{(a,b)\})}\right\}.
\end{align*}  
Finally, by Lemma \ref{lem:bpdBij}, \begin{align*}
&\#\left\{\mathcal{B} \in \bigcup_{u \in \Phi(w,z_{a,b})} \bpd(u): P = I_{D(\mathcal{B})}\right\}+\#\left\{\mathcal{B} \in \bpd(v): P = I_{D(\mathcal{B} \cup \{(a,b)\})}\right\}\\
 = & \#\left\{\mathcal{B} \in \bpd(w) : P = I_{D(\mathcal{B})} \right\},
\end{align*} and so $\mult_P(\init_\sigma(I_w)) = \#\left\{\mathcal{B} \in \bpd(w) : P = I_{D(\mathcal{B})} \right\}$, as desired.
\end{proof}

Because the term order $\sigma$ in Theorem \ref{thm:main} is a lexicographic term order, we may understand $\init_\sigma(J)$ via a sequence of geometric vertex decompositions, in the manner that the induction in the proof of Theorem \ref{thm:main} carries out.  When $X_{\init} = \spec(R/\init_\sigma(J))$ fails to be generically reduced, this procedure allows us to pinpoint when in the degeneration from $X = \spec(R/J)$ to $X_{\init}$ this failure emerges. Using Lemma \ref{lem:primaryMult}, for each $I_u \in \min(\init_{z_{a,b}}(J))$, \[
\mult_{I_u}(R/\init_{z_{a,b}}(J)) = \sum_{i \in [r]} \mult_{I_u}(R/\init_{z_{a,b}}(I_{w_i})).
\]

Thus, using Proposition \ref{prop:linkDecomp}, $\spec(R/\init_{z_{a,b}}(J))$ fails to be generically reduced if and only if there is some $i \in [r]$ so that $(a,b)$ is a lower outside corner of $D(w_i)$ and some $j \in [r]$ so that $w_j \in \Phi(w_i,z_{a,b})$ or some $i' \in [r]$, $i \neq i'$, so that $(a,b)$ is also a lower outside corner of $D(w_{i'})$ and $\Phi(w_i,z_{a,b}) \cap \Phi(w_{i'},z_{a,b}) \neq \emptyset$.
  
The example below illustrates Theorem \ref{thm:main} for a union of two matrix Schubert varieties of codimension one.

\begin{example}
Let $w_1=213$ and $w_2=132$.  
Then $I_{w_1}=\left(z_{11}\right)$ and $I_{w_2}=\left(\begin{vmatrix} z_{11} & z_{12} \\
z_{21} & z_{22} 
\end{vmatrix} \right)$.  Let $J=I_{w_1}\cap I_{w_2}=(z_{11}^2z_{22} - z_{11}z_{12}z_{21})$.  With respect to any lexicographic from southeast term order (or indeed any diagonal term order $\sigma$), we have $\init_\sigma(J)=(z_{11}^2z_{22}),$ which has primary decomposition $(z_{11}^2)\cap (z_{22})$. As such, $\mult_{I_{\{(1,1)\}}}(S/\init_\sigma(J))=2$, and $\mult_{I_{\{(2,2)\}}}(R/\init_\sigma(J))=1$.

Note that \[\bpd(w_1)=\left\{	
\raisebox{-2em}{\begin{tikzpicture}[x=1.5em,y=1.5em]
	\draw[step=1,gray, thin] (0,0) grid (3,3);
	\draw[color=black, thick](0,0)rectangle(3,3);
	\draw[thick,rounded corners,color=blue] (.5,0)--(.5,1.5)--(3,1.5);
	\draw[thick,rounded corners,color=blue] (1.5,0)--(1.5,2.5)--(3,2.5);
	\draw[thick,rounded corners,color=blue] (2.5,0)--(2.5,.5)--(3,.5);
	\end{tikzpicture}}
\right\} \hspace{1em} \text{ and } \hspace{1em} \bpd(w_2)=\left\{\raisebox{-2em}{	\begin{tikzpicture}[x=1.5em,y=1.5em]
	\draw[step=1,gray, thin] (0,0) grid (3,3);
	\draw[color=black, thick](0,0)rectangle(3,3);
	\draw[thick,rounded corners,color=blue] (.5,0)--(.5,2.5)--(3,2.5);
	\draw[thick,rounded corners,color=blue] (1.5,0)--(1.5,.5)--(3,.5);
	\draw[thick,rounded corners,color=blue] (2.5,0)--(2.5,1.5)--(3,1.5);
	\end{tikzpicture}
\, , \,
	\begin{tikzpicture}[x=1.5em,y=1.5em]
	\draw[step=1,gray, thin] (0,0) grid (3,3);
	\draw[color=black, thick](0,0)rectangle(3,3);
	\draw[thick,rounded corners,color=blue] (.5,0)--(.5,1.5)--(1.5,1.5)--(1.5,2.5)--(3,2.5);
\draw[thick,rounded corners,color=blue] (1.5,0)--(1.5,.5)--(3,.5);
\draw[thick,rounded corners,color=blue] (2.5,0)--(2.5,1.5)--(3,1.5);
	\end{tikzpicture}
}\right\}.\]
Thus, we have two BPDs in $\bpd(w_1)\cup\bpd(w_2)$ that correspond to the prime $I_{\{(1,1)\}} = (z_{11})$ and one that corresponds to $I_{\{(2,2)\}} = (z_{22})$, as predicted by Theorem~\ref{thm:main}.  Observe that $w_1 \in \Phi(w_2,z_{22})$.
\end{example}

Recall that, given $w \in S_n$, Knutson and Miller \cite[Theorem A]{KM05} showed that the $\mathbb{Z}^{2n}$-graded multidegree of $R/I_w$ is $\mathfrak S_w(\mathbf x, -\mathbf y)$. With notation and assumptions from Theorem \ref{thm:main}, we remark that Theorem \ref{thm:main} shows that the $\mathbb{Z}^{2n}$-graded multidegree of $R/\init_\sigma(I_w)$, hence also of $R/I_w$, is $\sum_{\mathcal B\in \bpd(w)}\wt(\mathcal B)(\mathbf x, -\mathbf y)$.  That is, combining \cite[Theorem A]{KM05} and Theorem \ref{thm:main} recovers Lam, Lee, and Shimozono's bumpless pipe dream formula for double Schubert polynomials.  Alternatively, one may recover the formula using the fact that both Schubert polynomials and also appropriate droop moves on bumpless pipe dreams satisfy the transition equations, as described in Corollary \ref{cor:transitionWithPerm} (see also the proof of Theorem~\ref{thm:transition} in Section \ref{section:additionalConsequences}).

\begin{corollary}\label{cor:mainCor}
Suppose that $w_1, \ldots, w_r \in S_n$ are permutations of the same length, and set $J=\bigcap \{I_{w_i}:i\in[r]\}$. There exist diagonal term orders so that the irreducible components of $\spec(R/\init_\sigma(J))$, counted with multiplicity, are indexed by 
$ \bigcup\{\bpd(w_i):i\in[r]\}$.  
Precisely, the multiplicity of $\spec(R/P)$ along $\spec(R/\init_\sigma(J))$ is \[
\#\left\{\mathcal{B} \in \bpd(w_1)\cup \cdots \cup \bpd(w_r) : P = I_{D(\mathcal{B})} \right\}.
\]
\end{corollary}
\begin{proof}
Let $\sigma$ be the lexicographic order on the variables ordered starting from $z_{n,n}$ and progressing up column $n$, then from $z_{n, n-1}$ up column $n-1$ and so on.  Similarly, let $\sigma'$ be the lexicographic order on the variables ordered starting from $z_{n,n}$ and progressing left along row $n$, then from $z_{n-1, n}$ left along row $n-1$ and so on.  Recall that $\sigma$ and $\sigma'$ are both diagonal term orders that are also lexicographic from southeast with respect to $\{w_1, \ldots, w_r\}$.  Considering either of those term orders, the statement is immediate from Theorem \ref{thm:main}.
\end{proof}

With notation and assumptions as in Theorem \ref{thm:main}, the situation is especially nice when $\spec(R/\init_\sigma(J))$ is reduced.  In that case, for all $D\subseteq [n]\times [n]$, \[\#\{\mathcal B\in \displaystyle\bigcup_{i\in[r]}\bpd(w_i):D(\mathcal B)=D\}\leq 1;\] that is, there are no repeated diagrams occurring among the BPDs of the $w_i$.  We will also see good behavior in this case with respect to Cohen--Macaulayness in Corollary \ref{cor:CM}. 
We conclude this subsection by recording corollaries concerning similarities among the initial ideals arising from lexicographic from southeast term orders, particularly when some lexicographic from southeast initial scheme is known to be reduced.

\begin{corollary}\label{cor:reduced}
Suppose that $w_1, \ldots, w_r \in S_n$ are permutations of the same  length, and set 
$J = \bigcap\{ I_{w_i}:i \in [r]\}$.
Fix a lexicographic from southeast term order $\sigma$ with respect to $\{w_1, \ldots, w_r\}$.
The following two conditions are equivalent:
\begin{enumerate}
    \item $D(\mathcal{B}) = D(\mathcal{B}')$ implies $\mathcal{B} = \mathcal{B}'$ for all 
    $\mathcal{B}, \mathcal{B}' \in \bigcup\{ \bpd(w_i):i \in [r]\}$
    and $\init_\sigma(J)$ has no embedded primes
    \item $\init_\sigma(J)$ is radical.
\end{enumerate}
\end{corollary}
\begin{proof}
A radical ideal cannot have embedded primes.  If $I$ is an ideal without embedded primes, then $I$ is radical if and only if the multiplicity of $\spec(R/I)$ along each irreducible component is $1$.  The result now follows from Theorem \ref{thm:main}.
\end{proof}

\begin{corollary}\label{cor:uniqueDiag}
Suppose that $w_1, \ldots, w_r \in S_n$ are permutations of the same  length, and set 
$J = \bigcap\{ I_{w_i}:i \in [r]\}$.
Suppose that $\sigma$ is a term order satisfying the two equivalent conditions of Corollary \ref{cor:reduced}.  Then $\init_{\sigma}(J) = \init_{\sigma'}(J)$ for every term order $\sigma'$ that is lexicographic from southeast with respect to $\{w_1, \ldots, w_r\}$.
\end{corollary}
\begin{proof}
It follows from Theorem \ref{thm:main} and Corollary \ref{cor:reduced} that \[
\sqrt{\init_{\sigma'}(J)} = \sqrt{\init_{\sigma}(J)} = \init_{\sigma}(J).
\] Hence, $\hilb(R/\sqrt{\init_{\sigma'}(J)};\mathbf t) = \hilb(R/\init_{\sigma}(J);\mathbf t)$ (where the Hilbert functions may be computed with respect to any grading 
for which $J$ is homogeneous, for example the standard grading). Because $\init_{\sigma'}(J)$ and $\init_{\sigma}(J)$ are both initial ideals of $J$, $ \hilb(R/\init_{\sigma'}(J);\mathbf t) = \hilb(R/J;\mathbf t) =\hilb(R/\init_{\sigma}(J);\mathbf t)$.  The equality $\hilb(R/\sqrt{\init_{\sigma'}(J)};\mathbf t) = \hilb(R/\init_{\sigma'}(J);\mathbf t)$ precludes the proper containment $\init_{\sigma'}(J) \subsetneq \sqrt{\init_{\sigma'}(J)}$, and so $\init_{\sigma}(J) = \sqrt{\init_{\sigma'}(J)}=\init_{\sigma'}(J)$.
\end{proof}

\subsection{Applications to Cohen--Macaulayness of unions of matrix Schubert varieties}
\label{subsect:CMApplications}
It is by no means guaranteed that arbitrary equidimensional unions of matrix Schubert varieties will be Cohen--Macaulay (see Section \ref{section:furtherQuestions}).  However, we can use the results in Subsection \ref{subsect:mainProofs} to begin the study of when they are.  Of particular interest are the unions of matrix Schubert varieties that are ASM varieties. 

Given $w_1, \ldots, w_r \in S_n$, $J = \bigcap\{I_{w_i} : i \in [r] \}$, and $(a,b)$ a maximally southeast cell among $\bigcup\{D(w_i): i \in [r]\}$, we will first observe that, whenever $R/\init_{z_{a,b}}(J)$ is reduced, the geometric vertex decomposition of $J$ with respect to $z_{a,b}$ may be computed by taking the geometric vertex decomposition separately at each of the matrix Schubert varieties occurring as irreducible components of $\spec(R/J)$.  When there is a $z_{a,b}$-compatible term order $\sigma$ with respect to which $\spec(R/\init_\sigma(J))$ is reduced, we will use this observation to describe classes of Cohen--Macaulay unions of matrix Schubert varieties. (We include varieties whose products with affine factors are isomorphic to matrix Schubert varieties, as in those of the form $\spec(R/(N_{z_{a,b},I_w}+(z_{a,b})))$, in these unions.)   

In this section, we will be studying the ideal $\init_{z_{a,b}}(J)$ itself, rather than merely components and multiplicities.  For this purpose, we begin with a couple of lemmas.

\begin{lemma}\label{lem:linearJ}
Fix $w_1, \ldots, w_r \in S_n$, and set 
$J=\bigcap \{I_{w_i}:i\in[r]\}$.  Fix a maximally southeast cell $(a,b)$ among elements of 
$ \bigcup\{ D(w_i):i \in [r]\}$.
Then $J$ is linear in $z_{a,b}$.
\end{lemma}
\begin{proof}
We will proceed by induction on $r$.  If $r=1$, then $J$ is linear in $z_{a,b}$ in virtue of the Fulton generators.  For $r>1$, suppose that $z_{a,b}$ is involved in the Fulton generators of the $I_{w_i}$ for $i \in [q]$ for some $q \in [r]$ and that $z_{a,b}$ is not involved in the Fulton generators of $I_{w_i}$ for $i \in [q+1, r]$.  For $i \in [q]$, let $\pi_i$ be the bigrassmannian permutation whose unique essential cell is $(a,b)$ and whose rank condition at that cell is $\rk_{\pi_i}(a,b) = \rk_{w_i}(a,b)$.  Let $v_i = w_it_{a, {w_i}^{-1}(b)}$. Because $D(w_i) = D(v_i)\cup \{(a,b)\}$, we have $I_{w_i} = I_{v_i}+I_{\pi_i}$.  Order the $w_i$ so that $\rk_{w_q}(a,b) =  \max \{ \rk_{w_i}(a,b) :  i \in [q] \}$, in which case $I_{\pi_q} \subseteq I_{w_i}$ for each $i \in [q]$.
Set 
\[K' = \displaystyle \bigcap_{i \in [q-1]} I_{w_i}\] and \[
K =   \displaystyle \bigcap_{i \in [q]} I_{w_i} = K' \cap I_{w_q} = K' \cap (I_{v_q}+I_{\pi_q}).
\]
Because $I_{\pi_q} \subseteq K'$, we may apply the modular law to see \[
K' \cap (I_{v_q}+I_{\pi_q}) = I_{\pi_q} +(K' \cap I_{v_q}).
\]  By induction, $K'$ is linear in $z_{a,b}$.  Let $\sigma$ be any $z_{a,b}$-compatible term order.  Then the reduced Gr\"obner basis for $K'$ with respect to $\sigma$ is also linear in $z_{a,b}$.  Note that $I_{v_q}$ has a generating set that does not involve $z_{a,b}$, and so no product of generators of $K'$ and $I_{v_q}$ has any term divisible by $z_{a,b}^2$, which is to say that $K'I_{v_q}$ is linear in $z_{a,b}$.  
Because
\[
\init_\sigma(K'I_{v_q}) \subseteq \init_\sigma(K' \cap I_{v_q}) \subseteq \sqrt{\init_\sigma(K'I_{v_q})},
\] and $\init_\sigma(K'I_{v_q})$ has no monomial generator divisible by $z_{a,b}^2$, neither can $\init_\sigma(K' \cap I_{v_q})$.  Hence, the reduced Gr\"obner basis for $K' \cap I_{v_q}$ is linear in $z_{a,b}$.  By concatenating the reduced Gr\"obner basis for $K' \cap I_{v_q}$ with the Fulton generators of $I_{\pi_q}$, we obtain a generating set for $K$ that is linear in $z_{a,b}$.  Because 
$L =  \bigcap\{I_{w_i}:i \in [q+1,r]\} $
does not involve $z_{a,b}$, a similar argument to that given above for the intersection of $K'$ and $I_{w_q}$ shows that $J = K \cap L$ is linear in $z_{a,b}$.
\end{proof}

For an intersection $J = \bigcap \{I_{w_i}: i \in [r]\}$ of Schubert determinantal ideals and a maximally southeast cell $(a,b)$ among elements of $\bigcup \{D(w_i):i\in[r]\}$, we record a useful fact about the structure of $N_{z_{a,b},J}$, which can either be proved directly, as done here, or by appealing to standard elimination theory.
\begin{lemma}\label{lem:intersectNs}
Fix $w_1, \ldots, w_r \in S_n$, and set 
$J= \bigcap \{I_{w_i} : i \in [r]\}$.
Fix a maximally southeast cell $(a,b)$ among elements of $\bigcup \{D(w_i):i\in[r]\}$.
Then $N_{z_{a,b},J} = \bigcap\{N_{{z_{a,b}},I_{w_i}}: i \in [r]\}$.
\end{lemma}
\begin{proof} Fix a $z_{a,b}$-compatible term order $\sigma$ and Gr\"obner basis $\mathcal{G}_{J}$ of $J$ with respect to $\sigma$.  From Lemma \ref{lem:linearJ}, we have a geometric vertex decomposition $\init_{z_{a,b}}(J) = C_{z_{a,b},J} \cap (N_{z_{a,b},J}+(z_{a,b}))$. Using \cite[Theorem 2.1(a)]{KMY09}, the elements of $\mathcal{G}_J$ that do not involve $z_{a,b}$ form  a Gr\"obner basis $\mathcal{G}_{N}$ for $N_{z_{a,b},J}$ under $\sigma$.  Each such polynomial is an element of each $I_{w_i}$ that does not involve $z_{a,b}$.  Because $\sigma$ is $z_{a,b}$-compatible, each such element has a remainder of $0$ on division by $N_{z_{a,b},I_{w_i}}$ for all $i \in [r]$.  Thus,
$N_{z_{a,b},J} \subseteq  \bigcap\{ N_{z_{a,b},I_{w_i}}:i \in [r]\}$.
Conversely, because each $N_{z_{a,b},I_{w_i}}$ has a set of generators that does not involve $z_{a,b}$, so too does 
$ \bigcap\{ N_{z_{a,b},I_{w_i}}:i \in [r]\}$.
Each such generator is an element of $J$ that does not involve $z_{a,b}$ and so has a remainder of $0$ on division by $\mathcal{G}_{N}$.  Therefore, 
$ \bigcap\{N_{z_{a,b},I_{w_i}} :i \in [r]\}\subseteq  N_{z_{a,b},J} $ 
as well.
\end{proof}

\begin{notation}\label{not:htN}
Let $w_1, \ldots, w_r \in S_n$ be distinct permutations of the same length, and set 
$J= \bigcap \{I_{w_i}: i \in [r]\}$.
Suppose that $(a,b)$ is a maximally southeast cell among elements of 
$\bigcup \{D(w_i):i\in[r]\}$
and that $z_{a,b}$ is involved in some Fulton generator of $I_{w_i}$ for $i \in [q]$ but not for $i \in [q+1,r]$ for some $q \in [r]$.  
Set $N_{z_{a,b},J}^{\rm ht} =  \bigcap \{N_{{z_{a,b}},w_i}: i \in [q] \}$.  
\end{notation}

\begin{corollary}\label{cor:htNrightht}
 With notation and assumptions as in Notation \ref{not:htN}, $N_{z_{a,b},J}^{\rm ht}$ is the intersection of the minimal primes of $N_{z_{a,b},J}$ of height $\hgt (N_{z_{a,b},J})$. 
\end{corollary}
\begin{proof}
By Lemma \ref{lem:intersectNs}, $N_{z_{a,b},J} = \bigcap\{ N_{z_{a,b},I_{w_i}}:{i \in [r]}\}$.  By Proposition \ref{prop:linkDecomp} and \cite[Proposition 3.3]{Ful92}, each $N_{z_{a,b},I_{w_i}}$ is prime.  By Lemma \ref{lem:NisSchub}, the $N_{z_{a,b},w_i}$ with $i \in [q]$ have height $\ell(w_it_{a,w_i^{-1}(b)}) = \ell(w_i)-1=\hgt (J)-1$.  And $\hgt (J)-1= \hgt (N_{z_{a,b},J})$ by \cite[Lemma 2.8]{KR}. Meanwhile the $N_{z_{a,b},w_i} = I_{w_i}$ with $i \in [q+1,r]$ have height $\hgt (J)$. 
\end{proof}

With notation and assumptions as in Notation \ref{not:htN}, there may exist $i \in [q+1,r]$ so that $I_{w_i}+(z_{a,b})$ is a minimal prime of $N_{z_{a,b},J}+(z_{a,b})$.  However, in this case, $I_{w_i}$ is a minimal prime of $C_{z_{a,b},J}$, and so $I_{w_i}+(z_{a,b})$ is redundant in the primary decomposition of $\init_{z_{a,b}}(J)$.  Thus, 
\[\init_{z_{a,b}}(J) = C_{z_{a,b},J} \cap (N_{z_{a,b},J}+(z_{a,b})) = C_{z_{a,b},J} \cap (N_{z_{a,b},J}^{\rm ht}+(z_{a,b})).\]

\begin{corollary}\label{cor:CM}
Let $w_1, \ldots, w_r \in S_n$ be permutations (not necessarily of the same length), and set $J= \bigcap \{I_{w_i}:i\in[r]\}$.
Fix a maximally southeast cell $(a,b)$ among elements of $\bigcup \{D(w_i):i\in[r]\}$.
If $\init_{z_{a,b}} (J)$ is radical, then \[
\init_{z_{a,b}} (J) = \init_{z_{a,b}}(I_{w_1}) \cap \cdots \cap \init_{z_{a,b}}(I_{w_r}).
\]  Moreover, if $\spec(R/J)$ is Cohen--Macaulay and there exists a $z_{a,b}$-compatible term order $\sigma$ for which $\init_\sigma(J)$ is radical, then $\spec(R/\init_{z_{a,b}}(J))$ is also Cohen--Macaulay.
\end{corollary}
\begin{proof}
Because $\init_{z_{a,b}}(J)$ is radical, the containments \[
\init_{z_{a,b}}(J) \subseteq \init_{z_{a,b}}(I_{w_1}) \cap \cdots \cap \init_{z_{a,b}}(I_{w_r}) \subseteq \sqrt{\init_{z_{a,b}}(J)}
\] imply the equality \[
\init_{z_{a,b}}(J) = \init_{z_{a,b}}(I_{w_1}) \cap \cdots \cap \init_{z_{a,b}}(I_{w_r}).
\]  

Suppose that $\spec(R/J)$ is Cohen--Macaulay and that there is a $z_{a,b}$-compatible term order $\sigma$ for which $\init_\sigma(J)$ is radical.  By \cite[Corollary 2.11(iii)]{CV20}, we know $\spec(R/\init_\sigma(J))$ is Cohen--Macaulay.  Then, because $\spec(R/\init_\sigma(J))$ is a Cohen--Macaulay initial scheme of $\spec(R/\init_{z_{a,b}}(J))$, $\spec(R/\init_{z_{a,b}}(J))$ must be Cohen--Macaulay, as well.  
\end{proof}

With notation and assumptions as in Corollary \ref{cor:CM}, note that Proposition \ref{prop:linkDecomp} allows us to express $\spec(R/\init_{z_{a,b}}(J))$ as a union of matrix Schubert varieties (up to affine factors).
Indeed, working with their defining ideals, we have \begin{align*}
\init_{z_{a,b}}(J) &= \init_{z_{a,b}}(I_{w_1}) \cap \cdots \cap \init_{z_{a,b}}(I_{w_r})\\
&= C_{z_{a,b},I_{w_1}} \cap (N_{z_{a,b},I_{w_1}}+(z_{a,b})) \cap \cdots \cap C_{z_{a,b},I_{w_r}} \cap (N_{z_{a,b},I_{w_r}}+(z_{a,b}))\\
&=\bigcap_{i \in [r]} \left(\left(\bigcap_{u \in \Phi(w_i,z_{a,b})}  I_{u} \right)\cap (N_{z_{a,b},I_{w_i}}+(z_{a,b})) \right).
\end{align*}  

If $\spec(R/\init_{z_{a,b}}(J))$ is not reduced but $\spec(R/N_{z_{a,b},J})$ is known to be Cohen--Macaulay, we will also be able to infer the Cohen--Macaulayness of $\spec(R/\init_{z_{a,b}}(J))$.  

\begin{corollary}\label{cor:nonRadicalCM}
Let $w_1, \ldots, w_r \in S_n$ be permutations of the same  length, and set  
$J= \bigcap \{I_{w_i}:i\in[r]\}$.
Fix a maximally southeast cell $(a,b)$ among elements of 
$\bigcup \{D(w_i):i\in[r]\}$.
If $\spec(R/J)$ and $\spec(R/N_{z_{a,b},J})$ are Cohen--Macaulay, then $\spec(R/\init_{z_{a,b}}(J))$ is Cohen--Macaulay.  If additionally $C_{z_{a,b},J} \neq R$, then $\spec(R/C_{z_{a,b},J})$ is Cohen--Macaulay as well.  In particular, given any $w \in S_n$ and any lower outside corner $(a,b)$ of $D(w)$ with $\rk_w(a,b) \geq 1$,  both $\spec(R/\init_{z_{a,b}}(I_w))$ and $\spec(R/C_{z_{a,b},I_w})$ are Cohen--Macaulay.
\end{corollary}
\begin{proof}
Whenever $\rk_{w_i}(a,b) = 0$ for all $i \in [r]$, then $C_{z_{a,b},J} = R$ and $\init_{z_{a,b}}(J) = N_{z_{a,b},J}+(z_{a,b})$, and so the result is obvious.  If $\rk_{w_i}(a,b) \geq 1$ for some $i \in [r]$, then $C_{z_{a,b},J} \neq R$.  Without loss of generality, suppose $\rk_{w_1}(a,b)\geq 1$.  In that case, the reduced Gr\"obner basis for $J$ with respect to any $z_{a,b}$-compatible term order involves $z_{a,b}$.  Otherwise, we would have 
$J = N_{z_{a,b},J} =  \bigcap\{ N_{z_{a,b},I_{w_i}}:i \in [r]\}$
by Lemma \ref{lem:intersectNs}, and so $\hgt (J) \leq \hgt (N_{z_{a,b},I_{w_1}})$ while $\hgt (J) = \hgt (I_{w_1}) = \hgt (N_{z_{a,b},I_{w_1}})+1$ by Proposition \ref{prop:linkDecomp}.  Therefore, $\sqrt{C_{z_{a,b},J}} \neq \sqrt{N_{z_{a,b},J}}$ by \cite[Proposition 2.4]{KR}.  By Lemma \ref{lem:linearJ}, we know that $\init_{z_{a,b}}(J) = C_{z_{a,b},J} \cap (N_{z_{a,b},J}+(z_{a,b}))$, and so the Cohen--Macaulayness of $\spec(R/\init_{z_{a,b}}(J))$ and $\spec(R/C_{z_{a,b},J})$ follow directly from \cite[Corollary 4.11]{KR}.  

In all cases, the final statement follows from the others together with Proposition \ref{prop:linkDecomp} and \cite[Proposition~3.3(d)]{Ful92}.
\end{proof}

Fix $w \in S_n$ and a lower outside corner $(a,b)$ of $D(w)$. By Proposition \ref{prop:linkDecomp}, 
$C_{z_{a,b},I_w} =  \bigcap\{ I_{u}:u \in \Phi(w,z_{a,b})\}$
defines a union of matrix Schubert varieties that is, in particular, an ASM variety, and so Corollary \ref{cor:nonRadicalCM} gives a source of Cohen--Macaulay ASM varieties.  

We can also use this approach to study ASM varieties that are not equidimensional, and hence fail to be Cohen--Macaulay.  In the study of nonpure simplicial complexes, Stanley introduced \emph{sequentially Cohen--Macaulay} varieties.  For background on the sequentially Cohen--Macaulay property, we refer the reader to \cite[Section III.2]{Sta96}. An equidimensional variety is Cohen--Macaulay if and only if it is sequentially Cohen--Macaulay, just as a pure simplicial complex is shellable in the traditional sense if and only if it is shellable in the nonpure sense of \cite{BW96}.

\begin{corollary}\label{cor:seqCM}
Let $w_1, \ldots, w_r \in S_n$ be permutations (not necessarily of the same  length), and set 
$J= \bigcap \{I_{w_i}:i\in[r]\}$.
Fix a maximally southeast cell $(a,b)$ among elements of $\bigcup \{D(w_i):i\in[r]\}$.
Suppose that no minimal prime of $C_{z_{a,b},J}$ is a minimal prime of $N_{z_{a,b},J}$.  If $\spec(R/J)$ is sequentially Cohen--Macaulay and $\spec(R/N_{z_{a,b},J})$ is Cohen--Macaulay, then $\spec(R/\init_{z_{a,b}}(J))$ is sequentially Cohen--Macaulay.  If $C_{z_{a,b},J} \neq R$, then $\spec(R/C_{z_{a,b},J})$ is sequentially Cohen--Macaulay as well.  
\end{corollary}

\begin{proof}
The case $\rk_{w_i}(a,b) = 0$ for all $i \in [r]$ is the same as in Corollary \ref{cor:nonRadicalCM}. If $\rk_{w_i}(a,b) \geq 1$ for some $i \in [r]$, then $C_{z_{a,b},J} \neq R$.  The condition that no minimal prime of $C_{z_{a,b},J}$ be a minimal prime of $N_{z_{a,b},J}$ implies that $\sqrt{C_{z_{a,b},J}} \neq \sqrt{N_{z_{a,b},J}}$.  Again, Lemma \ref{lem:linearJ} implies that $\init_{z_{a,b}}(J) = C_{z_{a,b},J} \cap (N_{z_{a,b},J}+(z_{a,b}))$.  The result now follows from \cite[Theorem 7.1]{KR} (by a direct application of the forward direction for $C_{z_{a,b},J}$ and an application of the backward direction for $\init_{z_{a,b}}(J)$ with $I = \init_{z_{a,b}}(J)$).
\end{proof}

\section{Consequences for $\beta$-double Grothendieck transition recurrences}
\label{section:additionalConsequences}

Lascoux and Sch\"utzenberger \cite{LS85}  introduced a recurrence on Schubert polynomials called \emph{transition}.  These transition equations imply that each Schubert polynomial expands as a positive sum of monomials.  This fact was subsequently reproved by the introduction of combinatorial formulas for these coefficients (e.g., \cite{BJS93,FS94}), as well as through  geometric arguments (e.g., \cite{KM00,Kog00,BS02,KM05}). Analogous transition equations were also given for double Schubert polynomials  (e.g., \cite{KV97}). Also of interest are transition equations for \notdefterm{Grothendieck polynomials}, which are K-theoretic analogues of Schubert polynomials.  Lascoux \cite{Las01} gave transition equations for single Grothendieck polynomials (see also \cite{Len03}). He also described the double Grothendieck version in \cite{Las}.  All of these formulas follow from specializations of a recurrence on \notdefterm{$\beta$-double Grothendieck polynomials}, which represent classes in connective K-theory \cite{Hud14}. 

An algebraic proof of transition for $\beta$-double Grothendieck polynomials was given in \cite[Appendix~A]{Wei21}.  
The goal of this section is to give a geometrically motivated proof of this recurrence.  
We do so by analyzing the multigraded Hilbert series of matrix Schubert varieties.   The geometric interpretation of transition in Proposition \ref{prop:hilbertAndKPoly} was previously known to Knutson and Yong in unpublished work (see their slides \cite{YSlides}), as was an unpublished proof of Proposition \ref{prop:linkDecomp} that relied on Lascoux's transition formula.

Fix $w\in S_n$, and take a lower outside corner $(a,b)$ of $D(w)$.  As before, let $v=wt_{a,w^{-1}(b)}$.  Recall \[
 \phi(w,z_{a,b})=\{i\in [a-1]:vt_{i,a}>v \text{ and } \ell(vt_{i,a})=\ell(v)+1\}.
 \] Given $U=\{i_1,\ldots,i_k\}$ with $1\leq i_1<i_2<\cdots <i_k<a$, let \[w_U= v (a\, i_k \,i_{k-1}\, \ldots\, i_1)\] (where the second permutation is written in cycle notation).  Note that if $U=\emptyset$, then $w_U=v$.

\begin{lemma}
\label{lemma:reallylattice}
Given $w\in S_n$, let $(a,b)$ be a lower outside corner of $D(w)$.
If $U\subseteq \phi(w,z_{a,b})$ and $U\neq \emptyset$, then $w_U=\vee\{vt_{i,a}:i\in U\}.$
\end{lemma}
\begin{proof}
Take $U\subseteq \phi(w,z_{a,b})$ so that $U\neq \emptyset$.  
Write $U=\{i_1,\ldots,i_k\}$ where $1\leq i_1<i_2<\cdots < i_k<a$.  Then $b>v(i_1)>v(i_2)>\cdots >v(i_k)$.  We obtain the permutation matrix for $w_U$ from the permutation matrix for $v$ by doing the following:
\begin{enumerate}
\item change the $1$'s in positions $(i_\ell,v(i_\ell))$ to $0$'s for all $\ell \in [k]$,
\item change the $1$ in position $(a,b)$ to a $0$, and
\item place $1$'s in positions $(a,v(i_k)),  (i_k,v(i_{k-1})),\ldots,(i_2,v(i_1)),(i_1,b)$.
\end{enumerate}

\noindent Thus, 
\[\rk_v(c,d)-\rk_{w_U}(c,d)=
\begin{cases}
1& \text{if } (c,d)\in \displaystyle \bigcup_{\ell \in [k]} [i_\ell,a-1]\times [v(i_\ell),b-1], \\
0& \text{otherwise.}
\end{cases}\]
In particular,  given $\ell\in [k]$,
\[\rk_v(c,d)-\rk_{vt_{{i_\ell},a}}(c,d)=
\begin{cases}
1& \text{if } (c,d)\in [i_\ell,a-1]\times[v(i_\ell),b-1], \\
0& \text{otherwise.}
\end{cases}
\]
Thus, \[\rk_v(c,d)-\rk_{w_U}(c,d)=\max\{\rk_v(c,d)-\rk_{v t_{i,a}}(c,d):i\in U\}\] for all $c,d \in [n]$, which implies $\rk_{w_U}(c,d)=\min\{\rk_{v t_{i,a}}(c,d):i\in U\}$ for all $c,d\in [n]$.  Therefore, $w_U=\vee\{vt_{i,a}:i\in U\}$.
\end{proof}

\begin{example}
We continue with $w = 4721653$ and $(a,b) = (5,5)$ as in Example \ref{example:transitionDroop}. Recall that $v = 4721563$ and $\phi(w,z_{55}) = \{1,3,4\}$. Consider $U = \{1,3\} \subseteq \phi(w,z_{55})$, which corresponds to $\{u_1=5721463=vt_{1,5},u_2=4751263=vt_{3,5} \} \subseteq \Phi(w,z_{55})$.  Then $w_U = v(5\,3\,1) = 5741263 = \vee \{u_1,u_2\}$.  The Rothe BPDs for $v$ and $w_U$ are presented below with the region $[1,4] \times [4,4] \cup [3,4] \times [2,4]$, in which $\rk_v$ and $\rk_{w_U}$ differ, highlighted on both BPDs in orange.

\[\begin{tikzpicture}[x=1.5em,y=1.5em]
\node at (3.5, 7.5) {$v=4721563$};
\fill[fill=orange](3,3)rectangle(4,7);
\fill[fill=orange](1,3)rectangle(3,5);
\draw[step=1,gray, thin] (0,0) grid (7,7);
\draw[color=black, thick](0,0)rectangle(7,7);
\draw[thick,rounded corners] (.5,0)--(.5,3.5)--(7,3.5);
\draw[thick,rounded corners] (1.5,0)--(1.5,4.5)--(7,4.5);
\draw[thick,rounded corners] (2.5,0)--(2.5,.5)--(7,.5);
\draw[thick,rounded corners] (3.5,0)--(3.5,6.5)--(7,6.5);
\draw[thick,rounded corners] (4.5,0)--(4.5,2.5)--(7,2.5);
\draw[thick,rounded corners] (5.5,0)--(5.5,1.5)--(7,1.5);
\draw[thick,rounded corners] (6.5,0)--(6.5,5.5)--(7,5.5);
\end{tikzpicture}
\hspace{3cm}
\begin{tikzpicture}[x=1.5em,y=1.5em]
\node at (3.5, 7.5) {$w_U = 5741263$};
\fill[fill=orange](3,3)rectangle(4,7);
\fill[fill=orange](1,3)rectangle(3,5);
\draw[step=1,gray, thin] (0,0) grid (7,7);
\draw[color=black, thick](0,0)rectangle(7,7);
\draw[thick,rounded corners] (.5,0)--(.5,3.5)--(7,3.5);
\draw[thick,rounded corners] (1.5,0)--(1.5,2.5)--(7,2.5);
\draw[thick,rounded corners] (2.5,0)--(2.5,.5)--(7,.5);
\draw[thick,rounded corners] (3.5,0)--(3.5,4.5)--(7,4.5);
\draw[thick,rounded corners] (4.5,0)--(4.5,6.5)--(7,6.5);
\draw[thick,rounded corners] (5.5,0)--(5.5,1.5)--(7,1.5);
\draw[thick,rounded corners] (6.5,0)--(6.5,5.5)--(7,5.5);
\end{tikzpicture}
\]
\end{example}

The following two lemmas are standard exercises.
\begin{lemma}
\label{lemma:singleVarHilb}
Fix a $\mathbb Z^d$ grading on $S=\kappa[z_1,\ldots,z_n]$ and an ideal $I$ of $S$ that is homogeneous with respect to the $\mathbb{Z}^d$ grading.  Fix a variable $z_j$.  If $I$ has a generating set that does not involve $z_{j}$, then 
\[\hilb(S/(I+(z_{j}));\mathbf t)=(1-\mathbf t^{\deg(z_{j})})\cdot \hilb(S/I;\mathbf t).\] 
\end{lemma}

\begin{lemma} \label{lemma:hilbertIntersection}
Fix a $\mathbb Z^d$ grading on $S=\kappa[z_1,\ldots,z_n]$.  Suppose that $I_1,\ldots,I_r$ are monomial ideals of $S$.  If $J=\bigcap\{ I_{i}:i \in [r]\}$,
then
\[\hilb(S/J;\mathbf t)=\sum_{\emptyset \neq U\subseteq [r]} (-1)^{\#U-1}\hilb\left (S/\sum_{i\in U}I_i;\mathbf t\right ). \]
\end{lemma}

We now show that multigraded Hilbert series and K-polynomials of Schubert determinantal ideals satisfy a transition recurrence.  Our proof follows from Proposition~\ref{prop:linkDecomp}, which gave a recurrence on unions of matrix Schubert varieties.
\begin{proposition}
\label{prop:hilbertAndKPoly}
Fix $w\in S_n$, and let $(a,b)$ be a lower outside corner of $D(w)$.  Let $\deg(z_{i,j}) = e_i+e_{n+j}$, in which case $(\mathbf{x,y})^{\deg(z_{i,j})} = x_iy_j$.  Then the following hold:
\begin{thmlist}
    \item\label{proppart:hilbRecursion} \[\hilb(R/I_w;\mathbf x,\mathbf y)=\left(1-x_ay_b\right)\hilb(R/I_v;\mathbf x,\mathbf y)+x_ay_b\sum_{\emptyset \neq U\subseteq \phi(w,z_{a,b})}(-1)^{\#U-1}\hilb(R/I_{w_U};\mathbf x,\mathbf y).\]
    \item \label{proppart:KpolyRecursion} \[\mathcal K(R/I_w;\mathbf x,\mathbf y)=\left(1-x_ay_b\right)\mathcal K(R/I_v;\mathbf x,\mathbf y)+x_ay_b\sum_{\emptyset \neq U\subseteq \phi(w,z_{a,b})}(-1)^{\#U-1}\mathcal K(R/I_{w_U};\mathbf x,\mathbf y).\]
\end{thmlist}
\end{proposition}
\begin{proof}

\noindent (i) First suppose $\rk_w(a,b)=0$.  In this case, $z_{a,b}\in I_w$ and $I_w=I_v+(z_{a,b})$.  Furthermore, $\phi(w,z_{a,b})=\emptyset$.  The natural generators of $I_v$ do not involve $z_{a,b}$.  Thus, we apply Lemma~\ref{lemma:singleVarHilb} and see
\[\hilb(R/I_w;\mathbf x,\mathbf y)=\left(1-x_ay_b\right) \cdot \hilb(R/I_v;\mathbf x,\mathbf y),\] as desired.

Now assume $\rk_w(a,b)\geq 1$.  Applying Proposition~\ref{prop:linkDecomp}, $\init_{z_{a,b}}(I_w)= C \cap (N+(z_{a,b}))$, where 
$C=\bigcap\{I_{v t_{a,i}}:i\in \phi (w,z_{a,b})\}$
and $N=I_v$.

Using the short exact sequence \[
0 \rightarrow R/\init_{z_{a,b}}(I_w) \rightarrow R/C \oplus R/(N+(z_{a,b})) \rightarrow R/(C+N+(z_{a,b})) \rightarrow 0,
\] we have 
\begin{align*}
    \hilb(R/I_w;\mathbf x,\mathbf y)&= \hilb(R/\init_{z_{a,b}}(I_w);\mathbf x,\mathbf y)\\
    &=\hilb(R/C;\mathbf x,\mathbf y)+\hilb(R/(N+(z_{a,b}));\mathbf x,\mathbf y)\\
    &\quad -\hilb(R/(C+N+(z_{a,b}));\mathbf x,\mathbf y).
\end{align*}
Since $v\leq u$ for all $u\in \phi(w,z_{a,b})$, we have $N\subseteq C$ and so $C+N+(z_{a,b})=C+(z_{a,b})$.  As $C$ has a generating set that does not involve $z_{a,b}$, we apply Lemma~\ref{lemma:singleVarHilb} to conclude \[\hilb(R/(C+N+(z_{a,b}));\mathbf x,\mathbf y)=\hilb(R/(C+(z_{a,b}));\mathbf x,\mathbf y)=\left(1-x_ay_b\right)\hilb(R/C;\mathbf x,\mathbf y).\]
Thus, 
\[\hilb(R/I_w;\mathbf x,\mathbf y)=\left(1-x_ay_b\right)\hilb(R/N;\mathbf x,\mathbf y)+x_ay_b\hilb(R/C;\mathbf x,\mathbf y).\]

Let $<$ be an anti-diagonal term order.  By Proposition \ref{prop:linkDecomp} and Lemma \ref{lemma:asmIdealFacts} \ref{lempart:intersectInit}, $\init_<(C) = \bigcap \{\init_<(I_{vt_{a,i}}) : i \in \phi(w,z_{a,b})\}$.  Hence by Lemma~\ref{lemma:hilbertIntersection}, 
\begin{align*}\hilb(R/C;\mathbf x,\mathbf y)&=\hilb(R/\init_<(C);\mathbf x,\mathbf y) \\
&= \sum_{ \emptyset \neq U\subseteq \phi(w,z_{a,b})}(-1)^{\#U-1} \hilb(R/\sum_{i\in U} \init_<(I_{vt_{a,i}});\mathbf x,\mathbf y)\\
&= \sum_{ \emptyset \neq U\subseteq \phi(w,z_{a,b})}(-1)^{\#U-1} \hilb(R/\sum_{i\in U} I_{vt_{a,i}};\mathbf x,\mathbf y).
\end{align*}
If $U\subseteq \phi(w,z_{a,b})$ and $U\neq \emptyset$, Lemma~\ref{lemma:reallylattice} yields
\[\hilb(R/\sum_{i\in U} I_{vt_{a,i}};\mathbf x,\mathbf y)=\hilb(R/w_U;\mathbf x,\mathbf y),\]
from which the result follows. 

\noindent (ii) This is an immediate consequence of \ref{proppart:hilbRecursion}.
\end{proof}

We may recover transition equations for (double) Schubert polynomials and (double) Grothendieck polynomials from Proposition~\ref{prop:hilbertAndKPoly}.  We start by working with the $\beta$-double Grothendieck polynomials of \cite{FK94}, from which the other formulas follow immediately. 

Write $\mathbb Z[\beta][\mathbf x,\mathbf y]=\mathbb Z[\beta][x_1,\ldots,x_n,y_1,\ldots,y_n]$.  There is an action of the symmetric group $S_n$ on $\mathbb Z[\beta][\mathbf x,\mathbf y]$ defined by $w\cdot f=f(x_{w(1)},\ldots,x_{w(n)},y_1,\ldots, y_n)$.  Given $f\in \mathbb Z[\beta][\mathbf x,\mathbf y]$, define \[\pi_i(f)=\frac{(1+\beta x_{i+1})f-(1+\beta x_i)s_i\cdot f}{x_i-x_{i+1}}.\]
 Let $x_i\oplus y_j=x_i+y_j+\beta x_iy_j$.   We define the \mydef{$\beta$- double Grothendieck polynomial} $\mathfrak G^{(\beta)}_w(\mathbf x,\mathbf y)$ as follows:  If $w_0=n \, n-1 \, \ldots \, 1$, then 
 \[\mathfrak G^{(\beta)}_{w_0}(\mathbf x,\mathbf y)= \prod_{i+j\leq n} (x_i\oplus y_j).\]
 Otherwise, given $w\in S_n$ with $w(i)>w(i+1)$, we define $\mathfrak G^{(\beta)}_{ws_i}(\mathbf x,\mathbf y)=\pi_i(\mathfrak G^{(\beta)}_w(\mathbf x,\mathbf y))$.  Because the operators $\pi_i$ satisfy the same braid and commutation relations as the simple reflections $s_i$, the polynomial $\mathfrak G^{(\beta)}_w(\mathbf x,\mathbf y)$ is well defined.  Setting $\beta=0$ and replacing each $y_i$ with $-y_i$ recovers $\mathfrak S_w(\mathbf x,\mathbf y)$.
 
 Setting $\beta=-1$, we obtain $\mathfrak G^{(-1)}_w(\mathbf x,\mathbf y)$ by taking $\mathcal K(R/I_w;\mathbf x,\mathbf y)$ and substituting $x_i\mapsto (1-x_i)$ and $y_i\mapsto (1-y_i)$ for all $i\in [n]$ (see \cite[Theorem~2.1]{Buc02} and \cite[Theorem~A]{KM05}). From Proposition~\ref{prop:hilbertAndKPoly}, we recover \cite[Theorem~2.3]{Wei21}.

\begin{corollary}
\label{cor:transitiongrothendieck}
Fix $w\in S_n$, and let $(a,b)$ be a lower outside corner of $D(w)$.  Then \[\mathfrak G^{(\beta)}_w(\mathbf x,\mathbf y)=(x_a\oplus y_b) \mathfrak G_v^{(\beta)}(\mathbf x,\mathbf y)+(1+\beta(x_a\oplus y_b))\sum_{ \emptyset \neq U\subseteq \phi(w,z_{a,b})}\beta^{\#U-1}\mathfrak G_{w_U}^{(\beta)}(\mathbf x,\mathbf y).\]
\end{corollary}

\begin{proof}
  Making the substitutions $x_i\mapsto (1-x_i)$ and $y_i\mapsto (1-y_i)$ into  Proposition~\ref{prop:hilbertAndKPoly}\ref{proppart:KpolyRecursion} yields
\begin{align*}
    \mathfrak G^{(-1)}_w(\mathbf x,\mathbf y)&=(x_a+ y_b-x_ay_b) \mathfrak G_v^{(-1)}(\mathbf x,\mathbf y)\\&+(1-(x_a+ y_b-x_ay_b))\sum_{ \emptyset \neq U\subseteq \phi(w,z_{a,b})}(-1)^{\#U-1}\mathfrak G_{w_U}^{(-1)}(\mathbf x,\mathbf y).
\end{align*}
Thus, the $\beta=-1$ case follows immediately from Proposition~\ref{prop:hilbertAndKPoly}\ref{proppart:KpolyRecursion}.

We now derive the general equation from the $\beta=-1$ case. As a shorthand, write \[\mathfrak G^{(-1)}_u(-\beta\mathbf x,-\beta \mathbf y)=\mathfrak G^{(-1)}_u(-\beta x_1,\ldots, -\beta x_n,-\beta y_1,\ldots,-\beta y_n).\] Substituting $x_i\mapsto -\beta x_i$ and $y_i\mapsto -\beta y_i$ for all $i$  yields
\begin{align*}
 \mathfrak G^{(-1)}_w(-\beta\mathbf x,-\beta\mathbf y)&=-\beta(x_a\oplus y_b)\mathfrak G_v^{(-1)}(-\beta\mathbf x,-\beta\mathbf y)\\
    &\quad +(1+\beta(x_a\oplus y_b))\sum_{ \emptyset \neq U\subseteq \phi(w,z_{a,b})}(-1)^{\#U-1}\mathfrak G_{w_U}^{(-1)}(-\beta\mathbf x,-\beta\mathbf y).
\end{align*}
We have $\mathfrak G^{(-1)}_u(-\beta\mathbf x,-\beta \mathbf y)=(-\beta)^{\ell(u)}\mathfrak G^{(\beta)}_u(\mathbf x,\mathbf y)$ for all $u\in S_n$. Thus, 
\begin{align*}
    (-\beta)^{\ell(w)}\mathfrak G^{(\beta)}_w(\mathbf x,\mathbf y)&=-\beta(x_a\oplus y_b) (-\beta)^{\ell(v)}\mathfrak G_v^{(\beta)}(\mathbf x,\mathbf y)\\
    &\quad +(1+\beta(x_a\oplus y_b))\sum_{ \emptyset \neq U\subseteq \phi(w,z_{a,b})}(-1)^{\#U-1}(-\beta)^{\ell(w_U)}\mathfrak G_{w_U}^{(\beta)}(\mathbf x,\mathbf y).
\end{align*}
 If $U\subseteq \phi(w,z_{a,b})$, then $\ell(w_U)=\ell(w)+\#U-1$. Furthermore, $\ell(v)+1=\ell(w)$.  Thus, 
\begin{align*}
    (-\beta)^{\ell(w)}\mathfrak G^{(\beta)}_w(\mathbf x,\mathbf y)&=(-\beta)^{\ell(w)}(x_a\oplus y_b) \mathfrak G_v^{(\beta)}(\mathbf x,\mathbf y)\\
    &\quad +(-\beta)^{\ell(w)}(1+\beta(x_a\oplus y_b))\sum_{ \emptyset \neq U\subseteq \phi(w,z_{a,b})}\beta^{\#U-1}\mathfrak G_{w_U}^{(\beta)}(\mathbf x,\mathbf y),
\end{align*}
from which the claim follows.
\end{proof}

By making appropriate specializations of the equation in Corollary~\ref{cor:transitiongrothendieck}, we may recover transition equations for (double) Schubert and (double) Grothendieck polynomials. 
In particular, we recover Theorem~\ref{thm:transition}.

\begin{proof}[Proof of Theorem~\ref{thm:transition}]
As a consequence of Corollary~\ref{cor:transitiongrothendieck}, by setting $\beta=0$ we obtain 
\[\mathfrak G^{(0)}_w(\mathbf x,\mathbf y)=(x_a+y_b) \mathfrak G_v^{(0)}(\mathbf x,\mathbf y)+\sum_{  i\in \phi(w,z_{a,b})}\mathfrak G_{w_{\{i\}}}^{(0)}(\mathbf x,\mathbf y).\]
We have $w_{\{i\}}=vt_{i,a}$ for all $i\in \phi(w,z_{a,b}).$  Furthermore,
 $\mathfrak S_u(\mathbf x,\mathbf y)=\mathfrak G^{(0)}_u(\mathbf x,-\mathbf y)$.  Thus, the result follows immediately by substituting $y_i\mapsto -y_i$ for all $i\in [n]$.
\end{proof}

\section{Further inquiry}\label{section:furtherQuestions}

Computations in this section were assisted by Macaulay2 \cite{M2} and Sage \cite{sagemath}.  All computations were performed over the field $\kappa = \QQ$.  As usual, let $R=\kappa[z_{1,1}, \ldots, z_{n,n}]$.  

Closely related to the study of Gr\"obner degenerations is, of course, the study of Gr\"obner bases themselves.  The following problem remains open:

\begin{problem}\label{prob:Grob}
Describe the generating sets of the Schubert determinantal ideal $I_w$ that can occur as a Gr\"obner basis under some diagonal term order $\sigma$.
\end{problem}
For $w \in S_n$ and a diagonal term order $\sigma$, even when $\init_\sigma(I_w)$ is radical, explicit Gr\"obner bases are only known in a special case governed by pattern avoidance \cite{Kle}. One of the two patterns in $S_5$ that does not fall under that result is $21543$.  
When $\sigma$ is a diagonal, lexicographic from southeast term order with respect to $\{21543\}$, already there are elements of the reduced Gr\"obner basis of higher degree than any Fulton generator of $I_{21543}$.  By Corollary \ref{cor:uniqueDiag}, there is a unique lexicographic from southeast initial ideal. The ideal $\init_\sigma(I_{21543})$ has a reduced Gr\"obner basis containing nine elements, including one of degree five, though $I_{21543}$ has only eight Fulton generators, one of degree one, and seven of degree three.

The situation looks especially complicated for $w \in S_n$ for which there is some diagram $D \subseteq [n] \times [n]$ so that $\#\{\mathcal B\in \bpd(w):D(\mathcal B)=D\}>1$, in which case the authors expect there will always  exist distinct reduced Gr\"obner bases arising from different diagonal term orders. 
 For example,  consider $214365 \in S_6$, and  take $\sigma$ to be the lexicographic order on the variables ordered starting from $z_{n,n}$, progressing up column $n$ then up column $n-1$ and so on.  Similarly, let $\sigma'$ be the lexicographic order on the variables ordered starting from $z_{n,n}$, progressing left along row $n$, then row $n-1$ and so on.
Then 
\[
\init_{\sigma} (I_{214365}) = (z_{11} ,z_{12}z_{21}z_{33}, z_{12}z_{21}z_{34}z_{43}z_{55}, z_{12}z_{23}z_{31}z_{34}z_{43}z_{55}, z_{13}z_{21}^2z_{32}z_{34}z_{43}z_{55})
\] 
while 
\[
\init_{\sigma'} (I_{214365}) = (z_{11} ,z_{12}z_{21}z_{33}, z_{12}z_{21}z_{34}z_{43}z_{55}, {z_{13}z_{21}z_{32}}z_{34}z_{43}z_{55}, z_{12}^2z_{23}z_{31}z_{34}z_{43}z_{55}).
\]  
Both $(z_{11}, z_{12}, z_{21}^2)$   and $(z_{11}, z_{12}^2, z_{21} )$ are primary ideals giving multiplicity two at the prime ideal $(z_{11}, z_{12}, z_{21})$ (as predicted by Theorem \ref{thm:main}), but, as we see above, they contribute to distinct initial ideals.  We see from the initial ideals that the reduced Gr\"obner bases under $\sigma$ and under $\sigma'$ both have five elements, one each of degrees one, three, five, six, and seven; however, the Gr\"obner bases themselves are distinct.  The complexity of this example may be surprising 
since $I_{214365}$ has only three Fulton generators, one each of degrees one, three, and five.

Moreover, not only can diagonal initial ideals fail to be radical, they can even have embedded associated primes. For example, with $\sigma$ and $\sigma'$ as above, $\init_\sigma(I_{2143675})$ has $49$ associated primes, $43$ of which are of height four and six of which are of height five, 
and $\init_{\sigma'}(I_{2143675})$ has  $46$ associated primes, $43$ of which are height four and three of which are height five. Because  $\spec(R/\init_\sigma(I_{2143675}))$ and $\spec(R/\init_{\sigma'}(I_{2143675}))$ are equidimensional, all height five associated primes are embedded. The authors do not know if the condition $\#\{\mathcal B\in \bpd(w):D(\mathcal B)=D\}\leq 1$ for all $D \subseteq [n] \times [n]$ precludes embedded primes, nor even if it is possible for some diagonal initial ideal of some $I_w$ to have embedded primes while another does not.  

\begin{problem}
Characterize the permutations $w \in S_n$ for which there is some diagonal term order $\sigma$ so that $\init_\sigma(I_w)$ has embedded primes. Relatedly, characterize the permutations $w \in S_n$ so that $\init_\sigma(I_w)$ has embedded primes for all diagonal term orders $\sigma$.  
\end{problem}

We state below two problems related to Subsection \ref{subsect:CMApplications}.
\begin{problem}\label{prob:ASMCM}
Characterize the Cohen--Macaulay (or sequentially Cohen--Macaulay) ASM varieties.
\end{problem}

\begin{problem}\label{prob:SchubUnionCM}
Characterize the sets of permutations $w_1, \ldots, w_r \in S_n$ with 
$J= \bigcap\{ I_{w_i}:i \in [r]\}$
for which $\spec(R/J)$ is Cohen--Macaulay (or sequentially Cohen--Macaulay).
\end{problem}

In light of Lemma \ref{lemma:asmIdealFacts}, a complete solution to Problem \ref{prob:SchubUnionCM} would imply a complete solution to Problem \ref{prob:ASMCM}.  Equidimensional ASMs that are not Cohen--Macaulay (hence also not sequentially Cohen--Macaulay) appear as soon as $\asm(5)$ and even as a union of just two matrix Schubert varieties: If $w_1 = 34512$ and $w_2 = w_1^{-1} = 45123$, then \[
I_{w_1} \cap I_{w_2} = (z_{11},z_{12},z_{21}, z_{22})+((z_{13},z_{23}) \cap (z_{31},z_{32})),
\] which defines (up to affine factors) a standard first example of a scheme that is equidimensional but not Cohen--Macaulay.  One easily checks the equality $I_{w_1} \cap I_{w_2} = I_A$ for \[
A = \begin{pmatrix}
0 & 0 & 1 & 0 & 0\\
0 & 0 & 0 & 1 & 0\\
1 & 0 & -1 & 0 & 1\\
0 & 1 & 0 & 0 & 0\\
0 & 0 & 1 & 0 & 0
\end{pmatrix}.
\]
An example of an ASM that is not even equidimensional is $A = \begin{pmatrix} 0 & 0 & 1 & 0\\
1 & 0 & -1 & 1 \\
0 & 1 & 0 & 0 \\
0 & 0 & 1 & 0 \end{pmatrix}$, which satisfies $I_A = I_{4123} \cap I_{3412} = (z_{11}, z_{12}) + ((z_{13}) \cap (z_{21}, z_{22})).$  Because $\spec(R/I_A)$ is not equidimensional, it is not Cohen--Macaulay.  However, the Stanley--Reisner complex $\Delta(I_A)$ is shellable in the nonpure sense of \cite[Definition 2.1]{BW96}; therefore, $\spec(R/I_A)$ is sequentially Cohen--Macaulay.  

The questions of which ASMs or, more generally, which unions of matrix Schubert varieties are Cohen--Macaulay (or sequentially Cohen--Macaulay) are wide open and, in the opinion of the authors, quite interesting.  

Finally, we consider a problem concerning integer partitions.

\begin{problem}
Fix $N \in \NN$, and consider the standard grading on $R$ and any diagonal term order $\sigma$.  Characterize the partitions of $N$ that arise as a vector of multiplicities of the minimal primes of $\init_\sigma(J)$ for some 
$J =  \bigcap\{ I_{w_i}:i \in [r]\}$ 
for permutations $w_1, \ldots, w_r \in S_n$ of the same  length and satisfying $e(R/J) = N$.
\end{problem}

If one takes $\sigma$ as in the example below Problem \ref{prob:Grob}, then, using Theorem \ref{thm:main}, one can fix a monomial Schubert determinantal ideal $I_w$ and integer $m \in \ZZ_+$ and construct an intersection $J$ of Schubert determinantal ideals so that the multiplicity of $\spec(R/I_w)$ along $\spec(R/\init_\sigma(J))$ is exactly $m$ (provided $n$ is sufficiently large).  To give an example of this construction, it will be convenient to work in $S_{\infty} = \bigcup_{n \in \ZZ_+} S_n$ with the usual embeddings $S_n \hookrightarrow S_{n+1}$.  

\begin{example} If $I_w = (z_{11}, z_{12}, z_{21})$, then we may begin with $w=321 \in S_3$ and choose the lower outside corner $(a,b) = (2,1)$. We multiply by the $3$-cycle $(234) = (a(b^{-1})(n+1))$ to obtain $w' = 3142 \in S_4$.  We have constructed $w'$ to have a blank tile at $(3,2)$, where pipes $2$ and $w^{-1}(1) = 3$ cross in the Rothe BPD of $w$, so that we have $C_{z_{32}, I_{w'}} = I_w$.  
Similarly, we may multiply $w'$ by the $3$-cycle $(345)$ to obtain $w'' = 31254 \in S_5$, which satisfies $C_{z_{44}, I_{w''}} = I_{w'} \cap I_{41235}$.  It follows from Theorem \ref{thm:main} that, if $J = I_{32145} \cap I_{31425} \cap I_{31245} \cap I_{41235}$, then $\mult_{I_w}(R/\init_\sigma(J)) = 3$.  (At each step, if the newly created lower outside corner is $(a',b')$, then multiplying by the $3$ cycle $(a'(b'^{-1})(n+1))$ creates a new permutation whose unique droop of pipe $b'$ recovers the previous permutation via Lemma \ref{lem:bpdBij}.)  The cost of growing the multiplicity at $I_w$ to $3$ in this way is multiplicity $2$ at $(z_{11},z_{12}, z_{32})$ and multiplicity $1$ at $(z_{11},z_{12}, z_{43})$ as well as multiplicity $1$ at $(z_{11},z_{12},z_{13})$ coming from the unique droop of pipe $3$ in the Rothe BPD of $31254$.  
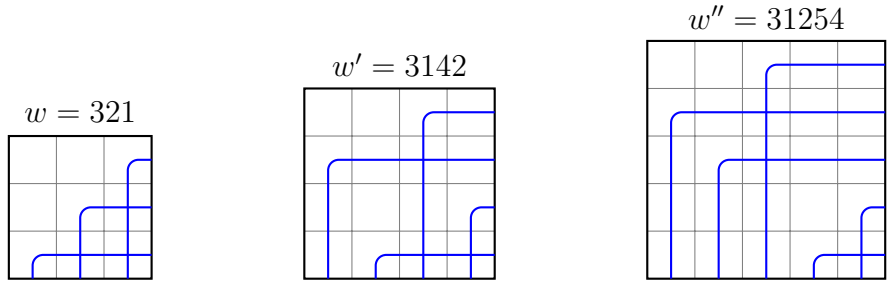
\[
\begin{tikzpicture}[x=1.5em,y=1.5em]
\node at (1.5, 3.5) {$w = 321$};
	\draw[step=1,gray, thin] (0,0) grid (3,3);
	\draw[color=black, thick](0,0)rectangle(3,3);
	\draw[thick,rounded corners,color=blue] (.5,0)--(.5,0.5)--(3,0.5);
	\draw[thick,rounded corners,color=blue] (1.5,0)--(1.5,1.5)--(3,1.5);
	\draw[thick,rounded corners,color=blue] (2.5,0)--(2.5,2.5)--(3,2.5);
	\end{tikzpicture}
	 \hspace{2cm}  \begin{tikzpicture}[x=1.5em,y=1.5em]
\node at (2, 4.5) {$w' = 3142$};
	\draw[step=1,gray, thin] (0,0) grid (4,4);
	\draw[color=black, thick](0,0)rectangle(4,4);
	\draw[thick,rounded corners,color=blue] (.5,0)--(.5,2.5)--(4,2.5);
	\draw[thick,rounded corners,color=blue] (1.5,0)--(1.5,0.5)--(4,0.5);
	\draw[thick,rounded corners,color=blue] (2.5,0)--(2.5,3.5)--(4,3.5);
	\draw[thick,rounded corners,color=blue] (3.5,0)--(3.5,1.5)--(4,1.5);
	\end{tikzpicture}
	 \hspace{2cm} \begin{tikzpicture}[x=1.5em,y=1.5em]
\node at (2.5, 5.5) {$w'' = 31254$};
	\draw[step=1,gray, thin] (0,0) grid (5,5);
	\draw[color=black, thick](0,0)rectangle(5,5);
	\draw[thick,rounded corners,color=blue] (.5,0)--(.5,3.5)--(5,3.5);
	\draw[thick,rounded corners,color=blue] (1.5,0)--(1.5,2.5)--(5,2.5);
	\draw[thick,rounded corners,color=blue] (2.5,0)--(2.5,4.5)--(5,4.5);
	\draw[thick,rounded corners,color=blue] (3.5,0)--(3.5,0.5)--(5,0.5);
	\draw[thick,rounded corners,color=blue] (4.5,0)--(4.5,1.5)--(5,1.5);
	\end{tikzpicture}
	\]
	
In this case, we have constructed the partition $7 = 3+2+1+1$ coming from $7 = e(R/J)$ and $3$, $2$, $1$, and $1$ the multiplicities at the distinct minimal primes of $R/\init_\sigma(J)$.  The intersection $J' = I_{3214} \cap I_{3142} \cap I_{2413} \cap I_{4123} \cap I_{2341}$ gives $7 = e(R/J') = 3+1+1+1+1$. The complete list of possible partitions arising in this way, however, is not immediately obvious. 
\end{example} 

\begin{example}
If $\lambda$ is a partition whose conjugate has distinct parts, it is possible to construct an ideal 
$J= \bigcap \{I_{w_i}:i\in[r]\}$
for which the vector of multiplicities is $\lambda$.  Fix a partition $\mu=(\mu_1,\ldots,\mu_r)$ with distinct parts and let $w_i$  be the simple transposition $s_{\mu_i}$.  Then $\bpd(w_i)$ has exactly $\mu_i$ elements and $\{D(\mathcal B):\mathcal B\in \bpd(w_i)\}=\{\{(i,i)\}:i\in[\mu_i]\}$.  Thus, if $\lambda=(\lambda_1,\ldots, \lambda_{\mu_1})$ is the conjugate of $\mu$ and $\sigma$ is as above,  the multiplicity of $\spec(R/(z_{i,i}))$ along $\spec(R/\init_\sigma(J))$ is exactly $\lambda_i$.
\end{example}

\bibliographystyle{amsalpha} 
\bibliography{refs}

\end{document}